\documentclass[11pt]{amsart}

\usepackage{quiver}
\usepackage[paperheight=11in,paperwidth=8.5in,left=1in,top=1.25in,right=1in,bottom=1.25in,]{geometry}
\usepackage{comment}
\usepackage[super]{nth}
\usepackage[none]{hyphenat}
\usepackage{graphicx, amssymb}
\usepackage{epstopdf}
\usepackage{youngtab}
\usepackage{mathtools}
\usepackage{enumitem}
\usepackage[boxsize=.35 em]{ytableau}

\usepackage{tikz}   

\usetikzlibrary{arrows,calc,intersections,matrix,positioning,through, arrows.meta,decorations.markings}
\usepackage{tikz-cd}
\tikzset{commutative diagrams/.cd,every label/.append style = {font = \normalsize}}

\usetikzlibrary{shapes.geometric}
\DeclareMathOperator{\Spec}{Spec}
\DeclareMathOperator{\Aut}{Aut}

\DeclareMathOperator{\inn}{in}
\DeclareMathOperator{\out}{out}

\DeclareMathOperator{\Gr}{Gr}
\DeclareMathOperator{\GL}{GL}
\DeclareMathOperator{\Mat}{Mat}

\DeclareMathOperator{\sgn}{sign}
\DeclareMathOperator{\alt}{alt}
\DeclareMathOperator{\Sk}{Sk}

\DeclareMathOperator{\area}{area}
\DeclareMathOperator{\Pf}{\mathbf{Pf}}
\newcommand{\Mut}{\operatorname{Mut}}

\newcommand{\AAA}{\mathcal{A}_{n,k,m}}

\newcommand{\gt}[1]{Z_{#1}}
\newcommand{\gto}[1]{Z_{#1}^\circ}

\def\FFcal{\mathcal{F}}
\newcommand{\Grk}{\Gr_{k,n}^{\ge 0}}
\newcommand{\lrangle}[1]{\langle #1 \rangle}
\def\AA{\mathcal{A}}

\newcommand{\Z}{\mathbb{Z}}
\newcommand{\R}{\mathbb{R}}
\newcommand{\C}{\mathbb{C}}
\newcommand{\CC}{\mathbb{C}}
\newcommand{\CCC}{\mathbb{C}}
\newcommand{\A}{\mathcal{A}}

\newcommand{\CGr}[2]{\CC[\widehat{\Gr}_{#1,#2}]}

\def\Xcal{\mathcal{X}}

\def\Q{\mathbb{Q}}

\def\ZZ{\mathbb{Z}}

\newcommand{\xx}{\mathbf{x}}

\usepackage{amsthm}
\usepackage{thmtools}
\usepackage{hyperref}
\usepackage[capitalize, nameinlink, noabbrev, nosort]{cleveref}

\newtheorem{theorem}{Theorem}[section]
\newtheorem{lemma}[theorem]{Lemma}
\newtheorem{proposition}[theorem]{Proposition}
\newtheorem{corollary}[theorem]{Corollary}
\newtheorem{conjecture}[theorem]{Conjecture}
\newtheorem{question}[theorem]{Question}

\theoremstyle{definition}
\newtheorem{definition}[theorem]{Definition}
\newtheorem{example}[theorem]{Example}

\theoremstyle{remark}

\newtheorem{remark}[theorem]{Remark}
\newtheorem{notation}[theorem]{Notation}

\setlist[itemize]{leftmargin=*}
\setlist[enumerate]{leftmargin=*}

\begin{document}
\begin{abstract}
The theory of cluster algebras is closely connected to 
the theory of total positivity; indeed, the desire to better understand 
total positivity was 
one of the main motivations for Fomin and Zelevinsky's introduction
of cluster algebras~\cite{ca1}.  In particular, any cluster variety
whose coordinate ring has a cluster structure has a natural notion of 
\emph{positive part}: the subset of the variety where all 
cluster variables are positive.  In this paper, we explain
that there are other \emph{signed cells} contained in cluster varieties
that are equally natural from a cluster-theoretic point of view.
These come from \emph{signed seeds}, which can be thought of as a 
$\Z/2\Z$-grading on cluster variables, and which were introduced in 
\cite{ELPSTW} in the context of the \emph{amplituhedron}.
More generally, given any abelian group $G$, 
we introduce the notion of a \emph{$G$-graded seed} for a cluster algebra,
which is a way of assigning elements of $G$
to each cluster variable which is compatible with the cluster structure.  
When $G$ is the multiplicative group
$\{-1,1\}$, this recovers the above notion of {signed seed};
when $G=\C^*$, this recovers the notion of \emph{cluster
automorphism group} \cite{GSV}
or \emph{cluster dilation group}
~\cite{neville2026mysteriouspointskeystrees};
and when $G=\Z^d$, this 
recovers the notion of \emph{graded cluster algebra} 
	studied by Grabowski-Launois \cite{Launois}, 
	Grabowski \cite{Grabowski} and 
Gekhtman-Shapiro-Vainshtein \cite[Section 5.2]{GSV} 
	(which had previously 
	appeared in special cases in work of Fomin-Zelevinsky \cite{ca4}).
	The examples we study include the space of square matrices,
	symmetric matrices, skew-symmetric matrices, positroid varieties,  and 
	amplituhedron tiles.  We also connect this notion to 
	tropical mutation when $G=\R$ or $\Z$.
\end{abstract}
\title{Signed seeds and $G$-gradings on cluster algebras}
\date{\today}
\author{Lauren K. Williams}
\address{Department of Mathematics, Harvard University, Cambridge, MA, United States}
\email{williams@math.harvard.edu}
\author{Alan Yan}
\address{Department of Mathematics, Harvard University, Cambridge, MA, United States}
\email{alanyan@math.harvard.edu}
\address{}
\maketitle
\setcounter{tocdepth}{1}
\tableofcontents

\section{Introduction}
	The theory of \emph{total positivity} began in the early 1900's as the study of square matrices with all minors positive; it has important connections to probability and stochastic processes. It was revitalized in the 1990's when Lusztig \cite{Lusztig3} realized that total positivity can be generalized in the setting of Lie theory, and that in particular, the totally positive matrices can be viewed as the ``totally positive'' part of the general linear group. Inspired by Lusztig's work, Fomin and Zelevinsky subsequently introduced \emph{cluster algebras}~\cite{ca1}, which provide a unifying algebraic and combinatorial framework for a variety of phenomena that occur in total positivity, quiver representations, Teichm\"uller theory, and Lie theory. The relationship between total positivity and cluster algebras is as follows. Many geometric objects have a natural notion of ``positive part'', and separately, the coordinate rings of many geometric objects have a cluster algebra structure.When both phenomena occur, one can usually describe the ``positive part'' of the geometric object as the locus where all cluster variables take positive values.  For example, the totally positive matrices $\Mat_{n\times n}^{>0}$ are the $n \times n$ matrices where the minors (the determinants of all square submatrices) are all positive. Meanwhile, one can view the set of invertible $n\times n$ matrices $\GL_n$ as a variety whose coordinate ring is generated by the minors.  This coordinate ring can actually be viewed as a cluster algebra, and the minors are all cluster variables, so this provides a very simple example of the phenomenon described above. However, the phenomenon occurs in much broader contexts, including \emph{Teichmuller theory}, and the study of various algebraic varieties such as Grassmannians and flag varieties.

	The above phenomenon is all about \emph{positivity}: all the minors and cluster variables are positive. However, in the last few years we have started to observe that other interesting sign patterns come up in relation to cluster algebras. For example, in recent work of the first author with Even-Zohar, Lakrec, Parisi, Sherman-Bennett, and Tessler~\cite{ELPSTW}, we observed that each \emph{tile} of the \emph{amplituhedron} (a geometric object which is related to scattering amplitudes in physics) can be described as the region where certain cluster variables are positive and others are negative. These sign patterns are not arbitrary: they obey a special algebraic condition (that we called \emph{signed seed}) that ensures that the signs of all cluster variables are well-defined when one expresses the tile in a different coordinate system.

	More generally, given any abelian group $G$, one can define the notion of a \emph{$G$-graded seed} for a cluster algebra, which is a way of assigning elements of $G$ to each cluster variable which is compatible with the cluster structure. When $G$ is the multiplicative group $\{-1,1\}$, this recovers the above notion of {signed seed}. When $G=\C^*$, this recovers the notion of \emph{cluster automorphism group} studied by Gekhtman, Shapiro, and Vainshtein~\cite{GSV}, also called the \emph{cluster dilation group}~\cite{neville2026mysteriouspointskeystrees}. And when $G=\Z^d$, this recovers the notion of \emph{graded cluster algebra} studied by Grabowski~\cite{Grabowski}, Grabowski-Launois~\cite{Launois}, Booker-Price~\cite{Booker}, and Gekhtman-Shapiro-Vainshtein~\cite[Section 5.2]{GSV} (which had previously 
	appeared in special cases in work of Fomin-Zelevinsky~\cite{ca4}).

	In this paper we investigate the notion of $G$-graded seed and signed seed, and prove various basic properties.  For instance, for any affine cluster variety, the \emph{signed cell} defined by a signed seed is homeomorphic to an open ball, and each signed seed provides a finite ``test for membership'' in that signed cell, see~\cref{thm:test}. We also present some useful results that can help enumerate the $G$-gradings of a given quiver, see~\cref{sec:enumeration}.

	We then take a detailed look at a handful of concrete examples. In~\cref{sec:matrices} we recall the cluster structure on the set of square matrices, and we classify the signed seeds and $G$-gradings. We also point out that certain \emph{sign regular} matrices come from signed cells. In~\cref{sec:symm} we prove that the space of symmetric matrices has a cluster structure, and we classify the $G$-gradings. In~\cref{sec:skew} we recall the cluster structure on the space of skew-symmetric matrices (which comes from~\cite{MR2427512}), and we classify the $G$-gradings. Building on work of Boretsky, Calvo Cortes, and El Maazouz~\cite{skew}, we show that the ``totally positive'' part of the space of skew-symmetric matrices (in the sense of Lusztig) comes from a signed seed on the associated cluster structure. We then study signed seeds in Grassmannians and positroid varieties (~\cref{sec:Grassmannian}),and in amplituhedron tiles (~\cref{sec:amplituhedron}). We classify all $G$-graded seeds for Grassmannians and positroid varieties. We recall the connection between amplituhedron tiles and signed seeds, and give a general conjecture (see~\cref{conj:signed}), which slightly strengthens~\cite[Conjecture 7.17 (ii)]{ELPSTW}, and says that any amplituhedron tile for the amplituhedron $\mathcal{A}_{n,k,m}(Z)$ comes from a signed seed in (a freezing of) a cluster for the Grassmannian $\Gr_{k,k+m}$. Finally in~\cref{sec:RZ}, we study $\R$- and $\Z$-graded seeds, connecting them to tropicalized mutation and to $Y$-seed mutation. These $\R$- and $\Z$-graded seeds played an important role in previous work of the first author with Rietsch on Newton-Okounkov bodies and mirror symmetry~\cite{RW, RW2}.

{\bf Acknowledgements.}  We would like to thank Jonathan Boretsky, Amanda Burcroff, Sergey Fomin, Rick Kenyon, Karola Meszaros, David Speyer, Matty Tyler, and Milen Yakimov for useful discussions. LW is supported by the National Science Foundation under Award No. DMS-2152991 and DMS-2552947. Any opinions, findings, and conclusions or recommendations expressed in this material are those of the authors and do not necessarily reflect the views of the National Science Foundation.

We used GPT 5.5 as a tool for doing computations with Pfaffians when we initially wrote the paper. Later we used GPT 5.6 and Claude Opus to generate referee reports on the paper which we used to edit it.

\section{Cluster algebras and cluster varieties}\label{sec:cluster}
We start by quickly reviewing the definition of cluster algebra; for a more leisurely introduction,  see \cite[Chapter 3]{FWZ1}.

\begin{definition}
Let~$s$ and~$r$ be two positive integers such that~$s \geq r$.
An $r \times r$ matrix $B=(b_{ij})$ with integer entries is called
\emph{skew-symmetrizable} if $d_i b_{ij} = -d_j b_{ji}$
for some positive integers $d_1, \dots, d_r$.
In other words, a matrix is skew-symmetrizable if it differs from a
skew-symmetric matrix by a rescaling of its rows by positive scalars.
	An $s \times r$ integer matrix $\tilde{B}$,
 whose top $r\times r$ submatrix is
skew-symmetrizable is called an \emph{extended skew-symmetrizable} matrix.
\end{definition}

\begin{remark}
If $\tilde{B}$ is an $s \times r$ integer matrix whose top $r \times r$ submatrix $B$ is 
skew-symmetric, then we say that $\tilde{B}$ \emph{comes from a quiver}, because $B$
is the adjacency matrix of a quiver $Q$ on $r$ \emph{mutable vertices} $v_1,\dots, v_r$ 
with $b_{ij}$ arrows from 
$v_i$ to $v_j$ for each positive entry $b_{ij}$.  Moreover, we can encode the bottom
$(s-r) \times r$ submatrix of $\tilde{B}$ by adding $(s-r)$ \emph{frozen vertices}
$v_{r+1},\dots, v_s$, where for $r+1 \leq i \leq s$, we have $b_{ij}$ arrows from 
$v_i$ to $v_j$ if $b_{ij}$ is positive, and we have $-b_{ij}$ arrows from
	$v_j$ to $v_i$ if $b_{ij}$ is negative.  The resulting quiver $\tilde{Q}$ is sometimes called an \emph{ice quiver}.
By convention  there are no arrows connecting two frozen vertices.
\end{remark}

\begin{definition}\label{def:matrix-mutation}
Let $\tilde{B} =(b_{ij})$ be an $s \times r$
extended skew-symmetrizable integer matrix.
For $1 \leq k \leq r$, the
\emph{matrix mutation $\mu_k$ in direction~$k$} transforms
$\tilde{B}$ into the $s \times r$ matrix $\mu_k(\tilde{B}) = (b'_{ij})$ whose
entries are given by:
\begin{equation}
\label{eq:matrix-mutation1}
b'_{ij} =
\begin{cases}
-b_{ij} & \text{if $i=k$ or $j=k$;} \\[.05in]
b_{ij}+b_{ik}b_{kj} & \text{if $b_{ik}>0$ and $b_{kj}>0$;}\\[.05in]
b_{ij}-b_{ik}b_{kj} & \text{if $b_{ik}<0$ and $b_{kj}<0$;}\\[.05in]
b_{ij} & \text{otherwise.}
\end{cases}
\end{equation}
\end{definition}

\begin{definition}\label{def:seed}
Let~$s$ and~$r$ be two positive integers such that~$s \geq r$, and 
let $\FFcal$ be a field 
	isomorphic to the field of rational functions over~$\Q$ (alternatively, over $\CC$)
in $s$ independent variables.
A \emph{labeled seed} $\Sigma$ in $\FFcal$ is a pair
$\Sigma:=(\tilde \xx, \tilde B)$~where
\begin{itemize}[leftmargin=.15in]
\item $\tilde \xx = (x_1, \dots, x_s)$ is an $s$-tuple of
elements of~$\FFcal$ forming a
\emph{free generating set}; that is, $x_1, \dots, x_s$
are algebraically independent, and $\FFcal = \Q(x_1, \dots, x_s)$;
\item $\tilde B=(b_{ij})$ is an $s \times r$ extended
  skew-symmetrizable integer matrix.
\end{itemize}
We use the following terminology:
$\tilde \xx$ is the (labeled)
\emph{extended cluster}
of the labeled seed $(\tilde \xx, \tilde B)$;
the $r$-tuple~$\xx = (x_1, \dots, x_r)$ is the (labeled)
\emph{cluster}
of this seed;
the elements $x_1, \dots, x_r$ are its \emph{cluster variables};
the remaining elements $x_{r+1}, \dots, x_s$ of $\tilde \xx$ are the
\emph{frozen variables} (or \emph{coefficient variables});
the matrix~$\tilde B$ is the \emph{extended exchange matrix} of the seed;
its top $r\times r$ submatrix~$B$ is the \emph{exchange matrix}.
\end{definition}

\begin{definition}\label{def:seed-mutation}
Let $(\tilde \xx, \tilde B)$ be a labeled seed, and let 
 $k\in\{1,\dots,r\}$.
The \emph{seed mutation} $\mu_k$ in direction~$k$
transforms
$(\tilde \xx, \tilde B)$ into the new labeled seed
$\mu_k(\tilde \xx, \tilde B) = (\tilde \xx', \tilde B')$
defined as follows:
\begin{itemize}[leftmargin=.15in]
\item
$\tilde B'= \mu_k(\tilde B)$
as in \cref{def:matrix-mutation}.
\item
	the extended cluster $\tilde \xx' = \mu_k(x_1,\dots,x_s):=(x_1',\dots,x_s')$ is
given by $x_j'=x_j$ for~$j\neq k$,
where $x'_k \in \FFcal$ is determined
by the \emph{exchange relation}
\begin{equation}
\label{eq:exch-rel-geom}
x_k \, x'_k = \prod_{b_{ik}>0} x_i^{b_{ik}}
+ \prod_{b_{ik}<0} x_i^{-b_{ik}}.
\end{equation}
\vspace{-10pt}
\end{itemize}
If the indexing set for one of the two monomials above is the empty
set, then by convention we set the corresponding product equal to~$1$.
\end{definition}

\begin{definition}
\label{def:cluster-algebra}
Let $\Sigma = (\tilde \xx, \tilde B)$ be a labeled seed as above, and let 
$\Xcal$ be the set of all cluster variables
that we can obtain by performing any sequence
of mutations starting at $\Sigma$.

We let the \emph{ground ring} be $R = \CCC[x_{r+1}, \dots, x_{s}]$,
the polynomial ring generated by the frozen variables.
(A~common alternative is to take $R = \Q[x_{r+1}^{\pm 1}, \dots, x_{s}^{\pm 1}]$,
the ring of Laurent polynomials in the frozen variables.)
	The \emph{cluster algebra}~$\AA:=\AA(\Sigma)$ (over~$R$)
is the $R$-subalgebra of the ambient field~$\FFcal$
generated by all cluster variables: $\AA= R[\Xcal]$.

\end{definition}

\begin{definition}\label{def:clustervariety}
An \emph{affine cluster variety} is the 
	variety
$V:=\Spec \A$, where $\A$ is a cluster algebra which is finitely generated.
Each cluster $\xx$ for $\A$ can be thought of as coordinates of an open 
	torus ${(\C^*)}^s \subset V$  called a \emph{cluster chart} or \emph{cluster torus}.
\end{definition}
\begin{remark}
The cluster variety agrees with the union of cluster tori up to codimension
$2$.
Fock and Goncharov~\cite{FG} 
gave an alternative definition
of \emph{cluster $A$-variety} (also called \emph{cluster manifold})
as the scheme consisting of 
the algebraic tori corresponding to seeds of $\A$ glued along 
birational maps corresponding to mutations.  This cluster manifold
can be identified with the union of all cluster tori in $\Spec \A$.

\end{remark}

\section{\texorpdfstring{$G$}{G}-balanced exchange matrices and quivers}\label{sec:G}

We now fix an abelian group $(G,+)$ and define the notion of \emph{$G$-balanced exchange matrices} and \emph{$G$-balanced quivers}. This will lead to the related notion of $G$-gradings on cluster variables. If $b$ is a positive integer and $\sigma \in G$, we use the notation $b\sigma$ to mean $\sigma+ \cdots + \sigma$, where the number of terms is $b$. Similarly, if $b$ is a negative integer, $b\sigma$ denotes $(-\sigma) + \cdots + (-\sigma)$,
where the number of terms is $|b|$.

\begin{definition}[$G$-balanced exchange matrix]\label{def:gradedexchange}
Let $(G,+)$ be an abelian group. Choose positive integers $r \leq s$, and let $\Upsilon = (\tilde{B}, (\sigma_1,\dots,\sigma_s))$, where $\tilde{B}=(b_{ij})$ is an $s \times r$ exchange matrix, and $\sigma = (\sigma_1,\dots,\sigma_s) \in G^s$, such that
\begin{equation}\label{gradedseed}
	\sum_{i=1}^s b_{ij} \sigma_i = 0 \hspace{1cm} \text{ for all }1 \leq j \leq r, \qquad \text{ i.e. }\sigma \tilde{B}=0.
\end{equation}
We then call $\Upsilon$ a \emph{$G$-balanced exchange matrix}, or that 
	$\Upsilon$ is \emph{$G$-balanced}.
\end{definition}

If the exchange matrix $\tilde{B}$ in~\cref{def:gradedexchange}
comes from a quiver, then the above definition can be rewritten as follows.
\begin{definition}\label{rem:quiver}
	Let $\Upsilon = (Q,(\sigma_1,\dots,\sigma_s))$, where $Q$ is a quiver on vertices $v_1,\dots,v_s$, and $(\sigma_1,\dots,\sigma_s)\in G^s$. For $v_i$ a mutable vertex of $Q$, let 
	\[
		\inn_{\sigma}^{Q}(v_i) = \sum_{v_j \to v_i} \sigma_j
		\hspace{.5cm} 
		\text{ and }
		\hspace{.5cm} 
		\out_{\sigma}^{Q}(v_i) = \sum_{v_i \to v_j} \sigma_j,
	\]
	where the first sum is over all arrows in $Q$ entering $v_i$,and the second sum is over all arrows in $Q$ leaving $v_i$. We say that $\Upsilon$  is \emph{$G$-balanced} at the mutable vertex $v_i$ if  
	\begin{equation}\label{eq:grading}
		\inn_{\sigma}^Q(v_i) = \out_{\sigma}^Q(v_i).
	\end{equation} 
	We say that $\Upsilon$  is a \emph{$G$-balanced quiver} if it is $G$-balanced at every mutable vertex.
\end{definition}

In the definition above, we will sometimes write $\sigma(v_i)$ to denote $\sigma_i$. We note that in the context of mirror symmetry,\cite{Akhtar:PolygonalQuivers} studied ``decorated balanced quivers''; these quivers are the same as $G$-balanced quivers in the case that $G = \Z^2$.

\begin{remark}\label{rem:balancepropagates} 
	Note that if~\eqref{gradedseed} holds, then we have that 
\begin{equation}\label{newgrading0}
	\left( \sum_{b_{ik}>0} b_{ik} \sigma_i \right)- \sigma_k = 
	\left( \sum_{b_{ik}<0} -b_{ik} \sigma_i \right) - \sigma_k.
\end{equation}
\end{remark}

This leads to the notion of mutation for $G$-balanced exchange matrices.

\begin{definition}\label{def:Gmutation}
	Let $(G,+)$ be an abelian group, and let $\Upsilon = (\tilde{B}, (\sigma_1,\dots,\sigma_s))$ be a $G$-balanced exchange matrix as above. For $k\in [1,r]$, we define the mutation $\mu_k(\Upsilon)$ of $\Upsilon$ by $\mu_k(\tilde{B}, (\sigma_1,\dots,\sigma_s)):= (\mu_k(\tilde{B}), (\sigma'_1,\dots,\sigma'_s))$, where 
	\[
		\sigma'_i :=\begin{cases}  
			\left( \sum_{b_{ik}>0} b_{ik} \sigma_i \right)- \sigma_k & \text{ if }i=k \\
			\sigma_i & \text{ otherwise.}
		\end{cases}
	\]	
\end{definition}

\begin{proposition}\label{prop:propagate}
If $\Upsilon$ is a $G$-balanced exchange matrix, then for 
any mutable direction $k$, we have that 
$\mu_k(\Upsilon)$ is a $G$-balanced exchange matrix.

\end{proposition}

\begin{proof}
	Suppose that~\eqref{gradedseed} holds. We need to check that after we mutate in direction $k$,~\eqref{gradedseed} still holds for $\mu_k(\Upsilon)$. Let $\tilde{B}'=(b'_{ij}):=\mu_k(\tilde{B})$.

	We first verify that~\eqref{gradedseed} holds when $j=k$. By the definition of mutation, the $k$th column of $\tilde{B}'$ is the negative of the $k$th column of $\tilde{B}$, and $b'_{kk} = 0$. Thus
	\[
		\sum_{i=1}^s b'_{ik} \sigma'_i = \sum_{i=1}^s (-b_{ik}) \sigma_i = 0.
	\]

	Now fix $j$ and suppose that $b_{kj}>0$ (the case where $b_{kj}<0$ is analogous). Then we have 
	\begin{align*}
		\sum_{i=1}^s b'_{ij} \sigma'_i &=
		\sum_{i=1, i\neq k}^s b'_{ij} \sigma_i + b'_{kj} \sigma'_k\\
		&= \sum_{i=1, i\neq k}^s b'_{ij} \sigma_i + 
		(-b_{kj})\left[ \left(\sum_{b_{ik}>0} b_{ik} \sigma_i \right) - \sigma_k \right] \\
		&= \sum_{i \text{ where } b_{ik}>0} (b_{ij} + b_{ik} b_{kj}) \sigma_i 
		+ \sum_{i \text{ where } b_{ik}<0} b_{ij} \sigma_i 
		- b_{kj} \left(\sum_{b_{ik}>0} b_{ik} \sigma_i \right)
		+ b_{kj} \sigma_k
		\\
		&= \sum_{b_{ik}>0} b_{ik} b_{kj} \sigma_i - 
		b_{kj} \left( \sum_{b_{ik}>0} b_{ik} \sigma_i \right) = 0,
	\end{align*}
	as desired. Finally, fix $j$ where $b_{kj}=0$. Then $b'_{ij}=b_{ij}$ for all $i$, and in particular $b'_{kj}=b_{kj}=0$, so 
	\[
		\sum_{i=1}^s b'_{ij} \sigma'_i = \sum_{i=1}^s b_{ij} \sigma_i = 0.
	\]
	This completes the proof that $\mu_k(\Upsilon)$ is a graded seed.
\end{proof}

When $G \subseteq \C^*$, $G$-balanced exchange matrices are closely connected to automorphisms of cluster algebras. Given a cluster algebra $\mathcal{A}$, let $\Aut(\mathcal{A})$ denote the group of algebra automorphisms $\phi$ such that, for each cluster variable $x \in \mathcal{X}$, we have $\phi(x) = \zeta(x) x$ for some $\zeta(x) \in \mathbb{C}^*$. This group is called the \textbf{cluster automorphism group} 
in~\cite{GSV, LamSpeyerCohomologyClusterVarieties,MR4673431,MR5111962} or 
the \textbf{cluster dilation group} in~\cite{neville2026mysteriouspointskeystrees}. We have the following proposition coming from~\cite[Lemma 2.3]{GSV03} and~\cite[Prop 5.1]{LamSpeyerCohomologyClusterVarieties} which specifies how to construct cluster automorphisms by specifying scalar factors on a particular seed. The result is essentially~\cref{prop:propagate} for $G = \mathbb{C}^*$. 

\begin{proposition}\label{prop:lamspeyercohomology}
	Let $((x_1,\dots,x_s),\tilde{B})$ be any seed of $\AA$.
	An element of $\Aut(\AA)$ is determined by its action on 
	$(x_1,\dots,x_s)$.  Given $(\zeta_1,\dots,\zeta_s)\in {(\C^*)}^s$, there
	is an automorphism $\phi$ with $\phi(x_i) =\zeta_i x_i$ if and only if 
	$\zeta \in \ker(\tilde{B}^t)$.\footnote{We interpret this kernel multiplicatively.}
\end{proposition}

\section{\texorpdfstring{$G$}{G}-graded seeds and signed seeds for cluster algebras}\label{sec:G2}

We now explain how the notion of $G$-balanced exchange matrix naturally gives
rise to a $G$-grading on cluster variables.

\begin{definition}[$G$-graded seed]\label{def:gradedseed}
Let $\Upsilon = (\tilde{B}, (\sigma_1,\dots,\sigma_s))$ be a $G$-balanced
exchange matrix as in~\cref{def:gradedexchange}, and 
suppose that $\Sigma = (\tilde{B}, (x_1,\dots,x_s))$ is a seed with the same
$s \times r$ exchange
matrix $\tilde{B}$.
We then call the triple
	$(\tilde{B}, (x_1,\dots,x_s), (\sigma_1,\dots,\sigma_s))$ a 
\emph{$G$-graded seed}.
\end{definition}

By~\cref{prop:propagate}, the mutation of a $G$-graded seed is again a $G$-graded seed. Moreover, if we let $\sigma(x_i):=\sigma_i$ be an assignment of elements of $G$ to the initial cluster variables, this extends to a well-defined $G$-grading of the Laurent polynomial ring $\mathbb{Z}[x_1^{\pm}, \ldots, x_s^{\pm}]$ where the Laurent monomials are homogeneous. From~\cref{rem:balancepropagates}, any cluster variable must be homogeneous with respect to the this grading (mutation of a $G$-graded seed produces a new cluster variable which is homogeneous). Thus, we get a grading on the Laurent polynomial ring which restricts to a $G$-grading on all cluster variables. That is, we have the following: 

\begin{lemma}\label{rem:gradingpropagates} 
	Let $(\tilde{B}, (x_1,\dots,x_s), (\sigma_1,\dots,\sigma_s))$ be a {$G$-graded seed}. Let $\sigma(x_i):= \sigma_i$ for $1\leq i \leq s$ be a grading on the initial cluster variables. Since~\eqref{gradedseed} holds, when we mutate in direction $k$, we obtain 
	\[
		x'_k = \frac{\prod_{b_{ik}>0} x_i^{b_{ik}} + \prod_{b_{ik}<0} x_i^{-b_{ik}}}{x_k},
	\]
	and we have a well-defined grading on the new cluster variable $x'_k$, namely
	\begin{equation}\label{newgrading}
		\sigma(x'_k) := 
		\left( \sum_{b_{ik}>0} b_{ik} \sigma_i \right)- \sigma_k = 
		\left( \sum_{b_{ik}<0} -b_{ik} \sigma_i \right) - \sigma_k.
	\end{equation}
\end{lemma}

\begin{definition}[$G$-gradings on cluster variables]\label{def:gradedclustervariable}
	Let $(\tilde{B}, (x_1,\dots,x_s), (\sigma_1,\dots,\sigma_s))$ be a {$G$-graded seed} as in \cref{def:gradedseed}. We refer to  the assignment $\sigma(x_i):=\sigma_i$ as the \emph{$G$-grading} on the initial cluster variables.
	Then~\eqref{newgrading} gives rise to a well-defined grading
	\[
		\sigma: \Xcal \to G
	\]
	on all cluster variables.
\end{definition}

\subsection{$G$-gradings coming from torus actions}\label{sec:torus-action}

Let $(G, +)$ be an abelian group and $\Sigma = (\widetilde{B}, (x_1, \ldots, x_s))$ a seed.  Then, the set of $G$-gradings on $\Sigma$ has a group structure where we define 
\[
	\left (\widetilde{B}, (\sigma_1, \ldots, \sigma_s) \right ) + \left (\widetilde{B}, (\sigma_1', \ldots, \sigma_s') \right ) := \left (\widetilde{B}, (\sigma_1 + \sigma_1', \ldots, \sigma_s + \sigma_s') \right ). 
\]

\begin{definition}
	Let $(G, +)$ be an abelian group, and let $\widetilde{B} = (b_{ij})$ be an $s \times r$ exchange matrix. We denote $\mathfrak{gr}(\mathcal{A}(\widetilde{B}), G)$ as the group of $G$-gradings on $\widetilde{B}$. 
\end{definition}

%By forgetting the extended exchange matrix in a $G$-graded seed, 
We obtain a natural embedding $\mathfrak{gr}(\mathcal{A}(\widetilde{B}), G) \hookrightarrow G^{\oplus s}$ which identifies $\mathfrak{gr}(\mathcal{A}(\widetilde{B}), G)$ with the following:
\begin{align*}
	\mathfrak{gr}(\mathcal{A}(\widetilde{B}), G) & = \ker (\widetilde{B}^t: G^{\oplus s} \to G^{\oplus r}) = \text{Hom}_{\mathbb{Z}}  \left (\text{coker} (\widetilde{B} : \mathbb{Z}^{\oplus r} \to \mathbb{Z}^{\oplus s}), G \right ). 
\end{align*}
The first equality is a direct restatement of~\cref{def:gradedexchange}, and the second equality is the $\ZZ$-module $G$-valued analogue of $\ker(A^t) = \text{coker}{(A)}^*$ for morphisms of vectors spaces over a field. In particular, the functor $\mathfrak{gr}(\mathcal{A}(\widetilde{B}), -) : \mathsf{Ab} \to \mathsf{Ab}$ is  representable. As a consequence, we have a canonical isomorphism $\mathfrak{gr}(\mathcal{A}(\widetilde{B}), G \oplus H) = \mathfrak{gr}(\mathcal{A}(\widetilde{B}), G) \oplus \mathfrak{gr}(\mathcal{A}(\widetilde{B}), H)$ for any abelian groups $G$ and $H$. For \textit{finitely generated abelian groups}, the group structure of $\mathfrak{gr}(\mathcal{A}(\widetilde{B}), G)$ can be recovered from the cases $G = \mathbb{Z}$ and $G = \mathbb{Z} / p^{\ell} \mathbb{Z}$ for $p$ prime and $\ell \geq 1$. 

\begin{remark}
	In forthcoming work, Burcoff et al.~\cite{enumeration_signed_seeds} enumerate the number of $G$-gradings in the case $G = \mathbb{Z} / p \mathbb{Z}$ for $p$ prime. They prove that if $B$ has full rank over $\mathbb{F}_p$, then there are $p^{\text{\# frozen}}$ $\mathbb{Z}/p \mathbb{Z}$-gradings. 
\end{remark}
Cluster algebras which come from geometry often come with a torus $T = {(\CC^*)}^n$ which acts on the cluster variety. This corresponds to a $\ZZ^n$-graded algebra structure of $\mathcal{A}$ which often gives a $\ZZ^n$-grading as a cluster algebra. Denote this $\ZZ^n$-grading by $\sigma_{\text{std}} \in \mathfrak{gr}(\mathcal{A}, \ZZ^n)$. This allows us to construct many other gradings with respect to other groups $G$. By picking $\mathbf{g} = (g_1, \ldots, g_n) \in G^{\oplus n}$, we have a group homomorphism 
\[
	\Phi_{\mathbf{g}} : \ZZ^n \to G, \quad \mathbf{e}_i \mapsto g_i. 
\]
This induces a morphism ${(\Phi_\mathbf{g})}_* : \mathfrak{gr}(\mathcal{A}, \ZZ^n) \to \mathfrak{gr}(\mathcal{A}, G)$. We say that ${(\Phi_\mathbf{g})}_* (\sigma_{\text{std}})$ is a $G$-grading which \textit{comes from the torus action}. This construction gives a group homomorphism  
\begin{equation}
		\Psi_G: G^{\oplus n} \to \mathfrak{gr}(\mathcal{A}, G), \quad \mathbf{g} \mapsto {(\Phi_\mathbf{g})}_* (\sigma_{\text{std}}). 
\end{equation}
The image of $\Psi_G$ is the set of all $G$-gradings coming from the torus action. If two $G$-gradings $\Upsilon_1, \Upsilon_2 \in \mathfrak{gr}(\mathcal{A}, G)$ differ by an element in $\text{im}(\Psi_G)$, we say that $\Upsilon_1$ and $\Upsilon_2$ differ by the torus action. In particular, $\mathfrak{gr}(\mathcal{A}, G) / \text{im}(\Psi_G)$ is the group of $G$-gradings up to the torus action. The following proposition follows readily from the above discussion. 

\begin{proposition}
	Let $\mathcal{A}$ be a cluster algebra with a torus action giving a $\ZZ^n$-grading $\sigma_{\text{std}} \in \mathfrak{gr}(\mathcal{A}, \ZZ^n)$. For any abelian groups $G$ and $H$, the following diagram commutes 
\[
	\begin{tikzcd}
		{{(G \oplus H)}^{\oplus n}} & {G^{\oplus n} \oplus H^{\oplus n}} \\
		{\mathfrak{gr}(\mathcal{A}, G \oplus H)} & {\mathfrak{gr}(\mathcal{A}, G) \oplus \mathfrak{gr}(\mathcal{A}, H)}
		\arrow["\cong", from=1-1, to=1-2]
		\arrow["{\Psi_{G \oplus H}}", from=1-1, to=2-1]
		\arrow["{\Psi_G \oplus \Psi_H}", from=1-2, to=2-2]
		\arrow["{\cong }", from=2-1, to=2-2]
	\end{tikzcd}
\]
\end{proposition}

Thus, if we want to classify $G$-gradings up to the torus action for $G$ finite abelian, it is enough to classify them for $G = \ZZ / p^{\ell} \ZZ$.  

\subsection{Multi-graded seeds and $g$-vectors}

If the group $G$ equals $\Z^d$ for some integer $d$,
then 
we can rewrite \eqref{gradedseed}
as 
\begin{equation}\label{eq:multi}
(\sigma_1,\dots,\sigma_s) \begin{pmatrix}
		                    b_{11} & \dots & b_{1r} \\
				    \vdots & \vdots & \vdots \\
				    b_{s1} & \dots & b_{sr}
\end{pmatrix} = 0
\end{equation}
or simply $\sigma \tilde{B}=0$. 
In this case the notion of graded seed recovers the notion of \emph{multi-graded seed}
from \cite[Definition 3.1]{Grabowski}.

$\Z^d$-gradings on cluster variables had appeared earlier, notably in the 
study of \emph{$g$-vectors} \cite{ca4}, as we now explain.  The following result
appears (using slightly different terminology) as \cite[Proposition 6.1]{ca4}.

\begin{proposition}\cite[Definition 3.1 and Proposition 6.1]{ca4}\label{prop:g}
Let $B$ be an $n \times n$ exchange matrix,
and let $\tilde{B}$ be a $2n \times n$ exchange matrix,
whose principal (i.e., top $n \times n$) part is $B$
and whose complementary (i.e., bottom) part is the $n \times n$ identity matrix.
Let $G = \Z^n$, and let $\mathbf{e}_1,\dots, \mathbf{e}_n$ denote the standard
basis vectors of $\Z^n$.
	Let $\Upsilon:=(\tilde{B}, (\sigma_1,\dots,\sigma_{2n}))$, where
$\sigma_i:=\mathbf{e}_i$ for $1 \leq i \leq n$,
and $\sigma_{j+n} = -\mathbf{b}_j$ for $1 \leq j \leq n$,
where $\mathbf{b}_j=\sum_i b_{ij} \mathbf{e}_i$ is the $j$th column of $B$.
Then $\Upsilon$ is a $G$-balanced exchange matrix.
\end{proposition}
\begin{proof}
We just need to check that~\eqref{eq:multi} holds.  In other words, we need to show
that for each $1\leq j \leq n$, we have
$\sigma_1 b_{1j} + \cdots + \sigma_n b_{nj} + \sigma_{n+j} = 0$.
But the left hand side equals 
$b_{1j} \mathbf{e}_1 + \cdots + b_{nj} \mathbf{e}_n - \mathbf{b}_j$,
which equals $0$.
\end{proof}

\cref{prop:g} implies that all cluster variables in 
the \emph{principal cluster algebra}
$$\mathcal{A}_{\bullet}(B) \subset 
\Z[x_1^{\pm}, \dots, x_n^{\pm}; y_1,\dots, y_n]$$ 
associated to $\tilde{B}$ are homogeneous with respect to the grading
\[
	\sigma(x_i) = \mathbf{e}_i \text{  and }\sigma(y_i) = -\mathbf{b}_i \text{ for }1\leq i \leq n.
\]
The (multi-) degrees of the cluster variables obtained using this grading
are called $g$-vectors.

\subsection{Signed seeds}
If $G$ is the abelian group with two elements, we prefer to write the group multiplicatively. In this case \cref{def:gradedseed}  recovers the notion of \emph{signed seed}, see \cref{def:signedseed}, which was defined for cluster algebras from quivers in \cite[Definition 9.22]{ELPSTW}.

\begin{definition}[Signed seed] 
	\label{def:signedseed}
	Let $G=\{-1,1\}$ be the abelian group on two elements. Choose positive integers $r \leq s$, and let $\Sigma = (\tilde{B}, (x_1,\dots,x_s))$ be a seed, where $\tilde{B}=(b_{ij})$ is an $s \times r$ exchange matrix, and $(x_1,\dots,x_s)$ is an initial cluster. Let $\tilde{\Sigma}: = (\Sigma, (\sigma_1,\dots,\sigma_s))$, where $\sigma = (\sigma_1,\dots,\sigma_s) \in G^s$, and 
	\begin{equation}\label{signedseed}
		\prod_{i=1}^s (\sigma_i)^{b_{ij}} = 1 
		\hspace{1cm} \text{ for all }1 \leq j \leq r.
	\end{equation}
	We then call $\tilde{\Sigma}$ a \emph{signed seed}.
\end{definition}

It can be useful to rewrite the characterization of signed
seeds in the case of a quiver.

\begin{remark}\label{rem:quiver2}
Consider a seed $\Sigma=(Q,(x_1,\dots,x_s))$ together with 
a tuple $(\sigma_1,\dots,\sigma_s) \in \{-1,1\}^s$.
If $v_i$ is a mutable vertex of $Q$,
let $$\inn_{\sigma}^{Q}(v_i) = \prod_{v_j \to v_i} \sigma_j
	\hspace{.5cm} 
	\text{ and }
	\hspace{.5cm} 
\out_{\sigma}^{Q}(v_i) = \prod_{v_i \to v_j} \sigma_j.$$
	Then $({\Sigma},(\sigma_1,\dots,\sigma_s))$  is a {signed seed} if 
for each mutable vertex $v_i$ we have 
\begin{equation}\label{eq:grading2}
	\inn_{\sigma}^Q(v_i) = \out_{\sigma}^Q(v_i).
\end{equation} 
\end{remark}

\begin{remark}\label{rem:positiveseed}
Suppose that $\sigma_i=1$ for all $1\leq i \leq s$.
Then for any seed $\Sigma$, the pair $(\Sigma, (\sigma_1,\dots, \sigma_s))$
is a signed seed, which gives rise to the constant grading 
$\sigma: \Xcal \to G$ in which $\sigma(x)=1$ for all $x\in \Xcal$.
We call this the \emph{totally positive signed seed}.
\end{remark}

\begin{definition}
Suppose that $V$ is an affine cluster variety as in \cref{def:clustervariety}. The coordinate ring $\C[V]$ can then be identified with a cluster algebra $\A = \A(\Sigma)$. Let $\Xcal$ be the set of cluster variables of $\A$. Given $v\in V$ and $x\in \Xcal$, we say that the \emph{sign of $x$ on $v$ is $1$} and write $\sgn(x(v))=1$ if $x(v)$ is real and strictly positive. Similarly, we say that the \emph{sign of $x$ on $v$ is $-1$} and write $\sgn(x(v))=-1$ if $x(v)$ is real and strictly negative.

Given a signed seed $\tilde{\Sigma}$ for $\C[V]$
with grading $\sigma: \Xcal \to \{-1,1\}$, we define the \textit{(open) signed cell} and \textit{(closed) signed cell} to be the following subsets of real points: 
\begin{align}
	V_{>0}^{\sigma} & :=\{v\in V(\mathbb{R}) \ \vert \ \sigma(x)x(v) > 0 \text{ for all }x\in \Xcal\}, \\
	V_{\geq 0}^{\sigma} & := \{v\in V(\mathbb{R}) \ \vert \ \sigma(x) x(v) \geq 0 \text{ for all }x\in \Xcal\}.
\end{align} 
\end{definition}
The fact that we refer to $V^{\sigma}_{>0}$ as a signed cell will be justified
by \cref{thm:test}, which says that $V^{\sigma}_{>0}$ is in fact a cell, i.e.
homeomorphic to an open ball.

\begin{remark}
If $\sigma$ is identically $1$, as in \cref{rem:positiveseed}, then
\begin{equation}
V_{>0}^{\sigma}=\{v\in V \ \vert \ x(v) >0  \text{ for all } x\in \Xcal.\}
\end{equation}
We may denote this signed cell by $V_{>0}$; it is often called the 
\emph{totally positive} part of $V$ \cite{sf-icm}.

If for example $V$ is the set of $n \times n$ matrices and $\C[V]$ is 
endowed with the usual
cluster structure (see \cref{sec:matrices}) then the cluster variables
include all minors of an $n \times n$ matrix, and 
$V_{>0}$ recovers the set of \emph{totally positive matrices}.
If $V$ is the (affine cone over the) Grassmannian 
$\Gr_{k,n}$ and $\C[\Gr_{k,n}]$ is endowed with the usual
cluster structure (see \cref{sec:Grassmannian}) then 
the cluster variables include all Pl\"ucker coordinates, and 
$V_{>0}$ recovers the \emph{totally positive Grassmannian}, that is,
the subset of the Grassmannian where all Pl\"ucker coordinates are positive.
\end{remark}

\begin{theorem}\label{thm:test}
Let $V$ be an affine cluster variety  whose coordinate ring
can be identified with a cluster algebra $\A$ with cluster variables $\Xcal$. Let $\sigma$ be a grading coming from a signed seed $\tilde{\Sigma}_0$. Then the signed cell $V^{\sigma}_{>0}$ is homeomorphic to an open ball of dimension $s$, where $s$ is the size of any extended cluster.

Moreover, for any signed seed $\tilde{\Sigma}=(\tilde{B}, (x_1,\dots,x_s), 
(\sigma_1,\dots,\sigma_s))$ 
which is mutation equivalent to $\tilde{\Sigma}_0$, we have 
\begin{align}
	V_{>0}^{\sigma}&=\{v\in V(\mathbb{R}) \ \vert \ \sigma(x) x(v) > 0 \text{ for all }x\in \Xcal\} \label{first}\\
	&=\{v\in V(\mathbb{R}) \ \vert \ \sigma(x_i) x_i(v) > 0 \text{ for }1 \leq i \leq s.\} \label{second} 
\end{align}	
In other words, each signed seed provides a finite ``test for membership'' for whether
	a given element $v$ lies in the signed cell $V^{\sigma}_{>0}$.
\end{theorem}

Before proving the theorem, we prove the following lemma.
\begin{lemma}\label{lem:ball}
If $\Sigma=(\tilde{B},(x_1,\dots,x_s))$ is a seed for the cluster variety $V$ and if $(\sigma_1,\dots,\sigma_s)$ is an arbitrary element of $\{\pm {1}\}^{s}$, then
\[
	V_{>0}^{(\sigma_1,\dots,\sigma_s)} :=\{v\in V(\mathbb{R}) \ \vert \ \sigma_i x_i(v) > 0 \text{ for }1 \leq i \leq s\}
\]
is homeomorphic to $(\R_{>0})^s$. 
\end{lemma}
\begin{proof}
By the Laurent phenomenon for cluster algebras, if 
the values of the cluster variables
$x_i(v)$ for $1\leq i \leq s$ are fixed and 
nonzero, this uniquely
determines 
the value of $x(v)$ for any $x\in \Xcal$.
Thus 
	$V^{(\sigma_1,\dots,\sigma_s)}_{>0}$
	is homeomorphic to 
$(\R_{>0})^{s}$ via the map 
$$v \mapsto (\pm x_1(v), \dots, \pm x_s(v)),$$
where the sign in front of $x_i(v)$ is given by $\sigma_i$.
\end{proof}

\begin{proof}[Proof of \cref{thm:test}]
We start by proving \eqref{second}.  (We have that \eqref{first} holds
by definition.)
Clearly 
\[
	\{v\in V(\mathbb{R}) \ \vert \ \sigma(x) x(v) > 0 \text{ for all }x\in \Xcal\}  \subseteq \{v\in V \ \vert \ \sigma(x_i)x_i(v) > 0 \text{ for }1 \leq i \leq s.\}.
\]
To show the other inclusion, consider some $v$ such that 
$\sgn(x_i(v))=\sigma(x_i) \text{ for }1 \leq i \leq s.$
It is enough to show that if we mutate one step away from $\tilde{\Sigma}$,
obtaining a new cluster variable $x'_k$, we have 
$\sgn(x'_k(v)) = \sigma(x'_k)$.

By 
\eqref{eq:exch-rel-geom}
	we have 
\begin{equation*}
	x'_k(v) = 
	\frac{\prod_{b_{ik}>0} x_i(v)^{b_{ik}}
	+ \prod_{b_{ik}<0} x_i(v)^{-b_{ik}}}{x_k(v)},
\end{equation*}
and if we write \eqref{newgrading} multiplicatively, we get 
\begin{equation*}
\sigma(x'_k) := 
	\frac{\prod_{b_{ik}>0} \sigma(x_i)^{b_{ik}}}{\sigma(x_k)} = 
	\frac{\prod_{b_{ik}<0} \sigma(x_i)^{|b_{ik}|}}{\sigma(x_k)}.
\end{equation*} 
In particular the sign of the two monomials
	$\prod_{b_{ik}>0} x_i(v)^{b_{ik}}$ and 
	$\prod_{b_{ik}<0} x_i(v)^{-b_{ik}}$ is equal, and 
	$$\sgn(x'_k(v))  = 
	\frac{\prod_{b_{ik}>0} \sgn(x_i(v))^{b_{ik}}}{\sgn(x_k(v))} = 
	\frac{\prod_{b_{ik}>0} \sigma_i^{b_{ik}}}{\sigma_k}=\sigma(x'_k)$$ as desired.

Now using \eqref{second} and \cref{lem:ball},
	it follows that $V_{>0}^{\sigma}$ is homeomorphic to $(\R_{>0})^s$.
\iffalse
an open ball of dimension $s$. 
By the Laurent phenomenon for cluster algebras, if 
the values of the cluster variables
$x_i(v)$ from a given cluster $(x_1,\dots,x_s)$ are fixed and 
nonzero, this uniquely
determines 
the value of $x(v)$ for any $x\in \Xcal$.
Now using \eqref{second}, we see that 
$V^{\sigma}$ is homeomorphic to 
$(\R_{>0})^{s}$ via the map 
$$v \mapsto (\pm x_1(v), \dots, \pm x_s(v)),$$
where the sign in front of $x_i(v)$ is given by $\sigma(x_i)$. The result follows.
\fi
\end{proof}

Whenever open signed cells are non-empty,~\cref{thm:test} implies that they are always homeomorphic to open cells. Moreover,  all non-empty closed signed cells are homeomorphic to each other.

\begin{theorem}
	If $\Sigma = (\tilde{B}, (x_1, \ldots, x_s), (\sigma_1, \ldots, \sigma_s))$ is a signed seed for the cluster variety $V$, then $V_{\geq 0} \cong V_{\geq 0}^{\sigma}$. 
\end{theorem}

\begin{proof}
	From \cref{prop:lamspeyercohomology}, the signs $(\sigma_1, \ldots, \sigma_s) \in (\CC^*)^s$ determines an automorphism of $\mathcal{A}$ where the induced automorphism on $V$ maps $V_{\geq 0}$ homeomorphically onto $V_{\geq 0}^\sigma$. 
\end{proof}

\subsection{Enumeration of \texorpdfstring{$G$}{G}-gradings}\label{sec:enumeration}
In this section we present two lemmas that can help us
to enumerate the $G$-gradings. Given a  quiver $Q$, a \emph{frozen leaf} is a frozen vertex with degree $1$. An \emph{isolated vertex} is a vertex with degree $0$. When the underlying quiver $Q$ has {frozen leaves} or {isolated vertices}, then we can reduce the enumeration of  $G$-gradings on $Q$ to the enumeration of $G$-gradings of a smaller quiver.

\begin{lemma}\label{lem:frozen-leaf}
	Let $Q$ be a quiver and $x \in Q$ a frozen leaf. Let $Q'$ be the quiver where we remove $x$, freeze the neighbor of $x$, and then remove all edges between frozen vertices. Then, the number of $G$-gradings of $Q$ equals the number of $G$-gradings of $Q'$.  
\end{lemma}

\begin{proof}
	The procedure for removing a frozen leaf is shown in Figure~\ref{fig:frozenleaf}. Let $y$ be the neighbor of $x$; it is necessarily mutable. Given a $G$-grading of $Q'$ that we want to extend to $Q$, 
there is exactly one element of $G$ we can assign to $x$ so that the 
resulting $G$-grading is $G$-balanced  at $y$. 
Conversely, a $G$-grading on $Q$ gives a $G$-grading on $Q'$. 
These correspondences are inverse, so we have a one-to-one correspondence between the $G$-gradings on $Q$ and the $G$-gradings on $Q'$. 
\end{proof}
\begin{figure}[h]
\begin{center}
	\begin{tikzpicture}[
		>=stealth,
		every node/.style={font=\small}
	]
		
		\node[draw, ellipse, minimum width=1.6cm, minimum height=3.2cm] (Q1)
		{$Q\setminus\{x,y\}$};

		\node[circle, draw, right=1.2cm of Q1] (y1) {$y$};
		\node[draw, rectangle, minimum size=6mm, right=0.9cm of y1] (x1) {$x$};

		\draw (y1.west) -- (Q1.60);
		\draw (y1.west) -- (Q1.0);
		\draw (y1.west) -- (Q1.-60);

		\draw (x1.west) -- (y1.east);

		\draw[->, thick] ($(x1.east)+(0.5,0)$) -- ++(1.8,0);

		\node[draw, ellipse, minimum width=1.6cm, minimum height=3.2cm,
			right=3.0cm of x1] (Q2) {$Q\setminus\{x,y\}$};

		\node[draw, rectangle, minimum size=6mm, right=1.2cm of Q2] (y2) {$y$};

		\draw (y2.west) -- (Q2.60);
		\draw (y2.west) -- (Q2.0);
		\draw (y2.west) -- (Q2.-60);
	\end{tikzpicture}

\end{center}
\caption{Removal of frozen leaf}\label{fig:frozenleaf}
\end{figure}
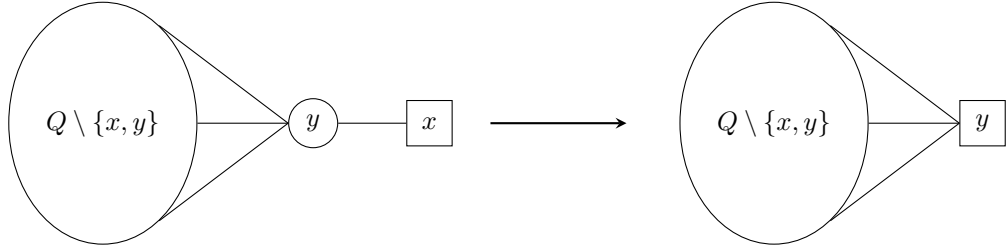

\begin{lemma}\label{lem:isolated-vertex}
	Let $Q$ be a quiver and $x \in Q$ an isolated vertex. Then,
	\[
		\# \{ \text{$G$-gradings of $Q$} \} = |G| \cdot \# \{ \text{$G$-gradings of $Q \backslash \{x\}$} \}. 
	\]
\end{lemma}

\begin{proof}
Given a $G$-grading on $Q \backslash \{x\}$, 
we can assign any element of $G$ to $x$ and obtain a $G$-grading on $Q$,
since $x$ is disconnected from the rest of the quiver.
\end{proof}

\begin{remark}
	There is a generalization of these two results to $G$-balanced exchange matrices. 
	The analogue of an isolated vertex is an index $h$ such that 
	$b_{hi}=0$ for all $1\leq i \leq r$.
	The analogue of a frozen leaf is a frozen index $f$ such that $b_{fi} = 0$ for all $1 \leq i \leq r$ except for exactly one $1 \leq j \leq r$ for which $b_{fj} = \pm 1$. 
	After removing $f$ and freezing $j$, a $G$-grading of the simplified exchange matrix can be extended to a $G$-grading of the original exchange matrix by the formula 
	\[
		\sigma_f = - \frac{1}{b_{fj}} \sum_{i \neq f} b_{ij} \sigma_i. 
	\]
\end{remark}

\section{\texorpdfstring{$G$}{G}-graded seeds for 
square matrices}\label{sec:matrices}

The coordinate ring 
$\CCC[\Mat_{n \times n}]$ 
of the space of $n \times n$
complex matrices is the polynomial ring $\CCC[z_{ij}]$, where $z=(z_{ij})$
denotes a generic $n \times n$ matrix with entries $z_{ij}$.
In this section we will briefly explain how the coordinate ring
$\CCC[\Mat_{n \times n}]$ has a cluster algebra structure, then 
classify its signed seeds.  It turns out that all signed cells
of the space of $n \times n$ matrices can be obtained from the 
\emph{totally positive matrices}
by rescaling some rows and some columns.

\subsection{Totally positive matrices and 
the cluster structure on $\CCC[\Mat_{n \times n}]$}

Let $z$ denote an $n \times n$ matrix.
If $I$ and $J$ are subsets of $[n]$ with the same cardinality, 
let $z_{I,J}$ denote the 
determinant of the submatrix of $z$ which uses rows $I$ and columns $J$.
We say that $z_{I,J}$ is a \emph{solid minor} 
if both $I$ and $J$ are consecutive 
subsets of integers.

\begin{definition}
We say that a matrix $M$ is \emph{totally positive} if all
of its minors are positive; in other words, every square submatrix
of $M$ has positive determinant.
\end{definition}

The first systematic study of totally positive matrices dates back to the 
1930's with works by Gantmacher and Krein \cite{gantmacher-krein}, 
as well as Schoenberg
\cite{schoenberg}.  One natural question to ask about totally positive
matrices concerns \emph{efficient testing} \cite{fz-intel}:
how many minors (and which ones) does one need to check in order to 
ascertain that a given $n \times n$ matrix $M$ is totally positive?
We will return to this question after reviewing the cluster structure
on the space of $n \times n$ matrices, following
 \cite[Section 6.6]{FWZ6}.

The ring 
$\CCC[\Mat_{n \times n}]$  has a cluster structure
with initial seed given as follows.

\begin{definition}
Let $\Sigma = \Sigma^{(n)}$ be the initial seed
shown in 
\cref{fig:initialquiver}, see also \cite[Figure 6.5]{FWZ6}.
(The figure shows the case where $n=4$ but the 
general pattern should be clear.)  
The quiver has $n^2$ total vertices, of which $2n-1$ are frozen.
The initial cluster variables are 
solid minors
which ``stick'' to the left edge 
or the upper edge of the matrix.
Let $\A = \A(\Sigma)$ denote the cluster algebra associated to $\Sigma$.
\end{definition}

\begin{figure}[ht]
\begin{center}
\setlength{\unitlength}{1.8pt}
\begin{picture}(160,60)(0,0)

\put(80,60){\makebox(0,0){$\boxed{1234,\mathbf{1234}}$}}

\put(40,40){\makebox(0,0){$\boxed{234,\mathbf{123}}$}}

\put(80,40){\makebox(0,0){${123,\mathbf{123}}$}}

\put(120,40){\makebox(0,0){$\boxed{123,\mathbf{234}}$}}

\put(-1,20){\makebox(0,0){$\boxed{34,\mathbf{12}}$}}
\put(40,20){\makebox(0,0){${23,\mathbf{12}}$}}
\put(80,20){\makebox(0,0){${12,\mathbf{12}}$}}
\put(120,20){\makebox(0,0){${12,\mathbf{23}}$}}
\put(160,20){\makebox(0,0){$\boxed{12,\mathbf{34}}$}}

\put(-41,0){\makebox(0,0){$\boxed{4,\mathbf{1}}$}}
\put(-1,0){\makebox(0,0){${3,\mathbf{1}}$}}
\put(40,0){\makebox(0,0){${2,\mathbf{1}}$}}
\put(80,0){\makebox(0,0){${1,\mathbf{1}}$}}
\put(120,0){\makebox(0,0){${1,\mathbf{2}}$}}
\put(160,0){\makebox(0,0){${1,\mathbf{3}}$}}
\put(200,0){\makebox(0,0){$\boxed{1,\mathbf{4}}$}}

\put(52,24){\vector(2,1){22}}
\put(108,24){\vector(-2,1){22}}
\put(52,4){\vector(2,1){22}}
\put(108,4){\vector(-2,1){22}}
\put(12,4){\vector(2,1){22}}
\put(148,4){\vector(-2,1){22}}

\put(70,40){\vector(-1,0){17}}
\put(90,40){\vector(1,0){17}}

\put(30,20){\vector(-1,0){21}}
\put(71,20){\vector(-1,0){22}}
\put(89,20){\vector(1,0){22}}
\put(129,20){\vector(1,0){20}}

\put(-7,0){\vector(-1,0){26}}
\put(33,0){\vector(-1,0){26}}
\put(75,0){\vector(-1,0){27}}
\put(87,0){\vector(1,0){26}}
\put(127,0){\vector(1,0){25}}
\put(167,0){\vector(1,0){25}}

\put(0,16){\vector(0,-1){12}}

\put(40,34){\vector(0,-1){12}}
\put(40,16){\vector(0,-1){12}}

\put(80,54){\vector(0,-1){11}}
\put(80,36){\vector(0,-1){12}}
\put(80,16){\vector(0,-1){12}}

\put(120,34){\vector(0,-1){12}}
\put(120,16){\vector(0,-1){12}}

\put(160,16){\vector(0,-1){12}}

\end{picture}
\end{center}

\caption{The initial seed for $\CCC[\operatorname{Mat}_{4 \times 4}]$,
where we denote the minor $z_{I,J}$ by 
$I,\mathbf{J}$.}
\label{fig:initialquiver}
\end{figure}

Let $M$ be an $n \times n$ matrix (either real or complex).
If all 
the cluster variables in a given cluster of $\A$ are positive on $M$, then 
all other cluster variables must also be positive on $M$.  Since all minors 
of an $n \times n$ matrix appear as cluster variables
 \cite[Section 6.6]{FWZ6}, 
it follows that 
all minors are positive on $M$.
For example, if all the minors from \cref{fig:initialquiver} are positive
on $M$, then $M$ must be totally positive. In summary,
the cluster variables of any given cluster comprise a 
``positivity test'' for membership in the set of totally positive matrices.

Our next goal is to classify the signed seeds of $\A$
and hence the signed cells of $\Mat_{n \times n}$.
Note that the totally positive signed seed, which comes from the constant
grading $\sigma=1$, gives rise to a signed cell which is exactly
the set of totally positive matrices.

\subsection{Signed cells of $\Mat_{n \times n}$}
In this section we will classify all signed
seeds for $\CCC[\Mat_{n \times n}]$.

\begin{theorem}\label{thm:matrices}
There are exactly $2^{2n-1}$ signed cells of the set of 
$n \times n$ matrices, each of which is homeomorphic to an open ball.
	Each signed cell can be obtained from the totally
positive matrices by choosing an arbitrary set of rows, and an arbitrary set of columns, and rescaling these rows and these columns by $-1$.
In other words, if $V^{\sigma}$ is a signed cell
of the space of $n\times n$ matrices, and 
$M=(m_{ij})\in V^{\sigma}$, then its signs are completely
determined by the signs of the entries
$m_{11}, m_{12},\dots, m_{1n}$ and $m_{21}, m_{31}, \dots, m_{n1}$.
	In particular, for $I, J \subset [n]$, with $|I|=|J|=\ell$, we have
	\[
		\sgn(M_{I,J}) = \sgn(m_{11})^{\ell} \prod_{i\in I} \sgn(m_{i1}) \prod_{j\in J} \sgn(m_{1j}).
	\]
The analogous result holds for 
arbitrary $G$-gradings: there are exactly $|G|^{2n-1}$ $G$-gradings,
and all of these come from the torus action.
\end{theorem}

Recall the characterization of signed seeds from \cref{rem:quiver2}.

	\begin{proof}[Proof of \cref{thm:matrices}]
Let us use the seed shown in 
\cref{fig:initialquiver}.  Suppose 
we fix the signs of the cluster variables
$z_{i,1}$ and $z_{1,j}$, that is, the signs of the entries 
in the first row and first column of matrix $z$.
This fixes the signs of the cluster variables in the bottom
row of the quiver in \cref{fig:initialquiver}.
Note that $z_{1,1}$ is incident to only one arrow with a vertex
in row $2$.  Thus \eqref{eq:grading2} uniquely determines the 
sign of $z_{12,12}.$  Once we know this sign, \eqref{eq:grading2}
uniquely determines the signs of $z_{23,12}$ and $z_{12,23}$.
We can then propagate this argument and determine the signs
of all cluster variables in the second row of \cref{fig:initialquiver}.
Now we see that $z_{12,12}$ is incident to only one arrow
with a vertex in row $3$.  Thus we can use \eqref{eq:grading2}
to determine the sign of $z_{123,123}$ and all other
cluster variables in the third row.
Continuing in this fashion, we see that if we fix
arbitrary signs for the cluster variables in the bottom row
of \cref{fig:initialquiver}, then  we can find a unique way 
to determine signs for the remaining cluster variables, such that the 
result is a signed seed.  The same argument works for $G$-gradings.

The fact that each signed cell is homeomorphic to an open 
ball is a consequence of \cref{thm:test}.  
\end{proof}

\subsection{Signed cells of matrices and sign patterns for matrix minors}

We describe some classes of real matrices characterized by explicit matrix minor sign pattern rules, and explain their connection with the signed cells of $\Mat_{n \times n}$ coming from the cluster structure. 

\begin{definition}
	A matrix is called \emph{sign regular} if for each $k$, 
	all its $k\times k$ minors have
	the same sign or are zero.
\end{definition}

It is not the case that all sign regular matrices come from 
signed cells.  However, suppose that $M$ is a matrix 
such that its $k \times k$ minors are positive whenever $k$ is even
and negative whenever $k$ is odd.
Then $M$ is sign regular and it comes from
a signed seed.

\begin{remark}
It is well known that totally positive matrices have all positive
eigenvalues.  Moreover, sign regular matrices have all real eigenvalues~\cite{Ando1987}.  
However, square matrices coming from more general signed seeds need not have real eigenvalues.
For example, the matrix 
\[
\begin{bmatrix}
	7 & -7 & 1\\
	4 & -5 & 2\\
	5 & -8 &6 
\end{bmatrix}
\] comes from a signed seed but does not have real eigenvalues.
\end{remark}

Sign-symmetric matrices and their spectral properties have been studied in~\cite{MR2148896,MR872063}. 

\begin{definition}
	A matrix $A$ is \textit{sign symmetric} if $\det A_{\alpha, \beta} \det A_{\beta, \alpha}\geq 0$ for all $\alpha, \beta$. We say a $G$-grading $\sigma : \mathcal{X} \to G$ is \textit{$G$-symmetric} if $\sigma(z_{\alpha, \beta}) = \sigma (z_{\beta, \alpha})$ for all $\alpha, \beta$. 
\end{definition}

Consider an arbitrary $G$-symmetric grading $\sigma : \mathcal{X} \to G$. Since this comes from the torus action, there are $x_1, \ldots, x_n, y_1, \ldots, y_n \in G$ such that 
\[
	\sigma({z_{I, J}}) = \sum_{i \in I} x_i + \sum_{j \in J} y_j. 
\]
Since $\sigma(z_{i, j}) = \sigma (z_{j, i})$ for all indices $i, j$, we must have $x_i - x_j = y_i - y_j$ in $G$. Thus, the $G$-symmetric gradings are determined by the values of $(x_1, x_1 - x_2, \ldots, x_1 - x_n, y_1)$ and it corresponds to the trivial grading if and only if $x_1 + y_1 = 0$ and $x_1 - x_k = 0$ for $2 \leq k \leq n$. Hence, there are $|G|^{n}$ total $G$-symmetric gradings.

\section{$G$-graded seeds for symmetric matrices}\label{sec:symm}

Let $\mathcal{A}_{\text{sym}} = \CC[\Mat_{n, n}] / I_{\text{sym}}$, where $I_{\text{sym}}$ is the ideal generated by $z_{ij} - z_{ji}$ for all $1 \leq i, j \leq n$, be the coordinate ring of the variety of symmetric matrices. We will construct a cluster structure on $\mathcal{A}_{\text{sym}}$. Let $N = \binom{n+1}{2}$, then consider the functions $X_1, \ldots, X_{N} \in \mathcal{A}_{\text{sym}}$ which are defined by 
\begin{equation}\label{eqn:initial-variables-symmetric}
    \mathcal{X}_{\text{initial}} = \left \{ X_{\binom{k}{2} + j} := z_{[j], [k-j+1,k]} \mid 1 \leq j \leq k \leq n \right \}.
\end{equation}
For general $n$, the initial symmetrizable exchange matrix for this cluster structure satisfies a clear pattern. The easiest way to display the exchange matrix is in the form of a valued quiver. 

\begin{definition}[Valued quiver]
    A \textbf{valued quiver} is a quiver $Q = (Q_0, Q_1, s, t)$ with the data of an ordered pair of positive integers $(v_1(\alpha), v_2(\alpha)) \in \ZZ_{>0}^2$ to every edge $\alpha \in Q_1$ and a decomposition $Q_0 = Q_0^{\text{mut}} \sqcup Q_0^{\text{fro}}$ of the vertices into mutable and frozen vertices which satisfies the following conditions:
    \begin{enumerate}[label = (\arabic*)]
        \item There are no loops;
        \item There is at most one edge between any pair of vertices, and there are no edges between any pair of frozen vertices;
        \item There is a choice of positive integers $d_i > 0$ for all $i \in Q_0$ such that for any edge $\alpha \in Q_1$ we have $d_{s(\alpha)} v_1 = d_{t(\alpha)} v_2$. 
    \end{enumerate} 
\end{definition}

\begin{remark}
    The notion of a valued quiver is equivalent to a symmetrizable exchange matrix. Given a valued quiver, we can define a symmetrizable exchange matrix where 
    \[
        b_{ij} = \begin{cases}
            0 & \text{if there is no edge between $i$ and $j$} \\
            v_1 & \text{if there is an edge $i \xrightarrow{(v_1,v_2)} j$} \\
            -v_2 & \text{if there is an edge $j \xrightarrow{(v_1, v_2)} i$}. 
        \end{cases}
    \]
\end{remark}

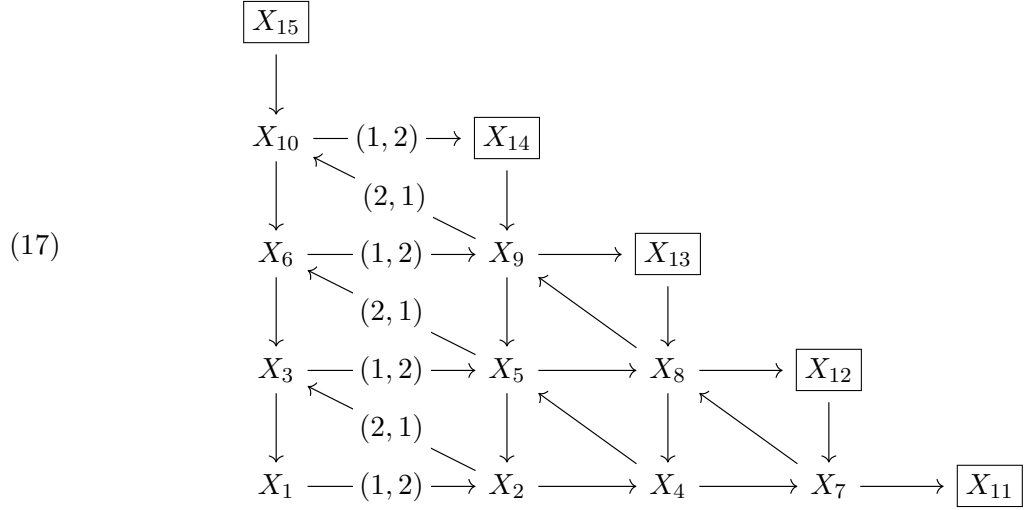
\begin{equation}\label{eqn:symmetric-initial-n=5}
\begin{tikzcd}
    {\boxed{X_{15}}} &&&&& \\
    X_{10} && {\boxed{X_{14}}} \\
    X_6 && X_9 & {\boxed{X_{13}}} \\
    X_3 && X_5 & X_8 & {\boxed{X_{12}}} \\
    X_1 && X_2 & X_4 & X_7 & {\boxed{X_{11}}}
    \arrow[from=1-1, to=2-1]
    \arrow["{(1,2)}"{description}, from=2-1, to=2-3]
    \arrow[from=2-1, to=3-1]
    \arrow[from=2-3, to=3-3]
    \arrow["{(1,2)}"{description}, from=3-1, to=3-3]
    \arrow[from=3-1, to=4-1]
    \arrow["{(2,1)}"{description}, from=3-3, to=2-1]
    \arrow[from=3-3, to=3-4]
    \arrow[from=3-3, to=4-3]
    \arrow[from=3-4, to=4-4]
    \arrow["{(1,2)}"{description}, from=4-1, to=4-3]
    \arrow[from=4-1, to=5-1]
    \arrow["{(2,1)}"{description}, from=4-3, to=3-1]
    \arrow[from=4-3, to=4-4]
    \arrow[from=4-3, to=5-3]
    \arrow[from=4-4, to=3-3]
    \arrow[from=4-4, to=4-5]
    \arrow[from=4-4, to=5-4]
    \arrow[from=4-5, to=5-5]
    \arrow["{(1,2)}"{description}, from=5-1, to=5-3]
    \arrow["{(2,1)}"{description}, from=5-3, to=4-1]
    \arrow[from=5-3, to=5-4]
    \arrow[from=5-4, to=4-3]
    \arrow[from=5-4, to=5-5]
    \arrow[from=5-5, to=4-4]
    \arrow[from=5-5, to=5-6]
\end{tikzcd}\end{equation}

\subsection{Initial valued quiver for $\mathcal{A}_{\text{sym}}$}
For $n = 5$, the initial valued quiver is given by \cref{eqn:symmetric-initial-n=5} with symmetrizers $d_{\binom{k+1}{2}} = 2$ for all $1\leq k \leq n$ and $d_\ell = 1$ otherwise. In \cref{eqn:symmetric-initial-n=5-with-variables} we draw the same initial valued quiver but replacing the variables with the actual values in \cref{eqn:initial-variables-symmetric}.

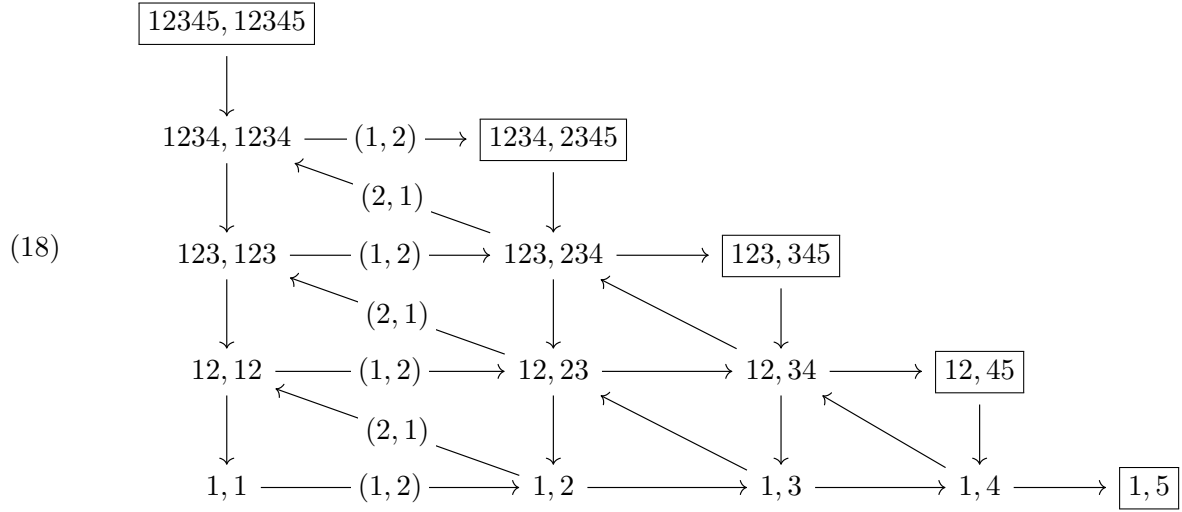
\begin{equation}\label{eqn:symmetric-initial-n=5-with-variables}
\begin{tikzcd}
    {\boxed{12345,12345}} &&&&& \\
    {1234,1234} && {\boxed{1234,2345}} \\
    {123,123} && {123,234} & {\boxed{123,345}} \\
    {12,12} && {12,23} & {12,34} & {\boxed{12,45}} \\
    {1,1} && {1,2} & {1,3} & {1,4} & {\boxed{1,5}}
    \arrow[from=1-1, to=2-1]
    \arrow["{(1,2)}"{description}, from=2-1, to=2-3]
    \arrow[from=2-1, to=3-1]
    \arrow[from=2-3, to=3-3]
    \arrow["{(1,2)}"{description}, from=3-1, to=3-3]
    \arrow[from=3-1, to=4-1]
    \arrow["{(2,1)}"{description}, from=3-3, to=2-1]
    \arrow[from=3-3, to=3-4]
    \arrow[from=3-3, to=4-3]
    \arrow[from=3-4, to=4-4]
    \arrow["{(1,2)}"{description}, from=4-1, to=4-3]
    \arrow[from=4-1, to=5-1]
    \arrow["{(2,1)}"{description}, from=4-3, to=3-1]
    \arrow[from=4-3, to=4-4]
    \arrow[from=4-3, to=5-3]
    \arrow[from=4-4, to=3-3]
    \arrow[from=4-4, to=4-5]
    \arrow[from=4-4, to=5-4]
    \arrow[from=4-5, to=5-5]
    \arrow["{(1,2)}"{description}, from=5-1, to=5-3]
    \arrow["{(2,1)}"{description}, from=5-3, to=4-1]
    \arrow[from=5-3, to=5-4]
    \arrow[from=5-4, to=4-3]
    \arrow[from=5-4, to=5-5]
    \arrow[from=5-5, to=4-4]
    \arrow[from=5-5, to=5-6]
\end{tikzcd}
\end{equation}
The pattern for general $n$ should be clear from the picture. Note that the initial valued quiver is the quiver in \cref{fig:initialquiver} but ``folded'' along the center. In order to account for repeated edges along the center, we need to replace standard edges with valued edges. 

Before proving that this provides a cluster structure to $\mathcal{A}_{\text{sym}}$, we first prove some useful identities. Recall that $\mathbf{P}_I := z_{I, I}$ is the principal minor with respect to $I$ and $P_I := z_{[|I|], I}$ is the flag minor with respect to $I$. By convention, we let $\mathbf{P}_\emptyset = P_\emptyset = 1$. All of the subsequent identities hold in $\CC[\Mat_{n \times n}]$ and thus they hold in $\mathcal{A}_{\text{sym}}$ after identifying $z_{ij} = z_{ji}$. 

\begin{theorem}[Desnanot-Jacobi Identity]
	Let $M$ be a $k \times k$ matrix. Let $M_{I}^J$ denote the matrix with the rows indexed by $I$ removed and the columns indexed by $J$ removed. Then, we have the identity
	\[
		\det (M) \det (M^{1, k}_{1, k}) = \det(M^{1}_1) \det (M^k_{k}) - \det(M^k_1) \det (M^1_k). 
	\]
\end{theorem}
\begin{lemma}\label{lem:principal-minors}
	For $k \geq 2$, we have 
	\[
		\mathbf{P}_{[k-1]} \mathbf{P}_{[2, k+1]} - \mathbf{P}_{[k+1]} \mathbf{P}_{[2, k-1]} = \frac{\mathbf{P}_{[k+1]} z_{[2, k],[k-1]} z_{[k-1], [2,k]} + \mathbf{P}_{[k-1]} z_{[2, k+1], [k]} z_{[k], [2, k+1]} }{\mathbf{P}_{[k]}}. 
	\]
\end{lemma}

\begin{proof}
	From the Desnanot-Jacobi identity, we have 
	\[
		\mathbf{P}_{[k+1]} \mathbf{P}_{[2,k]} = \mathbf{P}_{[2, k+1]} \mathbf{P}_{[k]} - z_{[2, k+1], [k]} z_{[k], [2, k+1]}.
	\]
	Thus, we can replace the right term in the numerator of the right hand side by 
	\[
		\mathbf{P}_{[k-1]} (\mathbf{P}_{[2, k+1]} \mathbf{P}_{[k]} - \mathbf{P}_{[k+1]} \mathbf{P}_{[2, k]})
	\]
	The numerator is then equal to 
	\[
		\mathbf{P}_{[k]} \mathbf{P}_{[k-1]} \mathbf{P}_{[2, k+1]} - \mathbf{P}_{[k+1]} (\mathbf{P}_{[k-1]} \mathbf{P}_{[2, k]} - z_{[2, k], [k-1]} z_{[k-1], [2,k]}).
	\]
	We can apply the Desnanot-Jacobi identity again to get 
	\[
		\mathbf{P}_{[k]} \mathbf{P}_{[2, k-1]} = \mathbf{P}_{[k-1]} \mathbf{P}_{[2,k]} - z_{[2, k],[k-1]} z_{[k-1], [2,k]}. 
	\]
	This completes the proof. 
\end{proof}

\begin{lemma}
	Let $2 \leq a < b$ be positive integers. Let $P_I = z_{[|I|], I}$ denote the flag minor corresponding to $I$. Then we have
	\[
		\frac{P_{[a-1, b-1]} P_{[a, b+1]} P_{[a+1, b]} + P_{[a-1, b]} P_{[a, b-1]} P_{[a+1, b+1]}}{P_{[a, b]}} = 
		P_{[a-1, b-1] \cup \{b+1\}} P_{[a+1, b]}
		- 
		P_{[a-1, b]} P_{[a+1, b-1] \cup \{b+1\}}. 
	\] 
\end{lemma}

\begin{proof}
	This follows from two applications of the three-term incidence-Pl\"ucker relations for flag minors. We use the following two relations:
	\begin{align}
		P_{[a-1, b-1]} P_{[a, b+1]} & = P_{[a-1, b-1] \cup \{b+1\}} P_{[a, b]} - P_{[a-1, b]} P_{[a, b-1] \cup \{b+1\}} \label{eqn:rel1}\\ 
		P_{[a, b]} P_{[a+1, b-1] \cup \{b+1\}} & =  P_{[a+1, b]} P_{[a, b-1] \cup \{b+1\}} - P_{[a, b-1]} P_{[a+1, b+1]} .\label{eqn:rel2}
	\end{align}
	We have: 
	\begin{align*}
		\mathsf{LHS} & = \frac{(P_{[a-1, b-1]} P_{[a, b+1]}) P_{[a+1, b]} + P_{[a-1, b]} P_{[a, b-1]} P_{[a+1, b+1]}}{P_{[a, b]}} \\
		& = \frac{(P_{[a-1, b-1] \cup \{b+1\}} P_{[a, b]} - P_{[a-1, b]} P_{[a, b-1] \cup \{b+1\}})P_{[a+1, b]} + P_{[a-1, b]} P_{[a, b-1]} P_{[a+1, b+1]}}{P_{[a,b]}}\\
		& = P_{[a-1, b-1] \cup \{b+1\}} P_{[a+1, b]} + \frac{P_{[a-1, b]}}{P_{[a, b]}} \left ( P_{[a, b-1]} P_{[a+1, b+1]} - P_{[a+1, b]} P_{[a, b-1] \cup \{b+1\}} \right ) \\
		& = P_{[a-1, b-1] \cup \{b+1\}} P_{[a+1, b]} + \frac{P_{[a-1, b]}}{P_{[a, b]}} \cdot (- P_{[a, b]} P_{[a+1, b-1] \cup \{b+1\}}) \\
		& = P_{[a-1, b-1] \cup \{b+1\}} P_{[a+1, b]} - P_{[a-1, b]} P_{[a+1, b-1] \cup \{b+1\}} = \mathsf{RHS}.
	\end{align*}
	where the second equality follows from~\cref{eqn:rel1} and the fourth equality follows from~\cref{eqn:rel2}. 
\end{proof}

\begin{lemma}\label{lem:symmetric-minors-irred-coprime}
	For any subsets $I, J \subseteq [n]$ of the same size, let $z_{I, J}$ denote the matrix minor in $\mathcal{A}_{\text{sym}}$ corresponding to row set $I$ and column set $J$. Then, the collection of matrix minors $\{z_{I, J}\}$ are irreducible. 
\end{lemma}

\begin{proof}
	We induct on the size $|I| = |J| = k$. For the base case $k = 1$, the matrix minors are matrix entries, which are clearly irreducible. Now, we prove that $z_{I, J}$ if irreducible under the assumption that smaller minors are irreducible and $k \geq 2$. 
	
	We first consider the case $I \neq J$. Pick some $i \in I \backslash J$, then via Laplace expansion we can write 
	\[
		z_{I, J} = \sum_{j \in J} \Delta_j \cdot z_{i, j}, \quad \quad \quad  \Delta_j = \pm z_{I \backslash \{i\}, J \backslash \{j\}}.
	\] 
	Consider an arbitrary factorization $z_{I, J} = P(\mathbf{z}) Q(\mathbf{z})$. When we view both sides as polynomials in $z_{i, j}$ with $j \in J$ over the UFD $R_i := \mathbb{C}[\{z_{a,b} : a \in I \backslash \{i\}, b \in J\}]$, we get an equality of the form 
	\[
		\sum_{j \in J} \Delta_j \cdot z_{i, j} = \left (C + \sum_{j \in J} A_j z_{i, j} \right ) \left ( D + \sum_{j \in J} B_j z_{i, j} \right ), \quad A_j, B_j, C, D \in R_i[z_{i, j}, j \in J].
	\]
	We cannot have that $A_j \neq 0$ for some $j$ and $B_{j'} \neq 0$ for some $j'$ because then the right hand side would be degree $2$ over $z_{i, j}$, $j \in J$. Thus, we can assume without loss of generality that $A_j = 0$ for all $j \in J$. This implies that $D = 0$, and all $\Delta_j$ have a common factor $C$. But from the inductive hypothesis, the $\Delta_j$ are distinct and irreducible. Hence, $C$ must be a unit. This proves that $z_{I, J}$ is irreducible. 

	We now show that determinants of generic symmetric matrices
	are irreducible.\footnote{We expect that this result is well-known
	but couldn't find a published reference.}
	Now consider the case where $I = J$, i.e. $z_{I, J}$ is a principal minor. It suffices to prove that if $M_n = [z_{i,j}]_{1 \leq i, j \leq n}$ is a generic symmetric matrix with $z_{i,j} = z_{j,i}$, then $\det (M_n)$ is an irreducible polynomial. Consider an arbitrary factorization 
	\[
		\det (M_n) = P(\mathbf{z}) Q(\mathbf{z}). 
	\] 
	Consider $P$ and $Q$ as polynomials in $z_{n, n}$. Since $\det(M_n)$ is linear in $z_{n, n}$, we can assume without loss of generality that $P(z)$ is linear in $z_{n, n}$ and $Q(z)$ is independent of $z_{n,n}$. For any $1 \leq k < n$, $Q(z)$ must be independent of $z_{k, n}$ or else $\det(M_n)$ would contain terms like $z_{n, n} z_{k, n}$. Hence $P(z)$ contains terms with $z_{k, n}$ for all $1 \leq k < n$. By similarly reasoning, $Q(z)$ must be independent of $z_{k, k}$. Hence $P(z)$ contains terms with $z_{k, k}$ for all $1 \leq k < n$. Finally, we can repeat the same argument to conclude that $Q(z)$ is independent of all variables $z_{a, b}$, i.e. it is constant. This suffices for the proof. 
\end{proof}
\begin{theorem}
    The cluster algebra given by the initial seed in \cref{eqn:symmetric-initial-n=5-with-variables} is a cluster structure for $\mathcal{A}_{\text{sym}}$. 
\end{theorem}

\begin{proof}
    We first prove $\mathcal{U}(Q) \subseteq \mathcal{A}_{\text{sym}}$ using~\cite[Proposition 6.4.1]{FWZ6}. It is enough to show that whenever $x$ is an initial cluster variable, then $x' \in \mathcal{A}_{\text{sym}}$ and $x'$ is coprime to $x$ and that all initial cluster variables are pairwise coprime. The fact that the initial variables are pairwise coprime already follows from Lemma~\ref{lem:symmetric-minors-irred-coprime}. Note that the one-step exchange relations in~\cref{eqn:symmetric-initial-n=5-with-variables} are exactly the same as the one-step exchange relations as~\cref{fig:initialquiver}. Thus, to compute $x'$ for an initial cluster variable $x \in \mathcal{X}_{\text{sym}}$, we can compute the mutation in the initial quiver for $\text{Mat}_{n \times n}$ and then identify $z_{ij} = z_{ji}$ for all $i, j$. 
	
	For any cluster variable $x = z_{1, k}$ on the first row, we get that $x' = z_{1, k-1} z_{2, k+1} - z_{1, k+1} z_{2,k}$ for $k \geq 2$ which is clearly coprime from $z_{1, k}$ and $x' = z_{1,2}$ for $k = 1$ which is coprime to $z_{1, 1}$. For any cluster variable $x = \mathbf{P}_{[k]}$ on the left-most boundary, Lemma~\ref{lem:principal-minors} implies that $x' = \mathbf{P}_{[k-1]} \mathbf{P}_{[2, k+1]} - \mathbf{P}_{[k+1]} \mathbf{P}_{[2, k-1]}$. Suppose for the sake of contradiction that $x$ and $x'$ were not coprime. Since $\mathbf{P}_{[k]}$ is irreducible in $\mathcal{A}_{\text{sym}}$, we must have for some $f \in \mathcal{A}_{\text{sym}}$ and $g \in \mathcal{A}_{\text{sym}}$ independent of $z_{k+1, k+1}$:
	\begin{align*}
		\mathbf{P}_{[k]} \cdot f & = \mathbf{P}_{[k-1]} \mathbf{P}_{[2, k+1]} - \mathbf{P}_{[k+1]} \mathbf{P}_{[2, k-1]} \\
		& = (\mathbf{P}_{[k-1]} \mathbf{P}_{[2, k]} - \mathbf{P}_{[k]} \mathbf{P}_{[2, k-1]}) z_{k+1, k+1} + g. 
	\end{align*}
	Since the right hand side is linear in $z_{k+1, k+1}$, $f$ must also be linear in $z_{k+1, k+1}$ and we must have 
	\[
		\mathbf{P}_{[k]} \,  | \,  \mathbf{P}_{[k-1]} \mathbf{P}_{[2, k]} - \mathbf{P}_{[k]} \mathbf{P}_{[2, k-1]} \implies \mathbf{P}_{[k]} \, | \, \mathbf{P}_{[k-1]} \mathbf{P}_{[2, k]}. 
	\]
	But this is a contradiction. Finally, any other cluster variable $x = P_{[a, b]}$ for positive integer $2 \leq a < b$ has 
	\[
		x' = P_{[a-1, b]} P_{[a+1, b-1] \cup \{b+1\}} - P_{[a-1, b-1] \cup \{b+1\}} P_{[a+1, b]}.
	\]
	When we let $z_{i, b} = 0$ for $1 \leq i \leq b-a + 1$, we get $x = 0$ and $x' = z_{b-a+2, b} P_{[a-1, b-1] \cup \{b+1\}}$. This implies that $x \nmid x'$ and proves that $x$ and $x'$ are coprime.

	To prove $\mathcal{A}_{\text{sym}} \subseteq \mathcal{A}(Q)$, we want to find a mutation sequence which lets us get the generators $z_{ij}$ for $1 \leq i \leq j \leq n$ as cluster variables. We provide this mutation sequence by inducting on the $n$. The inductive hypothesis is for all smaller $n$ we can find a mutation sequence where all of the matrix entries are realized as cluster variables. 
	
	When we mutate the bottom row from left to right, the row above from left to right, and so on, then we get the quiver in  the first row from left to right, the second row from left to right, the third row from left to right, and so on, we get the quiver in~\cref{eqn:mutated-quiver-symmetric}. Note that the bottom part of the valued quiver is the initial valued quiver for $n-1$ but all of the indices are translated up by $1$. From the inductive hypothesis, we can get $z_{i, j}$ for $2 \leq i \leq j \leq n$ and the entries $z_{1,1}, \ldots, z_{1, n}$ are already in the initial quiver. This suffices for the proof. 

An alternative way to check that $\mathcal{A}_{\text{sym}} \subseteq \mathcal{A}(Q)$ is to use a centrally symmetric version of double wiring diagrams, and show that 
	each entry $z_{i,j}$ lies in such a centrally symmetric double wiring diagram. For the non-centrally symmetric case, see~\cite[Section 6.6]{FWZ6}. The argument for the centrally symmetric case is similar. 
	
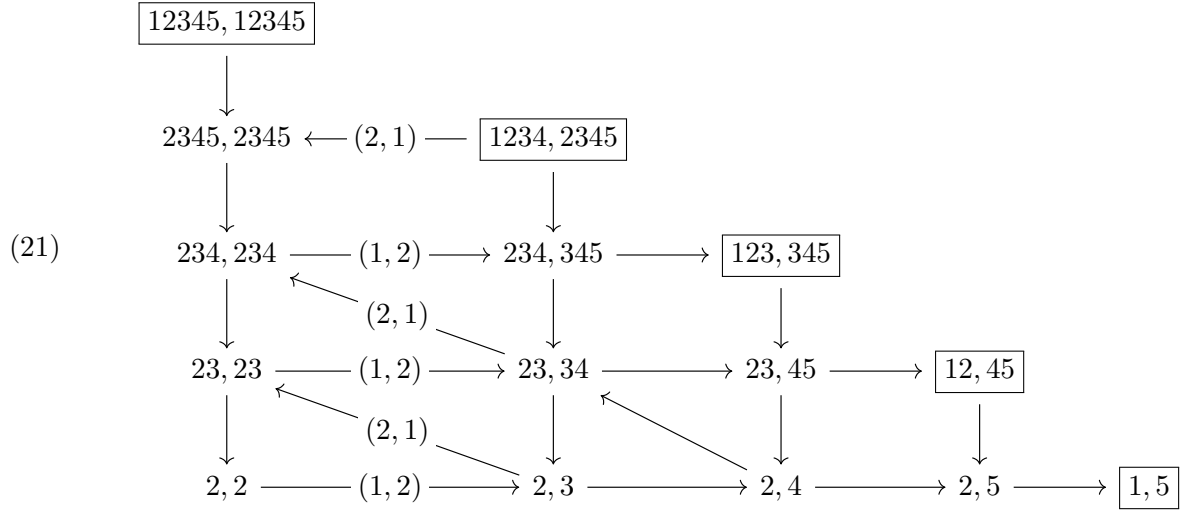
\begin{equation} \label{eqn:mutated-quiver-symmetric}\begin{tikzcd}
    {\boxed{12345,12345}} &&&&& \\
    {2345,2345} && {\boxed{1234,2345}} \\
    {234,234} && {234,345} & {\boxed{123,345}} \\
    {23,23} && {23,34} & {23,45} & {\boxed{12,45}} \\
    {2,2} && {2,3} & {2,4} & {2,5} & {\boxed{1,5}}
    \arrow[from=1-1, to=2-1]
    \arrow[from=2-1, to=3-1]
    \arrow["{(2,1)}"{description}, from=2-3, to=2-1]
    \arrow[from=2-3, to=3-3]
    \arrow["{(1,2)}"{description}, from=3-1, to=3-3]
    \arrow[from=3-1, to=4-1]
    \arrow[from=3-3, to=3-4]
    \arrow[from=3-3, to=4-3]
    \arrow[from=3-4, to=4-4]
    \arrow["{(1,2)}"{description}, from=4-1, to=4-3]
    \arrow[from=4-1, to=5-1]
    \arrow["{(2,1)}"{description}, from=4-3, to=3-1]
    \arrow[from=4-3, to=4-4]
    \arrow[from=4-3, to=5-3]
    \arrow[from=4-4, to=4-5]
    \arrow[from=4-4, to=5-4]
    \arrow[from=4-5, to=5-5]
    \arrow["{(1,2)}"{description}, from=5-1, to=5-3]
    \arrow["{(2,1)}"{description}, from=5-3, to=4-1]
    \arrow[from=5-3, to=5-4]
    \arrow[from=5-4, to=4-3]
    \arrow[from=5-4, to=5-5]
    \arrow[from=5-5, to=5-6]
\end{tikzcd}\end{equation}
\end{proof}

\begin{theorem}
	There are $|G|^{n}$ $G$-gradings for the cluster algebra $\mathcal{A}_{\text{sym}}$. The different $G$-gradings are determined by the $G$-degrees of $z_{1, 1}$ and $z_{[k], [2, k+1]}$ for $1 \leq k \leq n-1$. 
\end{theorem}

\begin{proof}
	Consider the columns $3, 4, \ldots, n$ in the initial quiver and repeatedly apply Lemma~\ref{lem:frozen-leaf} by ``shedding'' the outside frozen vertices until we have removed all vertices in these columns. We are left with the following quiver (which we have drawn for the $n = 5$ case): 
	
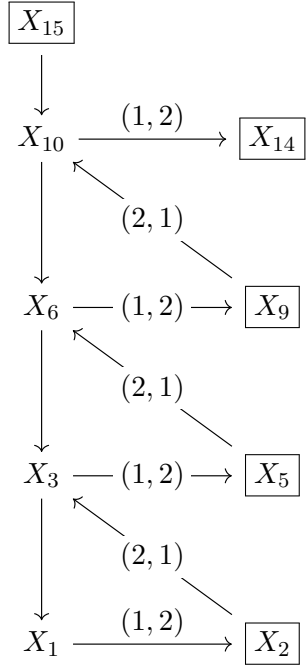
\[\begin{tikzcd}
	{\boxed{X_{15}}} && \\
		{X_{10}} && {\boxed{X_{14}}} \\
	\\
	{X_6} && {\boxed{X_9}} \\
	\\
	{X_3} && {\boxed{X_5}} \\
	\\
	{X_1} && {\boxed{X_2}}
	\arrow[from=1-1, to=2-1]
	\arrow["{(1,2)}", from=2-1, to=2-3]
	\arrow[from=2-1, to=4-1]
	\arrow["{(1,2)}"{description}, from=4-1, to=4-3]
	\arrow[from=4-1, to=6-1]
	\arrow["{(2,1)}"{description}, from=4-3, to=2-1]
	\arrow["{(1,2)}"{description}, from=6-1, to=6-3]
	\arrow[from=6-1, to=8-1]
	\arrow["{(2,1)}"{description}, from=6-3, to=4-1]
	\arrow["{(1,2)}", from=8-1, to=8-3]
	\arrow["{(2,1)}"{description}, from=8-3, to=6-1]
\end{tikzcd}\]
	Then repeatedly apply Lemma~\ref{lem:frozen-leaf} to the first column from top to bottom. 
\end{proof}

\begin{remark}
	In the paper~\cite{ross2023duallylorentzianpolynomials}, the authors give a lemma about non-negative symmetric matrices which allows one to check the Lorentzian property. Consider a symmetric $m \times m$ matrix $M$ with non-negative entries. Then, they show that $M$ has at most one positive eigenvalue if and only if for every $I \subseteq \{1, 2, \ldots, m\}$, we have $(-1)^{|I|} \det M_I \leq 0$ where $M_I$ denotes the principal minor of $M$ associated to $I$. It is an interesting question to ask if the signs $(-1)^{|I|} \det M_I \leq 0$ are compatible with one of the $2^{n}$ signed cells on $\mathcal{A}_{\text{sym}}$ that we have just constructed. Unfortunately, it is easy to see that none of the signed cells we've constructed can have the signs $(-1)^{|I|} \det M_I \leq 0$. Indeed, looking at the cluster variable $z_{1, 1}$, we immediately see that $z_{12, 12} \geq 0$. In particular, this provides an example where the signed cells do not cover the whole cluster variety. More generally, since the signs of the leading principal minors in a given signed seed of $\mathcal{A}_{\text{sym}}$ are predetermined, the eigenvalue signature is determined from Sylvester's law of inertia.\footnote{Recall that one version of Sylvester's law of inertia states that when every principal minor is non-zero, then the number of negative eigenvalues is equal to the number of sign changes in the sequence $\mathbf{P}_{\emptyset} := 1, \mathbf{P}_{[1]}, \mathbf{P}_{[2]}, \ldots, \mathbf{P}_{[n]}.$}
\end{remark}
\begin{question}
	Is there another cluster structure on symmetric matrices where the sign pattern $(-1)^{|I|} \det M_I < 0$ gives a signed cell? 
\end{question}

\section{$G$-graded
seeds for skew-symmetric matrices}\label{sec:skew}

In this section, we classify the $G$-graded seeds
for skew-symmetric matrices.  We also describe a particular signed seed,
and show how this signed seed is compatible with the positivity test for 
the totally positive orthogonal Grassmannian which was 
explicitly described in~\cite{skew}. For the remainder of this section, we let $\text{Sk}_n = \mathbb{C}[z_{i, j}] / (z_{i, j} + z_{j, i})$ be the \textit{coordinate ring of skew-symmetric matrices}. This is a polynomial ring in the $N := \binom{n}{2}$ variables $z_{i, j}$ for $1 \leq i < j \leq n$.

\subsection{Background on flag minors and Pfaffians}

\begin{definition}
	Let $R = (r_1, \ldots, r_m)$ and $S = (s_1, \ldots, s_m)$ be two sequences of elements in $[n]$. Then, we define $z_{R, S}$ to be the determinant of the $R \times S$ submatrix of $(z_{i,j})_{1 \leq i, j \leq n}$. Explicitly, we have 
	\[
		z_{R, S} := \sum_{\sigma \in \mathfrak{S}_m} \sgn(\sigma) \prod_{i = 1}^m z_{r_i s_{\sigma(i)}} \in \Sk_n. 
	\]
	For a sequence $I = (i_1 , \ldots , i_d)$ of elements in $[n]$, we define $P_I$ denote the \textbf{flag minor} $z_{[d], I}$. If $I_1, \ldots, I_k$ are sequences of elements in $[n]$, we also write $P_{I_1 \ldots I_k}$ for the flag minor corresponding to the concatenation of $I_1, \ldots, I_k$. If $S \subseteq [n]$, we write $P_S$ for the flag minor corresponding to the sequence consisting of elements of $S$ written in increasing order.

\end{definition}

\begin{proposition}[3-term Pl\"ucker incidence relation]\label{prop:3-term-plucker}
	For $\ell \geq 0$, let $S \in [n]^{\ell}$ be any sequence of length $\ell$. 
	Let $i, j, k \in [n] \backslash S$ be distinct indices. Then, we have 
	\[
		P_{ijS}P_{k S} - P_{ikS} P_{jS} + P_{jk S} P_{iS} = 0. 
	\]
\end{proposition}

\begin{definition}
    Let $I = (i_1, i_2, \ldots, i_d) \in [n]^d$. We define the \textbf{Pfaffian}, denoted  $\Pf_I$, inductively via the following rules: 
    \begin{enumerate}
        \item If $d$ is odd, then $\Pf_I = 0$.
        \item $\Pf_{\emptyset} = 1$
        \item If $d = 2$, then $\Pf_{(i_1, i_2)} = z_{i_1,i_2}$. 
		\item
		    For $d \geq 3$, we define $\Pf_I = \sum_{r = 2}^d (-1)^r z_{i_1,i_r} \Pf_{i_2, \ldots, \widehat{i_r}, \ldots, i_{d}} $
    \end{enumerate}
	For a subset $S \subset [n]$, we denote by $\Pf_S$ the Pfaffian corresponding to the elements of $S$ sorted in increasing order. 
\end{definition}

The following formula, which first appeared 
in~\cite{MR1576761},
relates Pfaffian and determinants of submatrices of skew-symmetric matrices.
For a modern account of this formula, see~\cite{MR4165809}.

\begin{theorem}[Brill's relation]\label{thm:brill-theorem}
    Let $R, S \in {[n]}^m$. Then, the following identity holds in $\Sk_n$:
    \[
        z_{R, S} = {(-1)}^{\lfloor \frac{m}{2} \rfloor} \sum_{k = 0}^{\lfloor \frac{m}{2} \rfloor} {(-1)}^k \sum_{\substack{|w| = 2k \\ w \subseteq R}} \sgn \binom{R}{w(R \backslash w)} \Pf_w \Pf_{(R\backslash w)S}.
    \]
\end{theorem}

By applying Theorem~\ref{thm:brill-theorem} to the case $R = \{1, \ldots, m\}$, we get the following formulae: 
\begin{corollary}\label{cor:brill-theorem-extremal-case}
	\phantom{h}
	\begin{enumerate}[label = (\roman*)]
		\item When $S = \{1, \ldots, \widehat{i}, \ldots, 2d, j\}$ for $1 \leq i \leq 2d$ and $j > 2d$, we have 
		\[
			P_S = \Pf_{[1, 2d]} \Pf_S. 
		\]
		\item When $S = \{1, \ldots, \widehat{i}, \ldots, 2d+1, j\}$ for $1 \leq i \leq 2d+1$ and $j > 2d+1$, we have 
		\[
			P_S = \Pf_{[1, 2d+1] \backslash \{i\}} \Pf_{[1, 2d+1] \cup \{j\}}. 
		\]
		\item When $S = \{1, \ldots, \widehat{i}, \ldots, \widehat{j}, \ldots, 2d, k, \ell\}$ for $1 \leq i < j \leq 2d$ and $2d < k < \ell$, we have 
		\[
			P_S = \Pf_{[1, 2d]} \Pf_S  - \Pf_{[1, 2d] \backslash \{i, j\}} \Pf_{[1, 2d] \cup \{k, \ell\}}.  
		\]
		\item When $S = \{1, \ldots, \widehat{i}, \ldots, \widehat{j}, \ldots, 2d+1, k, \ell\}$ for $1 \leq i < j \leq 2d+1$ and $2d+1 < k < \ell$, we have 
		\[
			P_S = \Pf_{[1, 2d+1] \backslash \{i\}} \Pf_{[1, 2d+1] \backslash \{j\}, k, \ell} - \Pf_{[1, 2d+1] \backslash \{j\}} \Pf_{[1, 2d+1] \backslash \{i\}, k, \ell}. 
		\]
	\end{enumerate}
\end{corollary}

\subsection{Cluster structure on \texorpdfstring{$\text{Sk}_n$}{Skn}}
We start by describing an initial seed for the cluster structure on $\text{Sk}_n$.

This construction comes from the work of~\cite{MR2427512} 
using cluster categorification and is extended to the quantum setting in~\cite{ZHANG2026677} using the methods from~\cite{MR3633289}. The initial seed for $\Sk_6$ and $\Sk_7$ are shown in \cref{figure:initial-seed-for-skew-symmetric-matrices-n=7}, from which the pattern for general $n$ is evident.

\begin{figure}[h]
\[\begin{tikzcd}
	{\Pf_{1,2}} & {\Pf_{1,3}} & {\Pf_{1,4}} & {\Pf_{1,5}} & {\boxed{\Pf_{1,6}} } \\
	& {\Pf_{2,3}} & {P_{3,4}} & {P_{4,5}} & {\boxed{P_{5,6}}} \\
	&& {\Pf_{1,2,3,4}} & {P_{3,4,5}} & {\boxed{P_{4,5,6}}} \\
	&&& {\boxed{\Pf_{2,3,4,5}}} & {\boxed{P_{3,4,5,6}}} \\
	&&&& {\boxed{\Pf_{1,2,3,4,5,6}}}
	\arrow[from=1-1, to=2-3]
	\arrow[from=1-2, to=1-3]
	\arrow[from=1-3, to=1-4]
	\arrow[from=1-3, to=2-3]
	\arrow[from=1-4, to=1-5]
	\arrow[from=1-4, to=2-4]
	\arrow[from=2-2, to=3-4]
	\arrow[from=2-3, to=1-2]
	\arrow[from=2-3, to=2-4]
	\arrow[from=2-4, to=1-3]
	\arrow[from=2-4, to=2-5]
	\arrow[from=2-4, to=3-4]
	\arrow[from=2-5, to=1-4]
	\arrow[curve={height=-24pt}, dashed, from=3-3, to=1-1]
	\arrow[from=3-3, to=4-5]
	\arrow[from=3-4, to=2-3]
	\arrow[from=3-4, to=3-5]
	\arrow[from=3-5, to=2-4]
	\arrow[curve={height=-30pt}, dashed, from=4-4, to=2-2]
	\arrow[from=4-5, to=3-4]
	\arrow[curve={height=-30pt}, dashed, from=5-5, to=3-3]
	\arrow[from=1-2, to=2-2]
	\arrow[from=2-3, to=3-3]
	\arrow[from=3-4, to=4-4]
\end{tikzcd}
\]

\[\begin{tikzcd}
	{\Pf_{1,2}} & {\Pf_{1,3}} & {\Pf_{1,4}} & {\Pf_{1,5}} & {\Pf_{1, 6}} & {\boxed{\Pf_{1, 7}}} \\
	& {\Pf_{2,3}} & {P_{3,4}} & {P_{4,5}} & {P_{5,6}} & {\boxed{P_{6,7}}} \\
	&& {\Pf_{1,2,3,4}} & {P_{3,4,5}} & {P_{4,5,6}} & {\boxed{P_{5,6,7}}} \\
	&&& {\Pf_{2,3,4,5}} & {P_{3,4,5,6}} & {\boxed{P_{4,5,6,7}}} \\
	&&&& {\boxed{\Pf_{1,2,3,4,5,6}}} & {\boxed{P_{3,4,5,6,7}}} \\
	&&&&& {\boxed{\Pf_{2,3,4,5,6,7}}}
	\arrow[from=1-1, to=2-3]
	\arrow[from=1-2, to=1-3]
	\arrow[from=1-3, to=1-4]
	\arrow[from=1-3, to=2-3]
	\arrow[from=1-4, to=1-5]
	\arrow[from=1-4, to=2-4]
	\arrow[from=1-5, to=1-6]
	\arrow[from=1-5, to=2-5]
	\arrow[from=2-2, to=3-4]
	\arrow[from=2-3, to=1-2]
	\arrow[from=2-3, to=2-4]
	\arrow[from=2-4, to=1-3]
	\arrow[from=2-4, to=2-5]
	\arrow[from=2-4, to=3-4]
	\arrow[from=2-5, to=1-4]
	\arrow[from=2-5, to=2-6]
	\arrow[from=2-5, to=3-5]
	\arrow[from=2-6, to=1-5]
	\arrow[curve={height=-30pt}, dashed, from=3-3, to=1-1]
	\arrow[from=3-3, to=4-5]
	\arrow[from=3-4, to=2-3]
	\arrow[from=3-4, to=3-5]
	\arrow[from=3-5, to=2-4]
	\arrow[from=3-5, to=3-6]
	\arrow[from=3-5, to=4-5]
	\arrow[from=3-6, to=2-5]
	\arrow[curve={height=-30pt}, dashed, from=4-4, to=2-2]
	\arrow[from=4-4, to=5-6]
	\arrow[from=4-5, to=3-4]
	\arrow[from=4-5, to=4-6]
	\arrow[from=4-6, to=3-5]
	\arrow[curve={height=-30pt}, dashed, from=5-5, to=3-3]
	\arrow[from=5-6, to=4-5]
	\arrow[curve={height=-30pt}, dashed, from=6-6, to=4-4]
	\arrow[from=1-2, to=2-2]
	\arrow[from=2-3, to=3-3]
	\arrow[from=3-4, to=4-4]
	\arrow[from=4-5, to=5-5]
\end{tikzcd}\]
	\caption{Initial seed for $\Sk_6$ and $\Sk_7$.}\label{figure:initial-seed-for-skew-symmetric-matrices-n=7}
\end{figure}
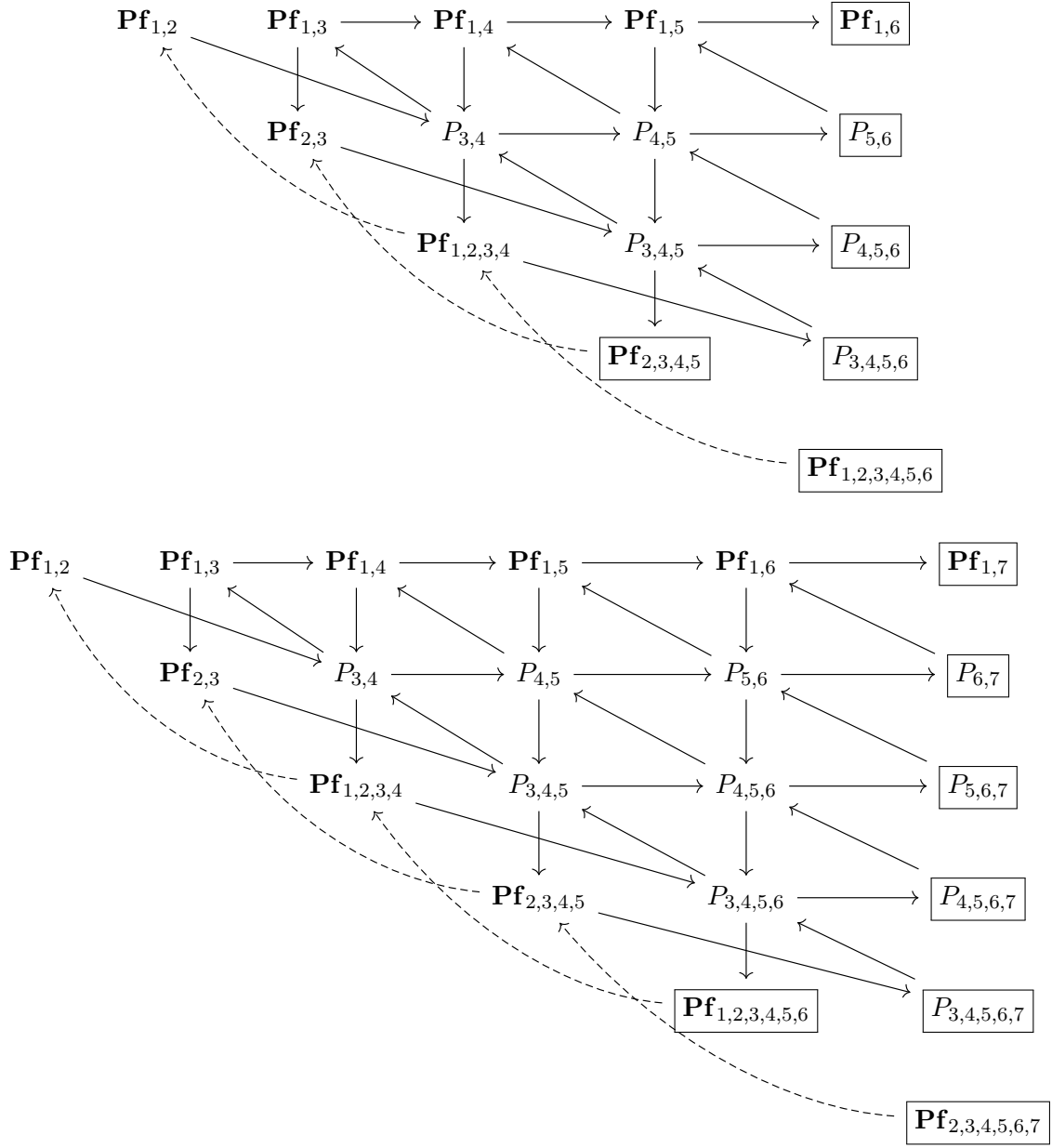

In \cref{figure:initial-seed-for-skew-symmetric-matrices-n=7}, the cluster variables are situated as if they were entries in the upper-triangular part of a $n \times n$ matrix. We use this picture to label the cluster variables as $x_1, \ldots, x_N$ for $N = \binom{n}{2}$ where the cluster variables are situated in the following fashion (again shown for $n = 6$ and $n = 7$ but the pattern is evident for general $n$):

\begin{equation}\label{eqn:matrix-position-of-cluster-variables}
\begin{bmatrix}
* & x_1 & x_2 & x_4 & x_7 & x_{11} \\
   * & *  & x_3 & x_5 & x_8 & x_{12} \\
 *   &  * & *  & x_6 & x_9 & x_{13} \\
   * & *  & *  & *  & x_{10} & x_{14} \\
   * & *  & *  &  * & *   & x_{15} \\
   * & * & * & * & * & *
	\end{bmatrix} \quad \text{ and } \quad 
   \begin{bmatrix}
	* & x_1 & x_2 & x_4 & x_7 & x_{11} & x_{16} \\
   * & *  & x_3 & x_5 & x_8 & x_{12} & x_{17} \\
 *   &  * & *  & x_6 & x_9 & x_{13} & x_{18} \\
   * & *  & *  & *  & x_{10} & x_{14} & x_{19}\\
   * & *  & *  &  * & *   & x_{15} & x_{20} \\
   * & * & * & * & * & * & x_{21} \\
	* & * & * & * & * & * & * 
	\end{bmatrix}
\end{equation}

Alternatively, we can label the cluster variables by coordinates $(r, c)$ based on its row and column indices in the matrix of \cref{eqn:matrix-position-of-cluster-variables}. For example, $x_{(1, 2)} = x_1$, $x_{(1, 3)} = x_2$, $x_{(2,3)} = x_3$, etc. The initial cluster variables, the mutable variables, and the frozen variables are 
\begin{align}
	x_{(r, c)} & = \begin{cases}
        \Pf_{[c]} & \quad \text{if } c-r = 1 \text{ and $c$ is even}, \\
        \Pf_{[2, c]} & \quad \text{if } c-r = 1 \text{ and $c$ is odd,} \\
        \Pf_{1, c} & \quad \text{if } r = 1, \\
       	P_{[c-r+1,c]} & \quad \text{otherwise.}
    \end{cases}\label{eqn:initial-cluster-skew-symmetric-matrices} \\ 
	\mathbf{mutable} & = \left \{x_k \mid 1 \leq k \leq \binom{n-1}{2} - 1 \right \} \label{eqn:exchangeable-vertices}, \\
	\mathbf{frozen} & = \left \{ x_k \mid \binom{n-1}{2} \leq k \leq \binom{n}{2} \right \}. \label{eqn:frozen-vertices} 
\end{align}

The fact that this data (\cref{eqn:initial-cluster-skew-symmetric-matrices}, \cref{eqn:exchangeable-vertices}, \cref{eqn:frozen-vertices}, \cref{figure:initial-seed-for-skew-symmetric-matrices-n=7}) gives a cluster structure on $\Sk_n$ is proved in~\cite[Section 13.1]{MR2427512} and later in the quantum setting in~\cite[Theorem 4.6]{ZHANG2026677}.

\subsection{\texorpdfstring{$G$}{G}-gradings on \texorpdfstring{$\Sk_n$}{Sk(n)}}

We can compute the number of $G$-gradings on the cluster structure
on $\Sk_n$, where $G$ is an arbitrary finite abelian group. 
There is a natural ${(\CC^*)}^n$-action on $\Sk_n$ induced by the $\ZZ^n$-grading $\deg(z_{i, j}) = \mathbf{e}_i + \mathbf{e}_j$. In coordinates, we have $(t_1, \ldots, t_n) . z_{i, j} = t_i t_j z_{i, j}$. The cluster variables in \cref{eqn:initial-cluster-skew-symmetric-matrices} are homogeneous with respect to this grading, with 
\[
	\deg (\Pf_I) = \sum_{i \in I} \mathbf{e}_i, \quad \deg P_{I} = \sum_{i \in I} \mathbf{e}_i + \sum_{j \leq |I|} \mathbf{e}_j
\]
whenever they are non-zero. From the initial quiver shown in \cref{figure:initial-seed-for-skew-symmetric-matrices-n=7}, it is evident that all exchange relations are compatible with this $\ZZ^n$-grading. Hence, this torus action defines a $\ZZ^n$-grading on the cluster structure of $\Sk_n$, and in turn, defines up to $|G|^{n}$ distinct $G$-gradings from the construction in~\cref{sec:torus-action}. Explicitly, any choice of $g_1, \ldots, g_n \in G$ gives a $G$-grading of the form 
\[
	\sigma (\Pf_I) = \sum_{i \in I} \mathbf{g}_i, \quad \sigma (P_{I}) = \sum_{i \in I} \mathbf{g}_i + \sum_{j \leq |I|} \mathbf{g}_j
\]
In this section, we compute the exact number of distinct $G$-gradings this produces for arbitrary $G$ along with the structure of $\mathfrak{gr}(\Sk_n, G)$ up to the torus action. 

\begin{theorem}
	There are $|G|^n$ distinct $G$-gradings on the cluster algebra structure of $\Sk_n$ for $n \geq 3$. Moreover, the $G$-gradings are determined by the $G$-degrees of the variables $\Pf_{1,2}, \Pf_{1,3}, \ldots, \Pf_{1,n}$, and $\Pf_{2,3}$.
\end{theorem}

\begin{proof}
	Apply Lemma~\ref{lem:frozen-leaf} repeatedly on the last column of \cref{figure:initial-seed-for-skew-symmetric-matrices-n=7} from bottom to top. We are left with the initial quiver for $\Sk_{n-1}$ and an isolated vertex. From Lemma~\ref{lem:isolated-vertex}, we have
	\[
		\# \{\text{$G$-gradings on $\Sk_n$}\} = |G| \cdot \# \{\text{$G$-gradings on $\Sk_{n-1}$}\}. 
	\]
	We get that the number of $G$-gradings on $\Sk_n$ is $|G|^n$ by induction. 
\end{proof}

\begin{theorem}
	Let $G$ be a finite abelian group. Then there are $|G|^n / |G[2]|$ gradings which come from the torus action. As a consequence, there are $|G[2]|$ distinct gradings up to the torus action. 
\end{theorem}

\begin{proof}
	For $2 \leq k \leq n$, the torus degree is given by $\deg (\Pf_{1, k}) = \mathbf{e}_1 + \mathbf{e}_k$. Let $\mathbf{g} = (g_1, \ldots, g_n) \in G^{\oplus n}$ be an element in the kernel of $\Psi_G : G^{\oplus n} \to \mathfrak{gr}(\Sk_n, G)$. Then, $g_1 + g_k = 0$ for all $2 \leq k \leq n$ and $\mathbf{g}$ is in the form 
	\[
		\mathbf{g} = (g, -g, -g, \ldots, -g)
	\]
	for some $g \in G$. From $\deg (\Pf_{2, 3})$, we necessarily need $2g = 0$. Finally, the condition $2g = 0$ is enough to guarantee that $\mathbf{g}$ is in the kernel since all degrees of initial cluster variables lie in the subgroup 
	\[
		\{(t_1, \ldots, t_n) \in \ZZ^n : t_1 + \cdots + t_n \equiv 0 \pmod 2 \}. 
	\]
	Thus, $\ker \Psi_G \cong G[2]$ which suffices for the proof. 
\end{proof}

We can compute all such $G$-gradings for $\Sk_n$ by providing representatives for $\mathfrak{gr}(\Sk_n, G)$ modulo the torus grading. It suffices to do this computation for $G = \ZZ / p^{\ell} \ZZ$ and $p$ prime. When $p$ is odd, then $G[2]$ is trivial. Hence, all $G$-gradings come from the torus action in this case. Otherwise, $G = \ZZ / 2^{\ell} \ZZ$ and $|G[2]| = 2$. To characterize all $G$-gradings up to the torus action, it suffices to find a $G$-grading which does not come from the torus action. Since any assignment of $G$-degrees to $\Pf_{1,k}$ for $2 \leq k \leq n$ and $\Pf_{2,3}$ defines a unique well-defined $G$-grading on $\Sk_n$, there is a $G$-grading for which $\sigma(\Pf_{1,k}) = 1$ and $\sigma(\Pf_{2,3}) = 1$ for all $2 \leq k \leq n$. This does not come from the torus grading because otherwise we would have $g_1, \ldots, g_n \in G$ satisfying $g_1 + g_k = 1$ for $2 \leq k \leq n$ and $g_2 + g_3 = 1$. This implies that $g_2 = \cdots = g_n$ and $2g_2 = 1$. But this cannot happen in $\ZZ / 2^{\ell} \ZZ$, which is a contradiction.

\subsection{Total positivity on \texorpdfstring{$\Sk_n$}{Sk(n)}}

In~\cite{skew}, the authors identify $\Sk_n$ with the coordinate ring of an open affine chart of the orthogonal Grassmannian. 
They study the space of skew-symmetric matrices which are \emph{totally positive} in the sense of~\cite{Lusztig3}. To this end, they provide the following (signed) positivity test which specify the signs of a collection of $\binom{n}{2}$ minors.
	
\begin{theorem}
	\cite[Theorem 1.4]{skew}\footnote{Our
	minors are the transpose of the minors listed in~\cite[Theorem 1.4]{skew}. We have changed the signs accordingly to match the current convention.} 
	\label{thm:total-positivity-test-skew}
	A skew-symmetric matrix $A$ is totally positive if and only if
	\[
		\sgn \left ( P_{[1, j-1] \cup [n-k+j, n]} \right ) = (-1)^{k + (n-j)(n-k)} \text{ for any } 1 \leq j \leq k \leq n-1. 
	\]
	When this holds, we also have $\sgn(\Pf_I) = (-1)^{\sum_{i \in I} i - \binom{|I| + 1}{2}}$ for all subsets $I \subseteq [n]$ of even size. As one special case, we have 
	\[
		\sgn (  P_{[1,a] \cup [b, n]} ) = (-1)^{ab + n(a + b)} \text{ for } a < b \text{ and } b-a \geq 2. 
	\]
\end{theorem}

Theorem~\ref{thm:total-positivity-test-skew} gives us a signed positivity test for skew-symmetric minors with $\dim \Sk_n$ minors. We might 
hope that we can 
explain this result by constructing a cluster structure on $\Sk_n$ with the minors in 
Theorem~\ref{thm:total-positivity-test-skew} 
as cluster variables.

Unfortunately, the collection of minors cannot form a cluster because the minors 
in Theorem~\ref{thm:total-positivity-test-skew} 
are not all irreducible. However, we can show that there are exactly $\binom{n}{2}$ distinct irreducible factors among the minors, suggesting the possibility of having a cluster structure with the irreducible factors as cluster variables.

\begin{proposition}\label{prop:irreducible-factors-of-skew}
    When $n = 2m$ is even, the distinct irreducible factors of the minors in Theorem~\ref{thm:total-positivity-test-skew} are
    \begin{equation}
		\begin{cases}
			\mathcal{M}_1 & = \{ P_{[j-1] \cup [n-k+j, n]}  \mid 1 \leq j < k \leq n-2 \} \\ 
			\mathcal{M}_2 & = \{ \Pf_{[2k]}\mid 1 \leq k \leq m \} \\
			\mathcal{M}_3 & = \{\Pf_{[2k-1] \cup\{n\}}  \mid 1 \leq k \leq m-1 \} \\
			\mathcal{M}_4 & = \{\Pf_{[n-1] \backslash \{k\}}\mid 1 \leq k \leq n-2 \}.
		\end{cases}
	\end{equation}
    When $n = 2m-1$ is odd, the distinct irreducible factors of the minors in Theorem~\ref{thm:total-positivity-test-skew} are 
	\begin{equation}
		\begin{cases}
			\mathcal{M}_1 & = \{ P_{[j-1] \cup [n-k+j, n]} \mid 1 \leq j < k \leq n-2 \} \\
			\mathcal{M}_2 & = \{ \Pf_{[2k]} \mid 1 \leq k \leq m-1 \} \\
			\mathcal{M}_3 & = \{ \Pf_{[2k-1] \cup \{n\}} \mid 1 \leq k \leq m-1 \} \\
			\mathcal{M}_4 & = \{ \Pf_{[n] \backslash \{k\}} \mid 1 \leq k \leq n-2 \}. 
		\end{cases}
	\end{equation}
\end{proposition}

This proposition relies on the computation of Corollary~\ref{cor:brill-theorem-extremal-case}, the fact that Pfaffians are irreducible, and the following lemma which guarantees that the minors not covered by Corollary~\ref{cor:brill-theorem-extremal-case} are irreducible.

\begin{lemma}
    Let $M = (x_{ij})$ be a skew-symmetric matrix with indeterminate entries. Then $z_{I, J}$ is irreducible if and only if $|I \Delta J| \geq 4$, where $I \Delta J$ denotes the symmetric difference of $I$ and $J$.
\end{lemma}

\begin{proof}
    From Corollary~\ref{cor:brill-theorem-extremal-case}, $z_{I, J}$ is reducible when $|I \Delta J| \leq 2$. It remains to prove that $z_{I, J}$ is irreducible when $|I \Delta J| \geq 4$. Equivalently, consider a $n \times n$ matrix $(X_{ij})_{1 \leq i, j \leq n}$ and let $1 \leq k \leq n-2$. We want to prove that the determinant polynomial remains irreducible when we set $X_{ji} = -X_{ij}$ for all $1 \leq i < j \leq k$ and $X_{ii} = 0$ for $1 \leq i \leq k$. It is enough to prove this for the case $k = n-2$. Indeed, if the determinant polynomial becomes factorizable for $k < n-2$, then it must be a homogeneous factorization. To recover the determinant polynomial for $k = n-2$, we specialize some of the diagonal variables to $0$ and $X_{ij} = -X_{ji}$ for some more indices $i, j$, we preserve the homogeneous degrees of the factors. Hence, we either get that the determinant polynomial for $k = n-2$ is zero or is factorizable. For the rest of the proof, we prove that the determinant is irreducible in the $k = n-2$ case, i.e. the case where $|I \Delta J| = 4$. 
    
    In this case, we have $I = S \cup \{a, b\}$ and $J = S \cup \{c, d\}$ where $a, b, c, d$ are distinct indices. When $S = \emptyset$, then $z_{I, J}$ is the determinant of a generic $2 \times 2$ matrix, which is irreducible. 
	
	Now let $|S| > 0$ and without loss of generality, we can assume $\max(S) < a < b$ and $\max (S) < c < d$. The polynomial $z_{I, J}$ is linear in each of the collection of variables $\{X_{a, *}\}$, $\{ X_{b, *} \} $, $\{ X_{*, c} \}$, and $\{ X_{*, d} \}$. For example, using Laplace expansion in the $a$th row and Corollary~\ref{cor:brill-theorem-extremal-case}, we get the formula 
	\begin{equation}\label{eqn:laplace-expansion-of-skew-determinant}
		z_{I, J} = \sum_{s \in S} X_{as} \cdot Y_s
	\end{equation}
	where the $Y_s$ are products of two Pfaffians not containing any variable of the form $X_{a*}$. We can find similar formulas for $X_{b,*}$, $X_{*, c}$, and $X_{*, d}$. In particular, this implies that $z_{I, J}$ is non-zero and any variables of the form $X_{a, *}$, $X_{b, *}$, $X_{*, c}$, and $X_{*, d}$ appears in $z_{I, J}$ with non-zero coefficient. 
	
	Suppose for the sake of contradiction that $z_{I, J} = P \cdot Q$ for $P, Q$ non-trivial. Without loss of generality, we can assume that $X_{a, c}$ appears in $P$ and not $Q$. Since $z_{I, J}$ is linear in the collections $\{X_{a, *}\}$ and $\{X_{*, c}\}$, this implies that any variable of the form $X_{a, *}$ and $X_{*, c}$ appears in $P$ and does not appear in $Q$ since $X_{a, c} X_{a, *}$ or $X_{a, c} X_{*, c}$ is forbidden to appear in the product. Thus, $X_{a, d}$ and $X_{b, c}$ appear in $P$, and by the same argument all variables of the form $X_{a, *}, X_{b, *}, X_{*, c}$, and $X_{*, d}$ appear in $P$ and $Q \in \CC[X_{s_1s_2}]_{s_1,s_2 \in S}$. From Equation~\ref{eqn:laplace-expansion-of-skew-determinant}, it is clear that $Q$ must be trivial since Pfaffians are irreducible. This suffices for the proof. 
\end{proof}

We now provide a signed seed on $\Sk_n$ and prove it recovers the positivity test of~\cite{skew}, which connects total positivity to the cluster algebra structure on the orthogonal Grassmannian.

\begin{theorem}\label{thm:proof-of-positivity-test-matches-cluster-structure}
Recall the extended cluster 
	$\{x_{(r,c)}\}$ from 
    \eqref{eqn:initial-cluster-skew-symmetric-matrices}.
	Then the assignment $\sgn(x_{(r, c)}) = (-1)^{rc}$ defines a signed seed 
	$\sigma_{Sk}$ on $\Sk_n$. Using $\sigma_{Sk}$, the irreducible factors in Proposition~\ref{prop:irreducible-factors-of-skew} get the signs specified in Theorem~\ref{thm:total-positivity-test-skew}. As a consequence, a skew-symmetric matrix is totally positive 
if and only if it comes from the signed seed $\sigma_{Sk}$ of the cluster algebra on $\Sk_n$. 
\end{theorem}
\begin{remark}
\cite[Theorem 6]{Speyer} shows that there are many skew-symmetric matrices all of whose
Pfaffians are nonnegative.  This subset of skew-symmetric matrices is different
from the totally positive subset of skew-symmetric matrices from 
\cref{thm:proof-of-positivity-test-matches-cluster-structure}; it would be interesting
to explore whether there is a relation.

\end{remark}

\begin{proof}
	Fix some point of the cluster torus defined by the initial seed from which we can evaluate our cluster variables. We will prove the following two claims by simultaneous induction which are equivalent to the total positivity test from~\cite{skew}. 

	\textbf{Claim 1.} Let $n = 2m$ be even. Suppose that the following expressions are real-valued with
	\[
		\begin{cases}
			\text{$\sgn \left (\Pf_{1,k} \right ) = (-1)^k$} \quad & \text{for $3 \leq k \leq n$.} \\
			\text{$\sgn \left (\Pf_{[2k]} \right ) = 1$} \quad & \text{for $1 \leq k \leq m$.} \\
			\text{$\sgn \left (\Pf_{[2, 2k-1]} \right ) = 1$} \quad & \text{for $2 \leq k \leq m$.} \\
			\text{$\sgn(P_{[i, j]}) = (-1)^{ij}$} \quad & \text{for $3 \leq i < j \leq n$}
		\end{cases}
	\]
	Then, the following expressions are real-valued with 
	\[
		\begin{cases}
			\text{$\sgn \left ( \Pf_{[2k]} \right ) = 1$} \quad & \text{for $1 \leq k \leq m$}, \\
			\text{$\sgn (P_{[j, n]}) = 1$} \quad & \text{for $3 \leq j \leq n-1$}, \\
			\text{$\sgn (P_{[a] \cup [b, n]}) = (-1)^{ab}$} \quad & \text{for $4 \leq a+3 \leq b \leq n-1$}, \\
			\text{$\sgn(\Pf_{[n-1] \backslash \{i\}}) = (-1)^{i+1}$} \quad & \text{for $1 \leq i \leq n-2$} \\
			\text{$ \sgn \left ( \Pf_{[2k-1]  \cup \{n\}} \right ) = 1$} \quad & \text{for $1 \leq k \leq m-1$}. 
		\end{cases}
	\]

	\textbf{Claim 2.} Let $n = 2m+1$ be odd. Suppose that the following expressions are real-valued with  
	\[
		\begin{cases}
			\text{$\sgn(\Pf_{1,k}) = (-1)^k$} \quad & \text{for $3 \leq k \leq n$} \\
			\text{$\sgn(\Pf_{[2k]}) = 1$} \quad & \text{for $1 \leq k \leq m$} \\
			\text{$\sgn(\Pf_{[2, 2k+1]}) = 1$} \quad & \text{for $1 \leq k \leq m$} \\
			\text{$\sgn(P_{[i, j]}) = (-1)^{ij}$} \quad & \text{for $3 \leq i < j \leq n$}
		\end{cases}
	\]
	Then, the following expressions are real-valued with 
	\[
		\begin{cases}
			\text{$\sgn(\Pf_{[2k]}) = 1$} \quad & \text{for $1 \leq k \leq m$}, \\ 
			\text{$\sgn(P_{[j, n]}) = (-1)^j$} \quad & \text{for $3 \leq j \leq n-1$}, \\
			\text{$\sgn(P_{[a] \cup [b, n]}) = (-1)^{ab+a+b}$} \quad & \text{for $4 \leq a+3 \leq b \leq n-1$} \\
			\text{$\sgn(\Pf_{[n] \backslash \{i\}}) = (-1)^{i+1}$} \quad & \text{for $1 \leq i \leq n-2$} \\
			\text{$\sgn(\Pf_{[2k-1] \cup \{n\}}) = -1$} \quad & \text{for $1 \leq k \leq m-1$} \\
			\text{$\sgn(\Pf_{[2k-1] \cup \{n\}}) = -1$} \quad & \text{for $k = m$}. 
		\end{cases}
	\] 
	It is easy to check the two claims for $n = 4$ and $n = 5$. Let's prove the claim for $n = 2m$ where $m \geq 3$ assuming that Claim 1 and Claim 2 are true for all smaller values of $n$. Since the inductive hypothesis of $n = 2m$ includes the inductive hypothesis for $n-1 = 2m-1$, we have
	\[
		\begin{cases}
			\text{$\sgn (P_{[j, n-1]}) = (-1)^{j}$} \quad & \text{for $3 \leq j \leq n-2$} \\
			\text{$\sgn(P_{[a] \cup [b,n-1]}) = (-1)^{ab + a + b}$} \quad & \text{for $4 \leq a +3 \leq b \leq n-2$} \\
			\text{$\sgn(\Pf_{[n-1] \backslash \{i\}}) = (-1)^{i+1}$} \quad & \text{for $1 \leq i \leq n-3$} \\
			\text{$\sgn (\Pf_{[2k-1] \cup \{n-1\}}) = -1$} \quad & \text{for $1 \leq k \leq m-2$} \\ 
			\text{$\sgn (\Pf_{[2k-1] \cup \{n-1\}}) = -1$} \quad & \text{for $k = m-1$}.     
		\end{cases}
	\]
	In general, for $3 \leq a+3 \leq b \leq n-1$, Proposition~\ref{prop:3-term-plucker} gives the relation
	\begin{equation}\label{eqn:inductive-3-term-plucker}
		P_{[a] \cup [b, n-1]} P_{[a+1] \cup [b+1, n]} = P_{[a+1] \cup [b+1, n-1]} P_{[a] \cup [b, n]} + P_{[a+1] \cup [b, n-1]} P_{[a] \cup [b+1, n]}. 
	\end{equation}
	In the special case where $b -a = 3$, we can further factor the term $P_{[a+1] \cup [b, n-1]}$ using Corollary~\ref{cor:brill-theorem-extremal-case}
	\begin{equation}\label{eqn:factorization-of-minor}
		P_{[a+1] \cup [b, n-1]} =
		\begin{cases}
			\Pf_{[n-2]} \Pf_{[n-1] \backslash \{a+2\}} & \text{ if $n$ is even} \\
			\Pf_{[n-1]}\Pf_{[n-2] \backslash \{a+2\}} & \text{ if $n$ is odd}. 
		\end{cases}
	\end{equation}
	We will use this relation to prove that $\sgn (P_{[a] \cup [b, n]}) = (-1)^{ab}$ for all $4 \leq a+3 \leq b \leq n-1$. Consider the possible pairs $(a, b)$ satisfying $4 \leq a + 3 \leq b \leq n-1$ in lexicographic order where $(a, b) < (a', b')$ if and only if $a < a'$ or $a = a'$ and $b < b'$. We prove the claim from lowest to highest in lexicographic order. For the base case $a = 1$ and $b = 4$, we can compute  
	\begin{align*}
		P_{[1] \cup [4, n]} P_{[3, n-1]} & = P_{[1] \cup [4, n-1]} P_{[3, n]} + P_{[1] \cup [3, n-1]} P_{[4, n]} \\ 
		& = P_{[1] \cup [4, n-1]} P_{[3, n]} + \Pf_{[n-1] \backslash \{2\}} \Pf_{[n-2]} P_{[4, n]}. 
	\end{align*}
	From the inductive hypothesis, this implies that $\sgn(P_{[1] \cup [4, n]}) = 1$ which matches the claim. In general, from \cref{eqn:inductive-3-term-plucker}, we have 
	\begin{equation}\label{eqn:minor-equation-1}
		P_{[a+1] \cup [b+1, n]} = \begin{dcases*}
			\frac{P_{[a+1] \cup [b+1, n-1]} P_{[a] \cup [b, n]} + P_{[a+1] \cup [b, n-1]} P_{[a] \cup [b+1, n]}}{P_{[a] \cup [b, n-1]}} & \text{if $b-a \geq 4$} \\
			\frac{P_{[a+1] \cup [b+1, n-1]} P_{[a] \cup [b, n]} + \Pf_{[n-2]} \Pf_{[n-1] \backslash \{a+2\}} P_{[a] \cup [b+1, n]}}{P_{[a] \cup [b, n-1]}} & \text{if $b-a = 3$}. 
		\end{dcases*}
	\end{equation}
	The inductive hypothesis gives $\sgn ( P_{[1, a+1] \cup [b+1, n]})  = (-1)^{(a+1)(b+1)}$ which verifies that these types of minors have the desired signs. Now, for $n$ even, it remains to show that $\sgn (\Pf_{[2k-1] \cup \{n\}}) = 1$ for $1 \leq k \leq m-1$. From the inductive hypothesis, this already holds for $k = 1$. For $k \geq 2$, Corollary~\ref{cor:brill-theorem-extremal-case} gives us the relation 
	\begin{equation}
		\Pf_{[2k-1] \cup \{n\}} \Pf_{[2k-3] \cup \{n-1\}} = P_{[2k-3] \cup \{n-1, n\}} + \Pf_{[2k-1] \cup \{n-1\}} \Pf_{[2k-3] \cup \{n\}}. 
	\end{equation}
	From the inductive hypothesis, we get $\sgn (\Pf_{[2k-1] \cup \{n\}}) = 1$ which completes the induction for the case $n$ is even. 

	The proof of the claim for $n = 2m+1$ where $m \geq 3$ is similar. We assume that Claim 1 and Claim 2 are true for all smaller values of $n$. In particular, the claim is true for $n-1 = 2m$ which gives 
\[
\begin{cases}
	\text{$\sgn (P_{[j, n-1]}) = 1$} \quad & \text{for $3 \leq j \leq n-2$}, \\
	\text{$\sgn (P_{[a] \cup [b, n-1]}) = (-1)^{ab}$} \quad & \text{for $4 \leq a+3 \leq b \leq n-2$}, \\
	\text{$\sgn(\Pf_{[n-2] \backslash \{i\}}) = (-1)^{i+1}$} \quad & \text{for $1 \leq i \leq n-3$} \\
	\text{$ \sgn \left ( \Pf_{[2k-1]  \cup \{n-1\}} \right ) = 1$} \quad & \text{for $1 \leq k \leq m-1$}. 
\end{cases}
\]
To prove that for $4 \leq a + 3 \leq b \leq n-1$ we have $\sgn (P_{[a] \cup [b, n]}) = (-1)^{ab + a + b}$, we similarly induct based on the lexicographic order of $(a, b)$. We have 
\[
	P_{[a+1] \cup [b+1, n]} = \begin{dcases*}
		\frac{P_{[a+1] \cup [b+1, n-1]} P_{[a] \cup [b, n]} + P_{[a+1] \cup [b, n-1]} P_{[a] \cup [b+1, n]}}{P_{[a] \cup [b, n-1]}} & \text{if $b-a \geq 4$} \\
		\frac{P_{[a+1] \cup [b+1, n-1]} P_{[a] \cup [b, n]} + \Pf_{[n-1]} \Pf_{[n-2] \backslash \{a+2\}} P_{[a] \cup [b+1, n]}}{P_{[a] \cup [b, n-1]}} & \text{if $b-a = 3$}
	\end{dcases*}
\] 
which is the same as \cref{eqn:minor-equation-1} except for the Pfaffians in the second case due to the parity condition in \cref{eqn:factorization-of-minor}. From the inductive hypothesis, we have the desired signs for minors of this form. Finally, we want to verify that 
\[
	\begin{cases}
		\text{$\sgn(\Pf_{[n] \backslash \{i\}}) = (-1)^{i+1}$} \quad & \text{for $1 \leq i \leq n-2$} \\
		\text{$\sgn(\Pf_{[2k-1] \cup \{n\}}) = -1$} \quad & \text{for $1 \leq k \leq m-1$} \\
		\text{$\sgn (\Pf_{[n-2] \cup \{n\}}) = \sgn(\Pf_{[n] \backslash \{n-1\}}) = -1$} \quad & \text{otherwise.}
	\end{cases}
\]
The last Pfaffian can be considered as a Pfaffian of the first type for $i = n-1$ and a Pfaffian of the second type for $k = m$. For this reason, we save the computation of its sign for last. 

The signs of $\Pf_{[n] \backslash \{1\}}$ and $\Pf_{1, n}$ are determined by the inductive hypothesis. For general $2 \leq i \leq n-2$, Corollary~\ref{cor:brill-theorem-extremal-case} gives us the relation  
	\begin{equation}
		\Pf_{[n] \backslash \{i\}}  = \frac{P_{[i-2] \cup [i+1, n]} + \Pf_{[n] \backslash \{i-1\}} \Pf_{[n-2] \backslash \{i\}}}{\Pf_{[n-2] \backslash \{i-1\}}}.
	\end{equation}
	This implies that we can induct on $i$ to prove $\sgn (\Pf_{[n] \backslash \{i\}}) = (-1)^{i+1}$, which is exactly what we wanted to show. For $2 \leq k \leq m-1$: 
	\[
		\Pf_{[2k-1] \cup \{n\}} = \frac{P_{[2k-3] \cup \{n-1, n\}} + \Pf_{[2k-1] \cup \{n-1\}} \Pf_{[2k-3] \cup \{n\}}}{\Pf_{[2k-3] \cup \{n-1\}}}. 
	\]
	From the inductive hypothesis, we get $\sgn (\Pf_{[2k-1] \cup \{n\}}) = -1$ again by an inductive argument, which matches the claim. 
	
	The final sign we need to verify is $\Pf_{[1, n-2] \cup \{n\}}$. This comes from another relation of Corollary~\ref{cor:brill-theorem-extremal-case}
	\[
		\Pf_{[n-2] \cup \{n\}} = \frac{P_{[n-4] \cup \{n-1, n\}} + \Pf_{[1, n-1]} \Pf_{[1, n-4] \cup \{n\}}}{\Pf_{[1, n-4] \cup \{n-1\}}}
	\]
	From the inductive hypothesis, we get $\sgn (\Pf_{[1, n-2] \cup \{n\}}) = -1$ which completes the induction. 
\end{proof}

\begin{remark}
	In the proof above, we technically only prove the equivalence between signed tests. However, the result as stated is still true. The relations which we use in the proof of \cref{thm:proof-of-positivity-test-matches-cluster-structure} are exactly the exchange relations of a mutation sequence which turns the initial seed of $\Sk_n$ to a seed which contains exactly the irreducible factors of the minors in~\cite{skew}. The mutations occur in exactly the same order in which we introduced the exchange relations. However, we could not find an enlightening proof of this fact other than a monstrous induction, which we omit. 
\end{remark}

\section{$G$-graded seeds for the Grassmannian and positroid varieties}\label{sec:Grassmannian}	
Let $\Gr_{k,n}$ denote the 
Grassmannian of $k$-dimensional subspaces
in~$\CC^n$. Any element of $\Gr_{k,n}$  can be represented (non-uniquely)
as the row span of a full rank $k \times n$ matrix $C$.
Given such a matrix and a $k$-element subset $J$ of $[n]$,
we let $P_J(C)$ denote the determinant of the $k \times k$ submatrix of $C$
located in columns $J$; this is called a \emph{Pl\"ucker coordinate}.
The Grassmannian $\Gr_{k,n}$ 
can be embedded into projective space of dimension $\binom{n}{k}-1$
via the Pl\"ucker embedding $C \mapsto (P_J(C))_{J\in {[n] \choose k}}$.
Note that the torus $(\C^*)^n$ acts on $\Gr_{k,n}$ by rescaling the columns of $C$.
In particular, if $t:=(t_1,\dots,t_n)$

is an element of the torus,  and $C$ is a $k \times n$ matrix as above, then
$P_{J}(t \cdot C) = \prod_{j\in J} t_j P_{J}(C)$.

Let $\widehat{\Gr}_{k,n}$ denote the affine cone over ${\Gr}_{k,n}$
in the Pl\"ucker embedding.
The 
homogeneous coordinate ring 
 $\CGr{k}{n}$
of~$\Gr_{k,n}$
is generated by the Pl\"ucker coordinates~$P_J$,
where $J$ ranges over all elements of ${[n] \choose k}$.
The Pl\"ucker coordinates satisfy the \emph{Grassmann-Pl\"ucker relations},
which are preserved by the torus action on the Grassmannian, and hence
$\CGr{k}{n}$ is $\Z^n_{\geq 0}$-graded.
Concretely,
$F\in \CGr{k}{n}$ 
lies in the component 
$\CGr{k}{n}_{(a_1,\dots,a_n)}$ 
with grading $(a_1,\dots,a_n)\in
\Z^n_{\geq 0}$ if it can be written as a polynomial in Pl\"ucker coordinates
where the index $i$ appears $a_i$ times in each term.

A particularly nice subset of the Grassmannian is the \emph{totally positive part}.
The definition below appeared in \cite{Postnikov}; a Lie-theoretic definition
which turned out to be equivalent in the case of the Grassmannian
\cite{TalaskaWilliams} was given for all partial flag varieties
$G/P$ in \cite{Lusztig3}.

\begin{definition}\label{def:TP}\cite{Postnikov, Lusztig3}
The \emph{(totally) positive part} $\Gr_{k,n}^{>0}$
	of $\Gr_{k,n}$ is the subset where all Pl\"ucker coordinates are strictly positive.  The \emph{totally nonnegative part} $\Gr_{k,n}^{\geq 0}$ of 
	$\Gr_{k,n}$ is the subset where all Pl\"ucker coordinates are nonnegative.
\end{definition}

\begin{remark}
Since the set of Pl\"ucker coordinates are coordinates on
projective space, we could equivalently define the totally positive part of the 
Grassmannian by requiring that all Pl\"ucker coordinates have the same sign.
\end{remark}

\subsection{The cluster structure on the Grassmannian}\label{sec:clusterGr}
As shown in \cite{Scott}, the ring $\CGr{k}{n}$ has a cluster algebra structure, 
see also \cite[Section 6.7]{FWZ6}. 
The $n$ Pl\"ucker coordinates whose indices are cyclically consecutive 
are the frozen variables.

In the case of $\Gr_{2,n}$, the seeds for the cluster structure
are in bijection with the triangulations of an $n$-gon, as shown at the 
left of 
\cref{fig:G37-quiver}, and the set of cluster variables $\Xcal$
coincides with the Pl\"ucker coordinates $\{P_{ij} \ \vert \ 1 \leq i < j \leq n\}$.  
More specifically, 
given a triangulation $T$ of an $n$-gon with vertices labeled $1,\dots, n$ in order,
we construct a quiver $Q_T$ by placing a mutable vertex on every diagonal of $T$ and a
frozen vertex on every side of the polygon.  Inside every triangle of $T$
we place three arrows which form a clockwise cycle (arrows between frozen vertices
can be omitted since they do not affect mutation).  Each vertex of $Q_T$ gets an associated
cluster variable obtained by identifying the corresponding side or diagonal 
 by a Pl\"ucker coordinate whose indices encode the two endpoints.  This gives us the data of 
a seed for the cluster structure, see \cref{fig:G37-quiver}.

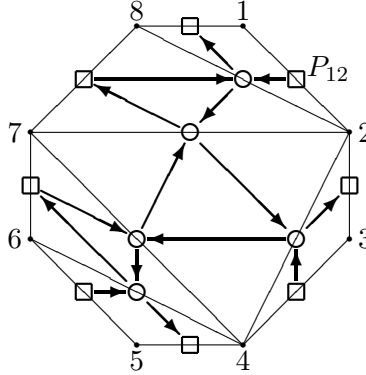
\begin{figure}[ht]
\begin{center}
\vspace{-.2cm}
\setlength{\unitlength}{2pt}
\begin{picture}(60,60)(0,0)
\put(56,52){\makebox(0,0){$P_{12}$}}
 
\put(40,-3){\makebox(0,0){$4$}}
\put(63,20){\makebox(0,0){$3$}}
\put(63,40){\makebox(0,0){$2$}}
\put(40,63){\makebox(0,0){$1$}} 
\put(20,63){\makebox(0,0){$8$}} 
\put(-3,40){\makebox(0,0){$7$}}
\put(-3,20){\makebox(0,0){$6$}}
\put(20,-3){\makebox(0,0){$5$}}

\thinlines
  \multiput(0,20)(60,0){2}{\line(0,1){20}}
  \multiput(20,0)(0,60){2}{\line(1,0){20}}
  \multiput(0,40)(40,-40){2}{\line(1,1){20}}
  \multiput(20,0)(40,40){2}{\line(-1,1){20}}

  \multiput(20,0)(20,0){2}{\circle*{1}}
  \multiput(20,60)(20,0){2}{\circle*{1}}
  \multiput(0,20)(0,20){2}{\circle*{1}}
  \multiput(60,20)(0,20){2}{\circle*{1}}

\put(40,0){\line(1,2){20}}
\put(0,40){\line(1,0){60}}
\put(0,20){\line(2,-1){40}}
\put(0,40){\line(1,-1){40}}
\put(20,60){\line(2,-1){40}}
\thicklines

\multiput(-1.5,28.5)(3,0){2}{\line(0,1){3}}
\multiput(-1.5,28.5)(0,3){2}{\line(1,0){3}}

\multiput(58.5,28.5)(3,0){2}{\line(0,1){3}}
\multiput(58.5,28.5)(0,3){2}{\line(1,0){3}}

\multiput(28.5,-1.5)(3,0){2}{\line(0,1){3}}
\multiput(28.5,-1.5)(0,3){2}{\line(1,0){3}}

\multiput(28.5,58.5)(3,0){2}{\line(0,1){3}}
\multiput(28.5,58.5)(0,3){2}{\line(1,0){3}}

\multiput(8.5,8.5)(3,0){2}{\line(0,1){3}}
\multiput(8.5,8.5)(0,3){2}{\line(1,0){3}}

\multiput(48.5,8.5)(3,0){2}{\line(0,1){3}}
\multiput(48.5,8.5)(0,3){2}{\line(1,0){3}}

\multiput(8.5,48.5)(3,0){2}{\line(0,1){3}}
\multiput(8.5,48.5)(0,3){2}{\line(1,0){3}}

\multiput(48.5,48.5)(3,0){2}{\line(0,1){3}}
\multiput(48.5,48.5)(0,3){2}{\line(1,0){3}}

{\put(20,10){\circle{3}}}
{\put(20,20){\circle{3}}}
{\put(30,40){\circle{3}}}
{\put(40,50){\circle{3}}}
{\put(50,20){\circle{3}}}

\linethickness{1.5pt}

\put(12,10){{\vector(1,0){6}}}
\put(22,8){{\vector(1,-1){6}}}

\put(2,29){{\vector(2,-1){16}}}
\put(20,18){{\vector(0,-1){6}}}
\put(18,12){{\vector(-1,1){16}}}

\put(21,22){{\vector(1,2){8}}}
\put(32,38){{\vector(1,-1){16}}}
\put(48,20){{\vector(-1,0){26}}}

\put(28,41){{\vector(-2,1){16}}}
\put(12,50){{\vector(1,0){26}}}
\put(38,48){{\vector(-1,-1){6}}}

\put(38,52){{\vector(-1,1){6}}}
\put(48,50){{\vector(-1,0){6}}}

\put(50,12){{\vector(0,1){6}}}
\put(52,22){{\vector(1,1){6}}}
\end{picture}
\end{center}

\caption{
The quiver $Q_T$ and seed for $\Gr_{2,8}$ obtained from a triangulation $T$ of an octagon.}
\vspace{-.2cm}
\label{fig:G37-quiver}
\end{figure}

For $k>2$, we do not in general have a classification of the cluster variables 
and seeds for the cluster structure on 
$\CGr{k}{n}$; while all Pl\"ucker coordinates are cluster variables, there are in general infinitely
many cluster variables.
However, 
for any $k$ and $n$
there is a nice initial seed called the \emph{rectangles seed}
which is shown at the right of 
\cref{fig:G37-quiver} for the case of $\Gr_{3,7}$.  
The quiver $Q_{k,n}$ is a \emph{triangulated grid quiver}, and 
the cluster/frozen variables of this seed are Pl\"ucker coordinates which are in bijection with the rectangles
which fit inside a $k \times (n-k)$ rectangle.
To map a rectangle $r \subset k \times (n-k)$ to a $k$-element subset of $[n]$, 
we first position $r$ so that it fits into the upper left of the $k \times (n-k)$ rectangle. We then 
traverse the path from the southwest corner to the northeast corner of the $k \times (n-k)$ rectangle
which cuts out $r$, labeling the steps of this path from $1$ to $n$, then read off the labels
of the vertical steps.\footnote{Sometimes one uses the convention in which we traverse the 
path from northeast to southwest. However, our convention has the advantage that the 
Pl\"ucker coordinate labels of $Q_{k,n}$ are a subset of those
of $Q_{k,n+1}$.}

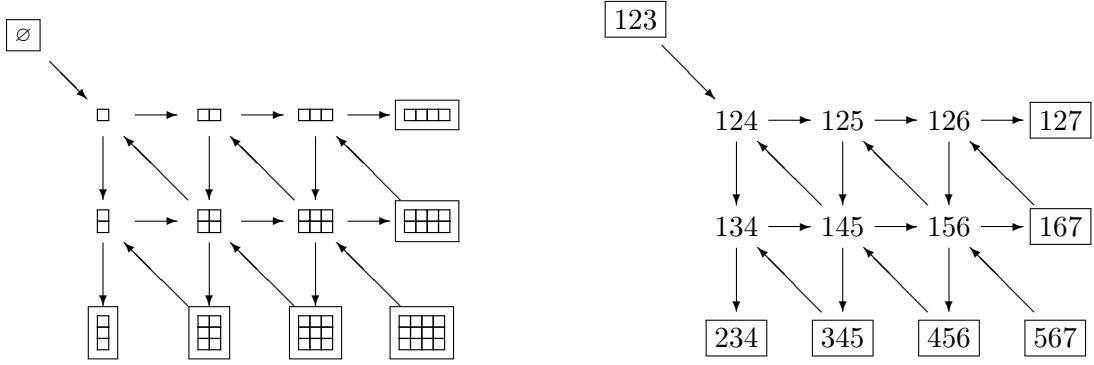
\begin{figure}[ht]
\begin{center}
 \setlength{\unitlength}{2pt}
\begin{picture}(80,55)(0,-1) 
\put(20,20){\makebox(0,0){${\ydiagram{1,1}}$}}
\put(20,40){\makebox(0,0){${\ydiagram{1}}$}}
\put(40,20){\makebox(0,0){${\ydiagram{2,2}}$}}
\put(40,40){\makebox(0,0){${\ydiagram{2}}$}}
\put(60,20){\makebox(0,0){${\ydiagram{3,3}}$}}
\put(60,40){\makebox(0,0){${\ydiagram{3}}$}}

\put(5,55){\makebox(0,0){$\boxed{\scriptstyle\varnothing}$}}
\put(20,-1){\makebox(0,0){$\boxed{\ydiagram{1,1,1}}$}}
\put(40,-1){\makebox(0,0){$\boxed{\ydiagram{2,2,2}}$}} 
\put(60,-1){\makebox(0,0){$\boxed{\ydiagram{3,3,3}}$}} 
\put(80,-1){\makebox(0,0){$\boxed{\ydiagram{4,4,4}}$}}
\put(81,20){\makebox(0,0){$\boxed{\ydiagram{4,4}}$}}
\put(81,40){\makebox(0,0){$\boxed{\ydiagram{4}}$}}

\put(20,16){\vector(0,-1){12}}
\put(40,36){\vector(0,-1){12}}

\put(20,36){\vector(0,-1){12}}
\put(60,36){\vector(0,-1){12}}

\put(26,20){\vector(1,0){8}}
\put(66,20){\vector(1,0){8}}
\put(46,40){\vector(1,0){8}}

\put(26,40){\vector(1,0){8}}
\put(66,40){\vector(1,0){8}}
\put(46,20){\vector(1,0){8}}

\put(10,50){\vector(1,-1){7}}

\put(40,16){\vector(0,-1){12}}
\put(60,16){\vector(0,-1){12}}

\put(36, 4){\vector(-1,1){12}}
\put(56, 4){\vector(-1,1){12}}
\put(76, 4){\vector(-1,1){12}}
\put(36, 24){\vector(-1,1){12}}
\put(56, 24){\vector(-1,1){12}}
\put(76, 24){\vector(-1,1){12}}
\end{picture}\hspace{2.5cm}
\setlength{\unitlength}{2pt}
\iffalse
	\begin{picture}(80,63)(0,0)
\put(20,20){\makebox(0,0){${457}$}}
\put(20,40){\makebox(0,0){${467}$}}
\put(40,20){\makebox(0,0){${347}$}}
\put(40,40){\makebox(0,0){${367}$}}
\put(60,20){\makebox(0,0){${237}$}}
\put(60,40){\makebox(0,0){${267}$}}

\put(1,59){\makebox(0,0){$\boxed{567}$}}
\put(20,-1){\makebox(0,0){$\boxed{456}$}}
\put(40,-1){\makebox(0,0){$\boxed{345}$}}
\put(60,-1){\makebox(0,0){$\boxed{234}$}}
\put(80,-1){\makebox(0,0){$\boxed{123}$}}
\put(81,20){\makebox(0,0){$\boxed{127}$}}
\put(81,40){\makebox(0,0){$\boxed{167}$}}

\put(20,16){\vector(0,-1){12}}
\put(40,36){\vector(0,-1){12}}

\put(20,36){\vector(0,-1){12}}
\put(60,36){\vector(0,-1){12}}

\put(26,20){\vector(1,0){8}}
\put(66,20){\vector(1,0){8}}
\put(46,40){\vector(1,0){8}}

\put(26,40){\vector(1,0){8}}
\put(66,40){\vector(1,0){8}}
\put(46,20){\vector(1,0){8}}

\put(6,54){\vector(1,-1){10}}

\put(40,16){\vector(0,-1){12}}
\put(60,16){\vector(0,-1){12}}

\put(36, 4){\vector(-1,1){12}}
\put(56, 4){\vector(-1,1){12}}
\put(76, 4){\vector(-1,1){12}}
\put(36, 24){\vector(-1,1){12}}
\put(56, 24){\vector(-1,1){12}}
\put(76, 24){\vector(-1,1){12}}
\end{picture}\fi
\begin{picture}(80,63)(0,0)
\put(20,20){\makebox(0,0){${134}$}}
\put(20,40){\makebox(0,0){${124}$}}
\put(40,20){\makebox(0,0){${145}$}}
\put(40,40){\makebox(0,0){${125}$}}
\put(60,20){\makebox(0,0){${156}$}}
\put(60,40){\makebox(0,0){${126}$}}

\put(1,59){\makebox(0,0){$\boxed{123}$}}
\put(20,-1){\makebox(0,0){$\boxed{234}$}}
\put(40,-1){\makebox(0,0){$\boxed{345}$}}
\put(60,-1){\makebox(0,0){$\boxed{456}$}}
\put(80,-1){\makebox(0,0){$\boxed{567}$}}
\put(81,20){\makebox(0,0){$\boxed{167}$}}
\put(81,40){\makebox(0,0){$\boxed{127}$}}

\put(20,16){\vector(0,-1){12}}
\put(40,36){\vector(0,-1){12}}

\put(20,36){\vector(0,-1){12}}
\put(60,36){\vector(0,-1){12}}

\put(26,20){\vector(1,0){8}}
\put(66,20){\vector(1,0){8}}
\put(46,40){\vector(1,0){8}}

\put(26,40){\vector(1,0){8}}
\put(66,40){\vector(1,0){8}}
\put(46,20){\vector(1,0){8}}

\put(6,54){\vector(1,-1){10}}

\put(40,16){\vector(0,-1){12}}
\put(60,16){\vector(0,-1){12}}

\put(36, 4){\vector(-1,1){12}}
\put(56, 4){\vector(-1,1){12}}
\put(76, 4){\vector(-1,1){12}}
\put(36, 24){\vector(-1,1){12}}
\put(56, 24){\vector(-1,1){12}}
\put(76, 24){\vector(-1,1){12}}
\end{picture}
\end{center}
\caption{The quiver $Q_{3,7}$. 
At left, the vertices are labeled by rectangles contained in a $3 \times 4$ rectangle,
arranged so that the width and height of the rectangles increase
from left to right and from top to bottom, respectively.
At right, the vertices are labeled by the corresponding $k$-element subsets of $[n]$;
we identify these with Pl\"ucker coordinates.}\label{fig:G37-Le-quiver2}
\end{figure}

If $C$ is (a matrix representing) 
an element of the Grassmannian $\Gr_{k,n}$ and $t:=(t_1,\dots,t_n)$

is an element of the torus, then for any $k$-element
subset $I=\{i_1, \dots, i_k\}$, we have that 
$P_{I}(t \cdot C) = t_{i_1} \dots t_{i_k} P_{I}(C)$.
It is not hard to see that if $C$ lies in a signed cell
of $\widehat{\Gr}_{k,n}$ and $t\in (\R^{*})^n$, then 
$t \cdot C$ also lies in a signed cell.
Equivalently, if $\sigma: \Xcal \to \{-1,1\}$ is a 
signed grading of $\widehat{\Gr}_{k,n}$, and $t=(t_1,\dots,t_n)\in \{-1,1\}^n$,
then we can define a new
signed grading $t \cdot \sigma: \Xcal \to \{-1,1\}$ as follows.
If $x\in \Xcal$ lies in the component 
$\CGr{k}{n}_{(a_1,\dots,a_n)}$, then  
$$ t\cdot \sigma(x) = t_1^{a_1} \dots t_n^{a_n} \sigma(x).$$

\subsection{The positive and alternating signed gradings for the 
Grassmannian $\Gr_{k,n}$}

Let $\Xcal$ be the set of cluster variables for $\Gr_{k,n}$. In this section, we prefer to work with the geometry of $\Gr_{k, n}$ rather than $\widehat{\Gr}_{k, n}$. A signed cell of $\Gr_{k, n}$ is defined to be the image of a signed cell of $\widehat{\Gr}_{k, n}$ under the projection map $\widehat{\Gr}_{k,n} \backslash \{0\} \to \Gr_{k, n}$. Note that the signed cells in $\widehat{\Gr}_{k, n}$ come in pairs: for a signed grading $\sigma : \mathcal{X} \to \{\pm 1\}$ we have the negative grading $-\sigma : \mathcal{X} \to \{ \pm 1\}$ which is defined by $(-\sigma)(x) := (-1) \cdot \sigma(x)$ whenever $x$ is a Pl\"ucker coordinate. This gives a well-defined signed grading, as we can check on the rectangles seed. Each pair $\{\sigma, -\sigma\}$ of signed gradings of $\widehat{\Gr}_{k, n}$ corresponds to one (projectivized) signed cell of $\Gr_{k, n}$.

\begin{definition}
The \emph{positive grading} for $\Gr_{k,n}$
is the grading \begin{equation} \label{eq:pos}
	\sigma^{+}: \Xcal \to \{-1,1\} \hspace{.5cm} \text{ such that }
	\hspace{.5cm}
	\sigma^{+}(P_{I}) = 1 \text{ for all }I.
\end{equation}
	And the \emph{totally positive part} 
	$\Gr_{k,n}^{>0}$ of the Grassmannian
	$\Gr_{k,n}$ is the subset 
$$\Gr_{k,n}^{>0}:=\{C \ \vert \ \sgn(P_{i_1,\dots,i_k}(C)) = 1.\}.$$

\end{definition}

\begin{definition}\label{def:alt}
	The \emph{alternating grading} for $\Gr_{k,n}$
	is the grading \begin{equation} \label{eq:alt}
		\sigma^{\alt}: \Xcal \to \{-1,1\} \hspace{.5cm} \text{ such that }
		\hspace{.5cm}
		\sigma^{\alt}(P_{i_1,\dots,i_k}) = (-1)^{i_1+ \dots + i_k}.
	\end{equation}
	And the \emph{alternating part} $\Gr_{k,n}^{\alt}$ of the Grassmannian $\Gr_{k,n}$ is the subset 
	\[
		\Gr_{k,n}^{\alt}:=\{C \ \vert \ \sgn(P_{i_1,\dots,i_k}(C)) = (-1)^{i_1+\dots+i_k}\}.
	\]
\end{definition}

For $V\in \Gr_{k,n}$, 
let $V^{\perp}\in \Gr_{n-k,n}$ denote the 
orthogonal complement of $V$. 
For $v=(v_1,\dots,v_n)\in \R^n$, let $\alt(v):=(v_1,-v_2,v_3,-v_4,\dots,
(-1)^{n-1} v_n)\in \R^n$. Similarly for $V\in \Gr_{k,n}$
we define $\alt(V)$ to be the span of all vectors
$\alt(v)$ where $v\in V$.

The following lemma first appeared in \cite{hilbert, hochster}, 
see also \cite[Lemma 1.11]{karp}.

\begin{lemma}\label{lem:alt1}
For any $V\in \Gr_{k,n}$, the subspaces $V$
and $\alt(V^{\perp})\in \Gr_{n-k,n}$ have the same Pl\"ucker coordinates,
that is,
$$P_I(V) = P_{[n]\setminus I}(\alt(V^{\perp}))
	\hspace{.5cm}
	\text{ for all }I \in {[n] \choose k}.$$

\end{lemma}

\begin{lemma}\label{lem:alt2}
We have that 
\begin{equation}
	\Gr_{k,n}^{\alt} = \{V^{\perp} \ \vert \ V\in \Gr_{n-k,n}^{>0} \},
\end{equation} that is, 
the alternating part $\Gr_{k,n}^{\alt}$ of the Grassmannian 
$\Gr_{k,n}$
is isomorphic to the totally positive part $\Gr_{n-k,n}^+$
of the Grassmannian $\Gr_{n-k,n}$ via the map 
$V \mapsto V^{\perp}$, where $V^{\perp}$ denotes
the orthogonal complement of $V$.
\end{lemma}
\begin{proof}
Note that for $V\in \Gr_{k,n}$ and $I=\{i_1,\dots,i_k\}$, 
we have $P_I(\alt(V)) = (-1)^{i_1+\dots+i_k} P_I(V)$.
	Using this observation together with \cref{lem:alt1}, it follows that
	 for $V\in \Gr_{n-k,n}^{>0}$,
	\begin{equation*}
		P_I(V^{\perp}) =(-1)^{i_1+\dots+i_k}P_I(\alt(V^{\perp})) = 
		(-1)^{i_1+\dots+i_k} P_{[n]\setminus I}(V) = 
		(-1)^{i_1+\dots+i_k}.
	\end{equation*}
	The result follows.
\end{proof}

\begin{proposition}\label{lem:posalt}
The gradings $\sigma^+$ and $\sigma^{\alt}$
come from signed seeds, and hence 
$\Gr_{k,n}^{>0}$ and 
$\Gr_{k,n}^{\alt}$ are homeomorphic to open balls.
\end{proposition}
\begin{proof}
We can prove the first statement by verifying that  
\eqref{eq:grading2} holds, and using Pl\"ucker relations 
to deduce how the value of these gradings propagates away from the 
initial seed.
The fact that $\Gr_{k,n}^{>0}$ and 
$\Gr_{k,n}^{\alt}$ are homeomorphic to open balls then follows from 
	\cref{thm:test}.   Note however that 
	it was already known from \cite{Postnikov} that $\Gr_{k,n}^{>0}$
	is homeomorphic to an open ball.
\end{proof}

\begin{remark}
By Mnev's Universality  Theorem~\cite{Mnev}, the topology of the 
realization space of an oriented matroid (that is, the 
subset of $\Gr_{k,n}$ with a fixed sign pattern on Pl\"ucker coordinates) can be as complicated as that of any algebraic variety. 
However, it follows from \cref{thm:test} that any signed cell of 
$\Gr_{k,n}$ is homeomorphic to an open ball.  
Besides 
the totally positive part of $\Gr_{k,n}$,
the alternating part of $\Gr_{k,n}$
provides another nice subset with a fixed sign pattern which is 
topologically a ball.
\end{remark}

\begin{remark}
Sometimes it is more convenient to work with
 $\sigma^{\alt}(P_{I}) = (-1)^{i_1+ \dots + i_k+1}$ 
	instead of \eqref{eq:alt}.
This does not affect the corresponding (projectivized) signed cell of the Grassmannian, since Pl\"ucker coordinates are only defined up to a common scalar.
\end{remark}

\begin{example}
Consider the Grassmannian $\Gr_{2,4}$ 
and the alternating grading $\sigma^{\alt}: \Xcal \to \{-1,1\}$
defined by $\sigma^{\alt}(P_{ij}) = (-1)^{i+j+1}$.
Then $\Gr_{2,4}^{\alt}$ is the subset of $\Gr_{2,4}$
in which $P_{12}, P_{23}, P_{34},$ and $P_{14}$ are positive, while
$P_{13}$ and $P_{24}$ are negative.
More generally, $\Gr_{2,n}^{\alt}$ is the subset of 
$\Gr_{2,n}$ in which $P_{i,j}$ is positive if and only if 
$i+j$ is odd.
\end{example}

\subsection{A warmup: signed seeds for the Grassmannian $\Gr_{2,n}$}

In this section we will 
show that all signed cells of $\Gr_{2,n}$ can be obtained
from the positive part via the torus action. In the proof below, we actually show that they can be obtained from either the positive part or the alternating part of $\Gr_{2,n}$ via the torus action, but the corresponding cells are already in the same torus orbit. In \cref{sec:Grass} we will generalize this result to the Grassmannian $\Gr_{k,n}$.

\begin{theorem}\label{thm:G2nsigned}
Every signed cell of the Grassmannian $\Gr_{2,n}$ 
(with respect to the cluster structure described in \cref{sec:clusterGr}) is obtained from  the positive part of $\Gr_{2,n}$ via the torus action.
\end{theorem}

\begin{proof}
	Consider an arbitrary signed cell $V^{\sigma}_{>0}$ of $\Gr_{2,n}$, 
	and let $C\in V^{\sigma}_{>0}$. 
If $t:=(t_1,\dots,t_n)$

is an element of the torus,  then 
$P_{ij}(t \cdot C) = t_i t_j P_{ij}(C)$.

Let $\{p_{ij}\ \vert \ 1 \leq i<j \leq n\} \in \{\pm 1\}^{n \choose 2}$
denote the signs of the Pl\"ucker coordinates of $C$.
We will find an element of the torus 
 $t:=(t_1,\dots,t_n)
\in \{\pm 1\}^n$  such that the Pl\"ucker coordinates
$P_{i, i+1}(t \cdot C)$ are positive for all $1 \leq i \leq n-1$.
To do so, let $t_1=1$, and  choose $t_2$ such that 
$p_{12}=t_1 t_2.$  Next we choose $t_3$ such that 
$p_{23} = t_2 t_3.$  We continue in this way, finally choosing $t_n$
such that $p_{n-1,n}=t_{n-1} t_n$.
Now we see that for $1 \leq i \leq n-1$, 
we get $\sgn(P_{i,i+1}(t \cdot C)) = t_i t_{i+1} \sgn(P_{i,i+1}(C)) = 
t_i t_{i+1} p_{i,i+1} =1$.

Our task is thus reduced to classifying the gradings
$\sigma=(\sigma_{ij})_{ij \in {[n] \choose 2}}$ from signed
seeds such that 
$\sigma_{i,i+1}>0$ for $1\leq i \leq n-1$.
Consider the rectangles seed for $\Gr_{2,n}$, shown for $n=9$ in 
\cref{fig:G29-quiver}.

\begin{figure}[ht]
\begin{center}
\setlength{\unitlength}{2pt}
\begin{picture}(145,43)(0,0)
\put(20,20){\makebox(0,0){${79}$}}

\put(40,20){\makebox(0,0){${69}$}}

\put(60,20){\makebox(0,0){${59}$}}
\put(80,20){\makebox(0,0){${49}$}}
\put(100,20){\makebox(0,0){${39}$}}
\put(120,20){\makebox(0,0){${29}$}}

\put(1,39){\makebox(0,0){$\boxed{89}$}}
\put(-8,39){\makebox(0,0){$+$}}
\put(20,-1){\makebox(0,0){$\boxed{78}$}}
\put(20,-8){\makebox(0,0){$+$}}
\put(40,-1){\makebox(0,0){$\boxed{67}$}}
\put(40,-8){\makebox(0,0){$+$}}
\put(60,-1){\makebox(0,0){$\boxed{56}$}}
\put(60,-8){\makebox(0,0){$+$}}
\put(80,-1){\makebox(0,0){$\boxed{45}$}}
\put(80,-8){\makebox(0,0){$+$}}
\put(100,-1){\makebox(0,0){$\boxed{34}$}}
\put(100,-8){\makebox(0,0){$+$}}
\put(120,-1){\makebox(0,0){$\boxed{23}$}}
\put(120,-8){\makebox(0,0){$+$}}
\put(140,-1){\makebox(0,0){$\boxed{12}$}}
\put(140,-8){\makebox(0,0){$+$}}
\put(140,20){\makebox(0,0){$\boxed{19}$}}

\put(20,16){\vector(0,-1){12}}
\put(26,20){\vector(1,0){8}}
\put(66,20){\vector(1,0){8}}
\put(46,20){\vector(1,0){8}}
\put(86,20){\vector(1,0){8}}
\put(106,20){\vector(1,0){8}}
\put(126,20){\vector(1,0){8}}

\put(6,34){\vector(1,-1){10}}

\put(40,16){\vector(0,-1){12}}
\put(60,16){\vector(0,-1){12}}
\put(80,16){\vector(0,-1){12}}
\put(100,16){\vector(0,-1){12}}
\put(120,16){\vector(0,-1){12}}

\put(36, 4){\vector(-1,1){12}}
\put(56, 4){\vector(-1,1){12}}
\put(76, 4){\vector(-1,1){12}}
\put(96, 4){\vector(-1,1){12}}
\put(116, 4){\vector(-1,1){12}}
\put(136, 4){\vector(-1,1){12}}
\end{picture}

\end{center}

		\caption{The rectangles seed for $\Gr_{2,9}$.}
\label{fig:G29-quiver}
\end{figure}

Using 
\eqref{eq:grading2}
 and the fact that $\sigma_{i,i+1}>0$ for 
$1 \leq i \leq n-1$,  we see that  for $\Gr_{2,9}$,
we must have $\sigma_{79}=\sigma_{59}=\sigma_{39} = \sigma_{19}$
and $\sigma_{69}=\sigma_{49}=\sigma_{29}$.
More generally, we have $\sigma_{1n}=\sigma_{in}$ whenever $i$ is odd,
and $\sigma_{2n} = \sigma_{in}$ whenever $i$ is even.
Moreover, applying \eqref{eq:grading2} to the vertex labeled by 
the Pl\"ucker coordinate $P_{n-2,n}$, we see that 
$\sigma_{n-3,n}>0$.
It follows that  $\sigma_{in}=1$ when $i$ and $n$ have different parities.

This gives us the signs shown in  
\cref{fig:G29-quiver2}.
\begin{figure}[ht]
\begin{center}
\setlength{\unitlength}{2pt}
\begin{picture}(145,43)(0,0)
\put(20,20){\makebox(0,0){${79}$}}

\put(40,20){\makebox(0,0){${69}$}}
\put(40,26){\makebox(0,0){$+$}}

\put(60,20){\makebox(0,0){${59}$}}
\put(80,20){\makebox(0,0){${49}$}}
\put(80,26){\makebox(0,0){$+$}}
\put(100,20){\makebox(0,0){${39}$}}
\put(120,20){\makebox(0,0){${29}$}}
\put(120,26){\makebox(0,0){$+$}}

\put(1,39){\makebox(0,0){$\boxed{89}$}}
\put(-8,39){\makebox(0,0){$+$}}
\put(20,-1){\makebox(0,0){$\boxed{78}$}}
\put(20,-8){\makebox(0,0){$+$}}
\put(40,-1){\makebox(0,0){$\boxed{67}$}}
\put(40,-8){\makebox(0,0){$+$}}
\put(60,-1){\makebox(0,0){$\boxed{56}$}}
\put(60,-8){\makebox(0,0){$+$}}
\put(80,-1){\makebox(0,0){$\boxed{45}$}}
\put(80,-8){\makebox(0,0){$+$}}
\put(100,-1){\makebox(0,0){$\boxed{34}$}}
\put(100,-8){\makebox(0,0){$+$}}
\put(120,-1){\makebox(0,0){$\boxed{23}$}}
\put(120,-8){\makebox(0,0){$+$}}
\put(140,-1){\makebox(0,0){$\boxed{12}$}}
\put(140,-8){\makebox(0,0){$+$}}
\put(140,20){\makebox(0,0){$\boxed{19}$}}

\put(20,16){\vector(0,-1){12}}
\put(26,20){\vector(1,0){8}}
\put(66,20){\vector(1,0){8}}
\put(46,20){\vector(1,0){8}}
\put(86,20){\vector(1,0){8}}
\put(106,20){\vector(1,0){8}}
\put(126,20){\vector(1,0){8}}

\put(6,34){\vector(1,-1){10}}

\put(40,16){\vector(0,-1){12}}
\put(60,16){\vector(0,-1){12}}
\put(80,16){\vector(0,-1){12}}
\put(100,16){\vector(0,-1){12}}
\put(120,16){\vector(0,-1){12}}

\put(36, 4){\vector(-1,1){12}}
\put(56, 4){\vector(-1,1){12}}
\put(76, 4){\vector(-1,1){12}}
\put(96, 4){\vector(-1,1){12}}
\put(116, 4){\vector(-1,1){12}}
\put(136, 4){\vector(-1,1){12}}
\end{picture}

\end{center}

	\caption{The rectangles seed for $\Gr_{2,9}$.}
\label{fig:G29-quiver2}
\end{figure}
Finally we can either let $\sigma_{in}=1$ when $i$ and $n$ have the same parity, 
or we can let $\sigma_{in}=-1$ when $i$ and $n$ have the same parity.
Either choice
satisfies \eqref{eq:grading2} and hence gives a signed seed.
The first choice gives the totally positive part of $\Gr_{2,n}$,
and the second choice gives the alternating part.
\end{proof}

\subsection{Signed seeds for the Grassmannian $\Gr_{k,n}$}\label{sec:Grass}

In this section we will show that all signed cells of $\Gr_{k,n}$ can be obtained from the positive part of $\Gr_{k,n}$ via the torus action. Recall that we are working with the geometry of $\Gr_{k, n}$ rather than $\widehat{\Gr}_{k, n}$, so the positive and alternating parts are torus equivalent to the positive part. The signed cells on $\Gr_{k, n}$ are exactly the projections of the signed cells of the affine cone. We will return to general $G$-gradings on $\widehat{\Gr}_{k, n}$ later in~\cref{sec:G-grading-for-Gr}. 

\begin{theorem}\label{thm:Gknsigned}
	Every signed cell of the Grassmannian $\Gr_{k,n}$ (with respect to the cluster structure described in \cref{sec:clusterGr}) is obtained from the positive part $\Gr_{k,n}$ via the torus action.
\end{theorem}

\begin{proof}
	Consider an arbitrary signed cell $V^{\sigma}_{>0}$ of $\Gr_{k,n}$, 
	and let $C\in V^{\sigma}_{>0}$. 
If $t:=(t_1,\dots,t_n)$ is an element of the torus, and $I=\{i_1<\dots < i_k\}$, then $P_{I}(t \cdot C) = t_{i_1} \dots t_{i_k} P_{I}(C)$.

We will prove \cref{thm:Gknsigned} by induction on $n$.  The base case
of $\Gr_{k,k+1}$ is trivial.
For the inductive step, we note that the seed for $\Gr_{k,n-1}$ is 
a subseed for $\Gr_{k,n}$: the quiver $Q_{k,n-1}$ and its Pl\"ucker coordinate labels
can be obtained from the corresponding quiver $Q_{k,n}$ by removing the rightmost 
column of vertices (and any incident arrows).
For example, the seed for $\Gr_{3,6}$ can be obtained from 
the seed for $\Gr_{3,7}$ shown in \cref{fig:G37-signed} by removing
the vertices labeled by $127, 167, 567$.

\begin{figure}[ht]
\begin{center}
 \setlength{\unitlength}{2pt}
\begin{picture}(80,63)(0,0)
\put(20,20){\makebox(0,0){${134}$}}
\put(23,25){\makebox(0,0){${+}$}}
\put(20,40){\makebox(0,0){${124}$}}
\put(20,45){\makebox(0,0){${+}$}}
\put(40,20){\makebox(0,0){${145}$}}
\put(43,25){\makebox(0,0){${+}$}}
\put(40,40){\makebox(0,0){${125}$}}
\put(40,45){\makebox(0,0){${+}$}}
\put(60,20){\makebox(0,0){${156}$}}
\put(63,25){\makebox(0,0){${+}$}}
\put(60,40){\makebox(0,0){${126}$}}
\put(60,45){\makebox(0,0){${+}$}}

\put(1,59){\makebox(0,0){$\boxed{123}$}}
\put(10,59){\makebox(0,0){${+}$}}
\put(20,-1){\makebox(0,0){$\boxed{234}$}}
\put(23,5){\makebox(0,0){${+}$}}
\put(40,-1){\makebox(0,0){$\boxed{345}$}}
\put(43,5){\makebox(0,0){${+}$}}
\put(60,-1){\makebox(0,0){$\boxed{456}$}}
\put(63,5){\makebox(0,0){${+}$}}
\put(80,-1){\makebox(0,0){$\boxed{567}$}}
\put(80,5){\makebox(0,0){${c}$}}
\put(81,20){\makebox(0,0){$\boxed{167}$}}
\put(81,26){\makebox(0,0){${b}$}}
\put(81,40){\makebox(0,0){$\boxed{127}$}}
\put(81,46){\makebox(0,0){${a}$}}

\put(20,16){\vector(0,-1){12}}
\put(40,36){\vector(0,-1){12}}

\put(20,36){\vector(0,-1){12}}
\put(60,36){\vector(0,-1){12}}

\put(26,20){\vector(1,0){8}}
\put(66,20){\vector(1,0){8}}
\put(46,40){\vector(1,0){8}}

\put(26,40){\vector(1,0){8}}
\put(66,40){\vector(1,0){8}}
\put(46,20){\vector(1,0){8}}

\put(6,54){\vector(1,-1){10}}

\put(40,16){\vector(0,-1){12}}
\put(60,16){\vector(0,-1){12}}

\put(36, 4){\vector(-1,1){12}}
\put(56, 4){\vector(-1,1){12}}
\put(76, 4){\vector(-1,1){12}}
\put(36, 24){\vector(-1,1){12}}
\put(56, 24){\vector(-1,1){12}}
\put(76, 24){\vector(-1,1){12}}
\end{picture}\hspace{2cm}
\begin{picture}(80,63)(0,0)
\put(20,20){\makebox(0,0){${134}$}}
\put(23,25){\makebox(0,0){${+}$}}
\put(20,40){\makebox(0,0){${124}$}}
\put(20,45){\makebox(0,0){${-}$}}
\put(40,20){\makebox(0,0){${145}$}}
\put(43,25){\makebox(0,0){${+}$}}
\put(40,40){\makebox(0,0){${125}$}}
\put(40,45){\makebox(0,0){${+}$}}
\put(60,20){\makebox(0,0){${156}$}}
\put(63,25){\makebox(0,0){${+}$}}
\put(60,40){\makebox(0,0){${126}$}}
\put(60,45){\makebox(0,0){${-}$}}
\put(1,59){\makebox(0,0){$\boxed{123}$}}
\put(10,59){\makebox(0,0){${+}$}}
\put(20,-1){\makebox(0,0){$\boxed{234}$}}
\put(23,5){\makebox(0,0){${-}$}}
\put(40,-1){\makebox(0,0){$\boxed{345}$}}
\put(43,5){\makebox(0,0){${+}$}}
\put(60,-1){\makebox(0,0){$\boxed{456}$}}
\put(63,5){\makebox(0,0){${-}$}}
\put(80,-1){\makebox(0,0){$\boxed{567}$}}
\put(80,5){\makebox(0,0){${c}$}}
\put(81,20){\makebox(0,0){$\boxed{167}$}}
\put(81,26){\makebox(0,0){${b}$}}
\put(81,40){\makebox(0,0){$\boxed{127}$}}
\put(81,46){\makebox(0,0){${a}$}}
\put(20,16){\vector(0,-1){12}}
\put(40,36){\vector(0,-1){12}}
\put(20,36){\vector(0,-1){12}}
\put(60,36){\vector(0,-1){12}}
\put(26,20){\vector(1,0){8}}
\put(66,20){\vector(1,0){8}}
\put(46,40){\vector(1,0){8}}
\put(26,40){\vector(1,0){8}}
\put(66,40){\vector(1,0){8}}
\put(46,20){\vector(1,0){8}}
\put(6,54){\vector(1,-1){10}}
\put(40,16){\vector(0,-1){12}}
\put(60,16){\vector(0,-1){12}}
\put(36, 4){\vector(-1,1){12}}
\put(56, 4){\vector(-1,1){12}}
\put(76, 4){\vector(-1,1){12}}
\put(36, 24){\vector(-1,1){12}}
\put(56, 24){\vector(-1,1){12}}
\put(76, 24){\vector(-1,1){12}}
\end{picture}
\end{center}
\caption{}
\label{fig:G37-signed}
\end{figure}

Before proving the inductive step in general, we will illustrate it for our running
example in \cref{fig:G37-signed}.  Suppose that we know that the theorem is true for 
$\Gr_{3,6}$; we want to prove it for $\Gr_{3,7}$.
Consider an arbitrary signed cell $V_{3,7}^{\sigma}$
of $\Gr_{3,7}$ and the associated grading $\sigma$ on cluster variables.
This grading restricts to a grading on the cluster variables of $\Gr_{3,6}$ and gives a signed
cell $V_{3,6}^{\sigma}$ on $\Gr_{3,6}$.
Hence by the induction hypothesis
we can apply an element
of the torus $(t_1,\dots,t_6)\in \{-1,1\}^6$ such that the resulting signed cell $V_{3,6}^{\sigma'}$
is either the positive part  
or the alternating part of $\Gr_{3,6}$, as indicated by the signs labeling the 
$\Gr_{3,6}$ vertices at the left and right of 
\cref{fig:G37-signed}, respectively. Let $a,b,c\in\{-1,1\}$ be the resulting signs
on $127, 167, 567$.  The resulting sign pattern represents a signed seed
for $\Gr_{3,7}$, so in particular, 
\eqref{eq:grading2} holds.

If $V_{3,6}^{\sigma'}$ is the positive part of 
$\Gr_{3,6}$, as in the left of \cref{fig:G37-signed}, then this ``balancing equation'' around vertex
$126$ implies that $a=b$, and the balancing equation around vertex $156$ implies that $b=c$. 
Thus $a=b=c$.  If $a=b=c=1$,  we set $t_7=1$, and if $a=b=c=-1$, we set $t_7=-1$.
It follows that acting on $V_{3,7}^{\sigma}$ by $(t_1,\dots,t_6,t_7)$ gives rise to the positive part
of $\Gr_{3,7}$.

Similarly, if $V_{3,6}^{\sigma'}$ is the alternating part of 
$\Gr_{3,6}$, as in the right of \cref{fig:G37-signed}, then the balancing equation around vertex
$126$ implies that $a=b$, and the balancing equation around vertex $156$ implies that $b=c$. 
Thus $a=b=c$.  If $a=b=c=1$,  we set $t_7=1$, and if $a=b=c=-1$, we set $t_7=-1$.
It follows that acting on $V_{3,7}^{\sigma}$ by $(t_1,\dots,t_6,t_7)$ gives rise to the alternating part
of $\Gr_{3,7}$.

The same argument works when we go from $\Gr_{k,n-1}$ to $\Gr_{k,n}$, but now we work with 
a more general quiver, as in 
\cref{fig:Gkn-signed}.
\begin{figure}[ht]
\begin{center}
 \setlength{\unitlength}{2pt}
\begin{picture}(80,63)(0,0)
\put(20,20){\makebox(0,0){${\dots}$}}
\put(20,40){\makebox(0,0){${\dots}$}}
\put(40,20){\makebox(0,0){${u_2}$}}
	\put(40,40){\makebox(0,0){${u_1}$}}
\put(60,20){\makebox(0,0){${v_2}$}}
\put(60,40){\makebox(0,0){${v_1}$}}

\put(1,59){\makebox(0,0){$\boxed{12\dots k}$}}
\put(20,-1){\makebox(0,0){${\dots}$}}
\put(40,-1){\makebox(0,0){$\boxed{u_k}$}}
\put(60,-1){\makebox(0,0){$\boxed{v_k}$}}
\put(80,-1){\makebox(0,0){$\boxed{w_k}$}}
\put(80,5){\makebox(0,0){${a_k}$}}
\put(81,20){\makebox(0,0){$\boxed{w_2}$}}
\put(81,26){\makebox(0,0){${a_2}$}}
\put(81,40){\makebox(0,0){$\boxed{w_1}$}}
\put(81,46){\makebox(0,0){${a_1}$}}

\put(20,12){\makebox(0,0){$\vdots$}}
\put(40,36){\vector(0,-1){12}}

\put(20,36){\vector(0,-1){12}}
\put(60,36){\vector(0,-1){12}}

\put(26,20){\vector(1,0){8}}
\put(66,20){\vector(1,0){8}}
\put(46,40){\vector(1,0){8}}

\put(26,40){\vector(1,0){8}}
\put(66,40){\vector(1,0){8}}
\put(46,20){\vector(1,0){8}}

\put(6,54){\vector(1,-1){10}}
\put(40,12){\makebox(0,0){$\vdots$}}
\put(60,12){\makebox(0,0){$\vdots$}}

\put(30,12){\makebox(0,0){$\ddots$}}
\put(50,12){\makebox(0,0){$\ddots$}}
\put(70,12){\makebox(0,0){$\ddots$}}
\put(36, 24){\vector(-1,1){12}}
\put(56, 24){\vector(-1,1){12}}
\put(76, 24){\vector(-1,1){12}}
\end{picture}
\end{center}
\caption{}
\label{fig:Gkn-signed}
\end{figure}
Here $u_1$, $v_1$, and $w_1$ represent the Pl\"ucker coordinates
$P_{1\dots k-1,n-2}$,
$P_{1\dots k-1,n-1}$, and
$P_{1\dots k-1,n}$; 
meanwhile $u_2$, $v_2$, and $w_2$ represent the Pl\"ucker coordinates
$P_{1\dots k-2,n-3,n-2}$,
$P_{1\dots k-2,n-2,n-1}$, and
$P_{1\dots k-2,n-1,n}$; etc.
If we start with an arbitrary signed cell $V^{\sigma}_{k,n}$ for $\Gr_{k,n}$,
as before, it restricts to a signed cell for $\Gr_{k,n-1}$, hence by the induction hypothesis
we can find a torus element $(t_1,\dots,t_{n-1})$ whose application produces a signed cell
$V^{\sigma'}_{k,n}$ whose restriction to $\Gr_{k,n-1}$ is 
either: (1) the positive part or (2) the alternating part of 
 $\Gr_{k,n-1}$.  Let $a_1,\dots,a_k$ be the remaining signs on the frozen variables $w_1,\dots,w_k$.

In Case (1), the balancing equation around $v_1, v_2,\dots,v_{k-1}$ implies $a_1=a_2=\cdots = a_k$. Since $w_1,\dots,w_k$ are the only Pl\"ucker coordinates labeling $Q_{k,n}$ involving the index $n$, we can choose an element $t_n$ that will ensure that these Pl\"ucker coordinates are positive. Thus, $(t_1,\dots,t_n) \cdot V^{\sigma}_{k,n}$ gives the positive part of $\Gr_{k,n}$.

In Case (2), the balancing equation around $v_1$ implies that 
$$a_2 (-1)^{1+\dots + k-1 + (n-2)} = a_1 (-1)^{1+ \dots + k-2 + n-2 + n-1},$$ and hence
$a_2 = a_1 (-1)^{n-k}$. 
Meanwhile, the balancing equation around $v_2$ implies that 
$$a_3 (-1)^{1+ \dots + k-2+ (n-3)+(n-2)} (-1)^{1+ \dots + k-1+ (n-1)} = 
a_2 (-1)^{1+ \dots + k-3 + (n-2)+(n-2) + (n-1)} (-1)^{1+ \dots + k-1 + (n-2)},$$ which 
also implies that $a_3 = a_2 (-1)^{n-k}$.
In general, the balancing equation around $v_i$ implies that $a_{i+1} = a_i {(-1)}^{n-k}$
for $1 \leq i \leq k-1$,  which is exactly what we need to make sure that the signs
of $w_1, \ldots , w_k$ are consistent with the alternating part of $\Gr_{k,n}$.
Our last step is to choose $t_n = \pm 1$, in order to ensure that 
after we apply $(t_1,\ldots,t_n)$ to the original signed seed,
the sign of $w_1 = P_{1\cdots k-1,n}$ is correct, that is, $(-1)^{1+ \cdots + k-1 + n}$.
\end{proof}

\subsection{$G$-gradings for the Grassmannian $\widehat{\Gr}_{k, n}$}\label{sec:G-grading-for-Gr}

In this section, we characterize the $G$-gradings on $\widehat{\Gr}_{k, n}$ for general finite abelian $G$ in terms of the $G$-gradings coming from the natural torus action. We first compute the number of $G$-gradings by analyzing the rectangles seed. 

\begin{proposition}\label{prop:number-of-G-gradings-in-Gr}
	The cluster variety $\widehat{\Gr}_{k, n}$ has $|G|^n$ distinct $G$-gradings. These distinct $G$-gradings are determined by associating an element of $G$ to each vertex represented by the empty set, or a single row or a single column in the quiver shown in  the left diagram in~\cref{fig:G37-Le-quiver2}.
\end{proposition}

\begin{proof}
	Apply Lemma~\ref{lem:frozen-leaf} and Lemma~\ref{lem:isolated-vertex} to the right hand side of the rectangles seed for $\widehat{\Gr}_{k, n}$ to reduce it to the rectangles seed for $\widehat{\Gr}_{k, n-1}$ and the result follows from induction. 
\end{proof}

The torus action on $\widehat{\Gr}_{k, n}$ corresponds to the $\ZZ^n$-grading given by 
\[
	\deg (P_I) = \sum_{i \in I} \mathbf{e}_i. 
\]
The $G$-gradings coming from the torus action correspond to the image of $\Psi : G^{\oplus n} \to G^{\oplus Q_0}$ where 
\[
	\mathbf{g} = (g_1, \ldots, g_n) \mapsto \left ( \sum_{i \in I} g_i : I \in Q_0 \right ). 
\]
Let us describe the torus $G$-gradings. Recall that we can view $Q_0$ both as the set of $k$-element subsets of $[n]$ or the set of rectangles contained in a $k \times (n-k)$ rectangle. The kernel consists of the $\mathbf{g}$ satisfying $\sum_{i \in I} g_i = 0$ for $I \in Q_0$. The rectangles of width $1$ correspond to the $k$-subsets of the form $[k+1] \backslash \{\ell\}$ for $1 \leq \ell \leq k+1$. These rectangles give us the equations
\[
	\left(\sum_{i = 1}^{k+1} g_i \right ) - g_{\ell} = 0 \quad \text{ for all } 1 \leq \ell \leq k+1. 
\]
As a consequence, we get $g_1 = \ldots = g_{k+1}$. The rectangles of height $1$ correspond to the $k$-subsets $[k-1] \cup \{\ell\}$ for $k \leq \ell \leq n$. This gives the equations 
\[
	g_{\ell} + \sum_{i = 1}^{k-1} g_i = 0 \quad \text{ for all } k \leq \ell \leq n. 
\]
As a consequence, we get $g_k = g_{k+1} = \ldots = g_n$. Thus, any element in the kernel must be of the form $(g, g, \ldots, g)$ satisfying $kg = 0$. In other words, there are exactly $|G|^n / |G[k]|$ $G$-gradings coming from the torus action.\footnote{For a group $G$, we define $G[k]$ as the \textbf{$k$-torsion subgroup} of $G$. This consists of the elements $g\in G$ with $kg=0$ (if our group operation is addition), or satisfying $g^k = 1$ (if our group operation is 
multiplication).}
Thus, up to the torus action, there are $|G[k]|$ distinct $G$-gradings (where we use Proposition~\ref{prop:number-of-G-gradings-in-Gr}). 

We will explicitly give representatives for these $G$-gradings when $G$ is a finite abelian group. It suffices to work with $G = \ZZ / p^m \ZZ$ for $p$ prime. When $v_p(k) = 0$, then $G[k]$ is trivial and every $G$-grading comes from the torus action. When $v_p(k) \geq 1$, then $|G[k]| = p^{\min \{v_p(k), m\}}$. For every $0 \leq x \leq p^{\min \{v_p(k), m \}} - 1$, consider the ``constant" grading $\sigma_x(P_I) := x$.\footnote{This grading is constant on Pl\"ucker coordinates, but certainly not on all cluster variables.} These are distinct up to the torus action. Indeed, the difference between any pair of these gradings is the constant grading $\sigma_x$ for $1 \leq x \leq p^{\min \{v_p(k), m\}} - 1$. Assume to the contrary that $\sigma_x$ comes from the torus action. Then, there is a choice $g_1, \ldots, g_n \in G$ such that 
\[
	\sum_{i \in I} g_i = x, \quad \text{ for all } I \in \binom{[n]}{k}. 
\]
From these equations, there is a common value $g \in G$ for which $g_1 = \ldots = g_n = g$, and we have $kg = x$. But this is a contradiction because $v_p(kg) > v_p(x)$. To summarize, we have proven the following result: 
\begin{theorem}\label{thm:G-gradings-grassmannian}
	Let $G$ be a finite abelian group. Then $\widehat{\Gr}_{k, n}$ has $|G[k]|$ distinct $G$-gradings up to the torus action. When $G = \mathbb{Z} / p^m \mathbb{Z}$, distinct representatives of the $G$-gradings up to the torus action are given by the constant gradings $\sigma_x(P_I) := x$ for all $I \in \binom{[n]}{k}$ and $0 \leq x \leq p^{\min\{v_p(k), m\}}-1$.
\end{theorem}

\begin{remark}
	If we specialize to the case $p^m = 2$, then \cref{thm:G-gradings-grassmannian} implies that when $k$ is even, the cluster variety $\widehat{\Gr}_{k, n}$ has two signed seeds up to the torus action. When $k$ is odd, $\widehat{\Gr}_{k, n}$ has only one signed seed up to the torus action. In the even case, the distinct representatives are the \textit{totally positive grading} and the \textit{negative alternating grading} given by $\sigma(P_I) = (-1)^{i_1 + \ldots + i_k + 1}$. Note that when we descend this grading on $\Gr_{k, n}$, it is the same as the alternating grading. 
\end{remark}
\subsection{$G$-gradings for positroid varieties}
Positroid varieties are subvarieties of the Grassmannian which appear in the study of total positivity and Poisson geometry~\cite{Lusztig2,Postnikov,KLS}. Positroid varieties are indexed by combinatorial objects called \textit{Le-diagrams}~\cite{Postnikov}. 

Building on work of~\cite{MR3534832},~\cite{MullerSpeyer}, and~\cite{SSW}, Galashin and Lam show in~\cite{GalashinLam} that the coordinate ring of an open positroid variety has a cluster structure. In their construction, the initial quiver for the cluster structure can be read off directly off of the corresponding Le-diagram. In this section, we enumerate all $G$-gradings for arbitrary positroid varieties.  

\begin{definition}
A \emph{Le-diagram} for $\Gr_{k, n}$ is the data of a partition contained in a $k \times (n-k)$ rectangle with a filling of the boxes with dots, so that whenever a box is both below a dot and to the right of a dot, then it contains a dot. Equivalently, if we place the dots in the center of their corresponding boxes and draw downward and rightward rays from each dot (which we call the \textbf{hook} of that dot), then whenever two hooks intersect they must intersect at a dot. 
See \cref{fig:Le-diagram}.
\end{definition}

	\begin{figure}[h]
\begin{center}
\begin{tikzpicture}
\draw (0, 0) -- (4,0)--(4,-2)--(2,-2)--(2,-3)--(0,-3)--cycle;
\draw (1, 0) -- (1, -3); 
\draw (2, 0) -- (2, -3); 
\draw (3,0) -- (3,-2);
\draw (0, -1) -- (4, -1);
\draw (0, -2) -- (4, -2);
\fill (1.5,-0.5) circle (3pt);
\fill (2.5,-0.5) circle (3pt);
\fill (1.5,-1.5) circle (3pt);
\fill (0.5,-1.5) circle (3pt);
\fill (2.5,-1.5) circle (3pt);

\draw (6, 0) -- (10,0)--(10,-2)--(8,-2)--(8,-3)--(6,-3)--cycle;
\fill (7.5,-0.5) circle (3pt);
\fill (8.5,-0.5) circle (3pt);
\fill (7.5,-1.5) circle (3pt);
\fill (6.5,-1.5) circle (3pt);
\fill (8.5,-1.5) circle (3pt);
\draw (6.5,-1.5)--(6.5,-3);
\draw (7.5, -0.5) -- (7.5,-3);
\draw (8.5,-0.5) -- (8.5,-2);
\draw (6.5,-1.5) -- (10,-1.5);
\draw (7.5,-0.5) -- (10,-0.5);
\end{tikzpicture}
\end{center}
\caption{Le-diagram and the corresponding hook diagram.}
\label{fig:Le-diagram}
\end{figure}
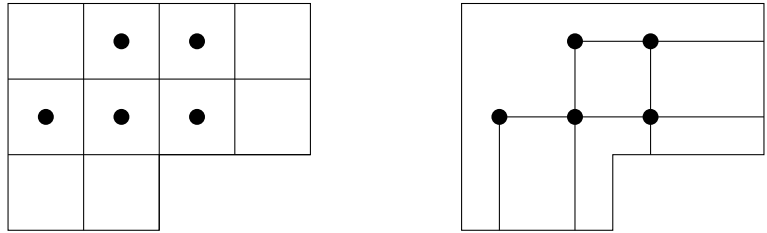

We now describe how to construct the initial quiver associated to a Le-diagram. 
\begin{definition}
	Let $D$ be a Le-diagram. We construct a quiver $Q_D$ with vertices corresponding to the dots in $D$. It is convenient to identify the vertex corresponding to a dot with the region directly south-east of the dot. We say a region is \textbf{frozen} if it touches the boundary of the Young diagram. Otherwise, we call a region \textbf{mutable}. The frozen vertices in $Q_D$ are the vertices which correspond to frozen regions. Given this identification, we have the local rules shown in Figure~\ref{fig:local-rules} for constructing edges in the quiver\footnote{The stars in \cref{fig:local-rules} can either be mutable or frozen vertices. We only construct edges if the endpoints are not both frozen. The top left star might be the top left region which does not correspond to any dot. For the sake of the construction, we assume that this region is a vertex, and then delete this vertex afterwards.}. 

 	See~\cref{fig:quiver-of-Le-diagram} for a Le diagram with $7$ frozen regions, $4$ mutable regions, and also the associated quiver. 
\end{definition}

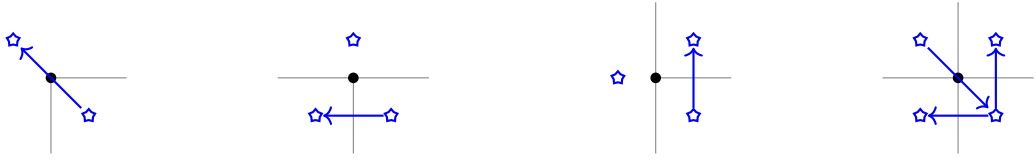
\begin{figure}[h]
		\begin{center}
			\begin{tikzpicture}
				\draw[opacity=0.5] (0,0) -- (1,0);
				\draw[opacity=0.5] (0,0) -- (0,-1);
				\fill (0,0) circle (2pt);
				\node[draw=blue,thick,minimum size=5pt,inner sep=0pt, star] at (0.5,-0.5) {};
				\node[draw=blue,thick,minimum size=5pt,inner sep=0pt, star] at (-0.5,0.5) {};
				\draw[->, blue, thick] (0.4,-0.4) -- (-0.4,0.4);

				\draw[opacity=0.5] (3,0)--(5,0);
				\draw[opacity=0.5](4,0) -- (4,-1);
				\fill (4,0) circle (2pt);
				\node[draw=blue,thick,minimum size=5pt,inner sep=0pt, star] at (3.5,-0.5) {};
				\node[draw=blue,thick,minimum size=5pt,inner sep=0pt, star] at (4,0.5) {};
				\node[draw=blue,thick,minimum size=5pt,inner sep=0pt, star] at (4.5,-0.5) {};

				\draw[->, blue, thick] (4.4,-0.5) -- (3.6,-0.5);

				\draw[opacity=0.5] (8,0) -- (8,1);
				\draw[opacity=0.5] (8,0) -- (9,0);
				\draw[opacity=0.5] (8,0) -- (8,-1);
				\fill (8,0) circle (2pt);
				\node[draw=blue,thick,minimum size=5pt,inner sep=0pt, star] at (8.5,-0.5) {};
				\node[draw=blue,thick,minimum size=5pt,inner sep=0pt, star] at (8.5,0.5) {};
				\node[draw=blue,thick,minimum size=5pt,inner sep=0pt, star] at (7.5,0) {};
				\draw[->, blue, thick] (8.5,-0.4) -- (8.5, 0.4);

				\draw[opacity=0.5] (12,1) -- (12,-1);
				\draw[opacity=0.5] (11,0) -- (13,0);
				\fill (12,0) circle (2pt);
				\node[draw=blue,thick,minimum size=5pt,inner sep=0pt, star] at (12.5,-0.5) {};
				\node[draw=blue,thick,minimum size=5pt,inner sep=0pt, star] at (12.5,0.5) {};
				\node[draw=blue,thick,minimum size=5pt,inner sep=0pt, star] at (11.5,-0.5) {};
				\node[draw=blue,thick,minimum size=5pt,inner sep=0pt, star] at (11.5,0.5) {};
				\draw[->, blue, thick] (12.5,-0.4) -- (12.5,0.4);
				\draw[->, blue, thick] (12.4,-0.5) -- (11.6,-0.5);
				\draw[->, blue, thick] (11.6,0.4) -- (12.4, -0.4);
		\end{tikzpicture}
	\end{center}
	\caption{Local rules for constructing edges in $Q_D$.  Each star represents
	either a frozen or mutable vertex.}
	\label{fig:local-rules}
\end{figure}

\begin{figure}[h]
	\begin{center}
		\begin{tikzpicture}
			\usetikzlibrary{arrows.meta}
			\tikzset{
				square vertex/.style={
					draw,
					rectangle,
					blue,
					thick,
					minimum size=2mm,
					inner sep=0pt
				},
				circle vertex/.style={
					draw,
					circle,
					fill,
					red,
					minimum size=2mm,
					inner sep=0pt
				}
			}
			\draw (0,0) -- (8,0) -- (8, -5) -- (3, -5) -- (3, -8) -- (0, -8) -- cycle;

			\draw (1,-1) -- (8, -1);
			\draw (1,-1) -- (1, -8);
			\fill (1,-1) circle (3pt);

			\draw (2,-1) -- (2,-8);
			\fill (2,-1) circle (3pt);

			\draw (1,-3) -- (8,-3);
			\fill (1,-3) circle (3pt);
			\fill (2,-3) circle (3pt);

			\draw (1,-6) -- (3,-6);
			\fill (1,-6) circle (3pt);
			\fill (2,-6) circle (3pt);

			\draw (5, -2) -- (8,-2);
			\draw (5,-2) -- (5,-5);
			\fill (5,-2) circle (3pt);
			\fill (5,-3) circle (3pt);

			\draw (7, -1) -- (7, -5);
			\fill (7,-1) circle (3pt);
			\fill (7, -2) circle (3pt);
			\fill (7,-3) circle (3pt);

			\node[square vertex] (A) at (7.5,-1.2) {};
			\node[square vertex] (B) at (7.5,-2.5) {};
			\node[square vertex] (C) at (7.5,-4) {};
			\node[square vertex] (D) at (5.3,-4) {};
			\node[square vertex] (E) at (3,-4) {};
			\node[square vertex] (F) at (2.8,-6.5) {};
			\node[square vertex] (G) at (1.5,-7) {};
			\node[circle vertex] (H) at (1.5,-4) {};
			\node[circle vertex] (I) at (1.5, -2) {};
			\node[circle vertex] (J) at (4.5,-1.25) {};
			\node[circle vertex] (K) at (5.5,-2.5) {};

			\draw[-{Latex[length=2mm, width=1.5mm]}, red!70!black] (G) -- (H);
			\draw[-{Latex[length=2mm, width=1.5mm]}, red!70!black] (H) -- (F);
			\draw[-{Latex[length=2mm, width=1.5mm]}, red!70!black] (E) -- (H);
			\draw[-{Latex[length=2mm, width=1.5mm]}, red!70!black] (H) -- (I);
			\draw[-{Latex[length=2mm, width=1.5mm]}, red!70!black] (I) -- (E);
			\draw[-{Latex[length=2mm, width=1.5mm]}, red!70!black] (J) -- (I);
			\draw[-{Latex[length=2mm, width=1.5mm]}, red!70!black] (E) -- (J);
			\draw[-{Latex[length=2mm, width=1.5mm]}, red!70!black] (K) -- (J);
			\draw[-{Latex[length=2mm, width=1.5mm]}, red!70!black] (A) -- (J);
			\draw[-{Latex[length=2mm, width=1.5mm]}, red!70!black] (J) -- (B);
			\draw[-{Latex[length=2mm, width=1.5mm]}, red!70!black] (B) -- (K);
			\draw[-{Latex[length=2mm, width=1.5mm]}, red!70!black] (D) -- (K);
			\draw[-{Latex[length=2mm, width=1.5mm]}, red!70!black] (K) -- (C);
			\draw[-{Latex[length=2mm, width=1.5mm]}, red!70!black] (J) -- (D);
		\end{tikzpicture}
	\end{center}
	\caption{Quiver coming from a Le-diagram}
\label{fig:quiver-of-Le-diagram}
\end{figure}
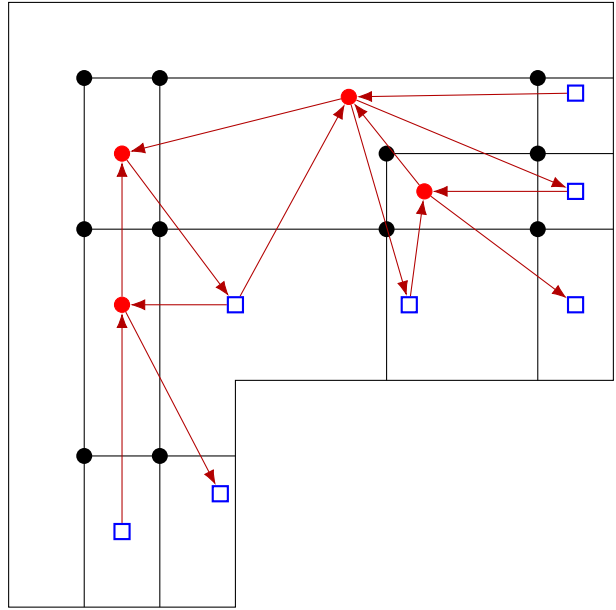

\begin{theorem}
	Let $Q_D$ be the quiver associated to a Le-diagram $D$ of 
	$\mathrm{Gr}_{k,n}$. Then the number of $G$-gradings on $\mathcal{A}(Q_D)$ is $|G|^{\# \text{frozen regions}}$. 
	
\end{theorem}

\begin{proof}
	We describe an algorithm which simplifies a Le-diagram until it is empty and keeps track of the number of $G$-gradings. 

	\textbf{Step 0.} We first remove all rows with no dots. This changes both the bounding permutation and the placement of dots but it doesn't change the quiver. 

	\textbf{Step 1.} Consider the dot $\textbf{p}$ on the lowest row and furthest to the right. There cannot be any dots contained in boxes directly above the boxes directly to the right of $\textbf{p}$. Otherwise, the conditions of a Le-diagram would imply that there would be dots to the direct right of $\textbf{p}$. Thus, we can safely remove the boxes directly to the right of $\textbf{p}$ without changing the underlying quiver nor the number of $G$-gradings.

	\textbf{Step 2.} After \textit{Step 1}, we can assume that $\textbf{p}$ is contained in the last box on the last row. Now consider the part of the quiver local to $\textbf{p}$. There are four cases which are shown in Figure~\ref{fig:local-rules} where $\mathbf{p}$ is the bottom right vertex. In the first three situations, all three vertices are frozen and the vertex corresponding to $\textbf{p}$ is an isolated frozen vertex. Removing $\mathbf{p}$ from the Le-diagram and keeping the bounding partition is the same amounts to removing the corresponding frozen vertex on the quiver side. From Lemma~\ref{lem:isolated-vertex}, the number of $G$-gradings is $|G|$ times the number of $G$-gradings with $\mathbf{p}$ removed. \textit{In these three situations, the number of frozen regions decreases by $1$ and the number of mutable regions stays the same $1$}. In the fourth case, there are two subcases. Either all four vertices are frozen or all except the top left vertex is frozen. In both of these subcases, removing the dot $\mathbf{p}$ \textit{along with} the box from the bounding partition containing $\mathbf{p}$ amounts to removing the vertex corresponding to $\mathbf{p}$ from the quiver and freezing the top left vertex if it is not already frozen. If the top left vertex was originally not frozen, then Lemma~\ref{lem:frozen-leaf} implies that the number of $G$-gradings stays the same. \textit{In this situation, the number of frozen regions stay the same and the number of mutable regions decrease by $1$}. If the top left vertex was originally frozen, then the vertex of $\textbf{p}$ is an isolated frozen. Then, upon removal of $\textbf{p}$ and the box containing $\textbf{p}$, Lemma~\ref{lem:isolated-vertex} implies that the number of $G$-gradings get multiplied by $|G|$. \textit{In this situation, the number of frozen regions decreases by $1$ and the number of mutable regions stays the same.}

	Now, we keep repeating this process until the Le-diagram is empty. Throughout the process, the number of $G$-gradings is multiplied by $|G|$ exactly whenever the number of frozen regions decreases by $1$. Hence, the number of $G$-gradings is $|G|^{\# \text{frozen regions}}$. 
\end{proof}

\section{Signed seeds and the  amplituhedron}\label{sec:amplituhedron}

The \emph{(tree) amplituhedron} 
is the image of the positive Grassmannian under a positive linear map;
it was introduced by 
Arkani-Hamed and Trnka
\cite{arkani-hamed_trnka} in the context of \emph{scattering
amplitudes} in $\mathcal{N}=4$ super Yang-Mills theory.
In this section we will explain how the 
alternating grading on the Grassmannian  
is connected to the amplituhedron $\mathcal{A}_{n,k,m}(Z)$; this was 
previously observed when $m=2$ \cite{PSW2}.
We will also explain how signed seeds come up 
in the context of \emph{tiles} for the amplituhedron.

Let $\Mat_{n,p}^{>0}$ denote the set of $n\times p$ 
matrices whose maximal minors
are positive.

\begin{definition}\label{defn_amplituhedron}
Fix $k, n, m$ positive integers such that $k+m \leq n$.
Let $Z\in \Mat_{n,k+m}^{>0}$. 
       The \emph{amplituhedron map}
$\tilde{Z}:\Gr_{k,n}^{\ge 0} \to \Gr_{k,k+m}$
     is
      defined by
       $\tilde{Z}(C):=CZ$,
       where
 $C$ is a $k \times n$ matrix representing an element of
       $\Gr_{k,n}^{\ge 0}$,  and  $CZ$ is a $k \times (k+m)$ matrix representing an
         element of $\Gr_{k,k+m}$.
       The \emph{amplituhedron} $\mathcal{A}_{n,k,m}(Z) \subset \Gr_{k,k+m}$ is the image
$\tilde{Z}(\Gr_{k,n}^{\ge 0})$
        inside $\Gr_{k,k+m}$.
\end{definition}

\begin{definition}\label{def:twistor}
       Choose
        $Z \in \Mat_{n,k+m}^{>0}$
        with rows  $Z_1,\dots, Z_n \in \R^{k+m}$.
        Given a matrix $Y$
        with rows $y_1,\dots,y_k$
        representing an element of $\Gr_{k,k+m}$,
        and $\{i_1,\dots, i_m\} \subset [n]$,
        we define the \emph{twistor coordinate}
        $\langle Y Z_{i_1} Z_{i_2} \cdots Z_{i_m} \rangle$
        to be  the determinant of the
        matrix with rows
        $y_1, \dots,y_k, Z_{i_1}, \dots, Z_{i_m}$.
        For $I=\{i_1<\dots <i_m\}$, we also
	use  $\lrangle{Z_I(Y)}$ to denote 
        $\langle Y Z_{i_1} Z_{i_2} \dots Z_{i_m} \rangle$.
\end{definition}

The following observation shows that when $k+m=n$, the amplituhedron
can be identified with the alternating part of the Grassmannian
(\emph{cf} \cref{def:alt}). 
\begin{proposition}
Fix $n,k,m$ such that $k+m=n$, 
and choose
$Z \in \Mat_{n,k+m}^{>0}=\Mat_{n,n}^{>0}.$
The amplituhedron 
$\A_{n,k,m}(Z) = \A_{n,n-m,m}(Z)$
is the closure of the set
	$$\{Y \in \Gr_{k,k+m}=\Gr_{k,n} \ \vert \ 
\sgn \lrangle{Y Z_{i_1} \dots Z_{i_m}} = (-1)^{i_1+ \dots + i_m +\delta} 
\text{ for all } 1 \leq i_1 < \dots < i_m \leq n\},$$
	where $\delta = mn + {m \choose 2}$ mod $2$.
\end{proposition}
\begin{proof}
The matrix $Z$ is invertible hence $\tilde{Z}$ is injective.
Moreover the amplituhedron can be written as
$\tilde{Z}(\Gr_{k,n}^{\ge 0}) = \overline{\tilde{Z}(\Gr_{k,n}^{>0})}$.
Choose $C\in \Gr_{k,n}^{>0}$, let $Y=CZ$, and choose $I=\{i_1< \dots < i_m\}$.
	It is not hard to check (see also \cite[(4.14)]{LKW})
	that  $$\lrangle{Y Z_{i_1} \dots Z_{i_m}} = 
	\lrangle{Z_I(Y)}=
	P_{[n]\setminus I}(C) \cdot \det(Z) (-1)^r,$$
	where $r$ is the number of swaps needed to 
put the sequence $1,2,\dots,\hat{i}_1, \dots, 
\hat{i}_m, \dots, n, i_1,\dots,i_m$ in increasing order.
In fact the number of swaps needed is 
$(n-i_1-(m-1))+(n-i_2-(m-2))+\dots+(n-i_m)$
which is congruent to $i_1+ \dots + i_m + mn+{m \choose 2}$
modulo $2$.
	Thus for all subsets $\{i_1 < \dots < i_m\} \subset [n]$ we have
 $$\lrangle{Y Z_{i_1} \dots Z_{i_m}} = 
	P_{[n]\setminus I}(C) \cdot \det(Z) (-1)^{i_1+ \dots + i_m+\delta}$$
	where $\delta = mn + {m \choose 2}$ mod $2$.
\end{proof}

\begin{remark} 

Twistor coordinates can be identified with Pl\"ucker coordinates as follows, see \cite[Lemma 3.10 and Proposition 3.12]{karpwilliams}.
Fix $Z \in \Mat_{n,k+m}^{>0}$
and let
      $W\in \Gr_{k+m,n}^{>0}$
       be the column span of $Z$.
	Note that $\Gr_m(W) \subseteq Gr_{m,n}$.
        Define a map $f_Z: \Gr_m(W) \to \Gr_{k,k+m}$ by
       $$f_Z(X):= Z(X^{\perp}) = \{Z(x): x\in X^{\perp}\}.$$
        If we let $Y:=f_Z(X)$ then
        we have
        \begin{equation}\label{eq:identify}
		p_{I}(X) = \lrangle{Z_{I}(Y)}
                \qquad \text{ for all } \quad
                I \in {[n] \choose m}. 
        \end{equation}
\end{remark}

\begin{notation}
	Given a cluster variable $x$ in $\Gr_{m,n}$ (which is
	therefore a polynomial in Pl\"ucker coordinates) and 
	$Y\in \Gr_{k,k+m}$, we will denote
	by $x(Y)$ the polynomial in twistor coordinates obtained by 
	replacing each Pl\"ucker coordinate $p_I$ on $\Gr_{m,n}$
	by the corresponding twistor coordinate $\lrangle{Z_I(Y)}$
	on $\Gr_{k,k+m}$.  When it is clear from context, we will
	talk about cluster variables and the corresponding
	polynomials in twistor coordinates interchangeably.
\end{notation}

\subsection{Example: amplituhedron tiles}

Recall from \cref{def:TP} that the \emph{totally nonnegative part} $\Gr_{k,n}^{\geq 0}$ of 
$\Gr_{k,n}$ is the subset where all Pl\"ucker coordinates are nonnegative.
If we partition $\Grk$ into strata based on which Pl\"ucker coordinates are strictly
positive and which are $0$, we obtain a cell decomposition of $\Grk$
into \emph{positroid cells} $S$ \cite{Postnikov}. Since the totally nonnegative Grassmannian surjects onto the
amplituhedron $\AAA(Z)$ (which has dimension $km$), it makes sense to 
try to understand the amplituhedron by partitioning it into subsets
which are images of $km$-dimensional cells of $\Grk$.

\begin{definition}[Tiles]\label{defn_tile}
	Fix $k, n, m$ with $k+m \leq n$ and choose
	$Z\in \Mat_{n,k+m}^{>0}$.  
	Given a positroid cell $S$ of 
	$\Gr_{k,n}^{\ge 0}$, we let 
	$\gto{S} := \tilde{Z}(S)$ and 
	$\gt{S}: = \overline{
		\tilde{Z}(S)} = \tilde{Z}(\overline{S})$.

	We call
	$\gt{S}$  and $\gto{S}$
	a \emph{tile}  and an \emph{open tile}
	for $\AAA$ 
	if $\dim(S) =km$ and 
	$\tilde{Z}$ is injective on $S$. 
\end{definition}

The following conjecture is a strengthening of 
\cite[Conjecture 7.17 (ii)]{ELPSTW}.
\begin{conjecture}\label{conj:signed}
	Let $n,k,m$ be positive integers such that 
	$n \geq k+m$, and let $Z\in \Mat^{>0}_{n,k+m}$.
	Let $s=m(n-m)+1$, and 
	let $\gto{S}$ be an open tile for $\AAA(Z)$.
	There is a  seed  
	$(\widetilde{\mathbf{x}}=(x_1,\dots,x_s),Q)$ for 
	$\Gr_{m,n}$ and signs $\sigma_i\in \{1,-1\}$ for each 
	$x_i\in \widetilde{\mathbf{x}}$ such that 
	\begin{equation*}
	\gto{S}=\{Y\in \Gr_{k,k+m}: \sigma_i \cdot x_i(Y) >0 \text{ for all }
		x_i\in \widetilde{\mathbf{x}}\}.
	\end{equation*}
Moreover, there is a subset $\mathbf{f} \subset \widetilde{\mathbf{x}}$ of \emph{facet variables} which are \emph{facets} of the tile
(see \cite[Definition 2.14]{ELPSTW}), and if we freeze these facet variables,
then the seed 
	$(\widetilde{\mathbf{x}},Q)$ together with the 
	signs $(\sigma_1,\dots,\sigma_s)$ is a signed seed.
\end{conjecture}

While open in general, this conjecture is true for $m=2$, see \cref{cor:signed}; as we will
explain, 
this follows from 
results of \cite{PSW2}.

We start by recalling the classification
of tiles for the $m=2$ amplituhedron, using bicolored subdivisions.

\begin{definition}
Consider an $n$-gon with boundary vertices labeled $1,\dots,n$ in clockwise
order.  We say that a \emph{diagonal} is a segment connecting
two boundary vertices.
	A \emph{bicolored subdivision} is a subdivision 
of the $n$-gon into polygons obtained by adding a collection of noncrossing
diagonals to the $n$-gon and coloring the resulting subpolygons 
black and white,
so that adjacent polygons have different colors.  A \emph{triangulated
	bicolored subdivision} $\mathcal{T}$ is a bicolored subdivision together with a 
triangulation of each black polygon (where the triangulation is 
obtained by adding some diagonals).
See \cref{fig:subdivision}. We say that the bicolored subdivision
	has \emph{type $(k,n)$} if a triangulation of it consists of 
	$k$ black triangles.
\end{definition}

\begin{figure}[htbp]
    \centering
    \resizebox{0.3\linewidth}{!}{
        \begin{tikzpicture}
        \node[draw, minimum size=3.3cm, regular polygon, regular polygon sides=9] (a) {};
        \filldraw[fill=gray!30!white](a.corner 1)--(a.corner 2)--(a.corner 3)--(a.corner 4)--(a.corner 1);
        \filldraw[fill=gray!30!white](a.corner 4)--(a.corner 6)--(a.corner 7)--(a.corner 8)--(a.corner 9)--(a.corner 4);
        \foreach \x in {1,2,...,9}
        \fill (a.corner \x) circle[radius=2pt];
        \draw[black,thin](a.corner 1)--(a.corner 4);
        \draw[black,thin](a.corner 9)--(a.corner 4);
        \draw[black,thin](a.corner 6)--(a.corner 4);
        \node[shift=(a.corner 1), anchor=south] {$1$};
        \node[shift=(a.corner 2),anchor=south] {$9$};
        \node[shift=(a.corner 3),anchor=east] {$8$};
        \node[shift=(a.corner 4),anchor=east] {$7$};
        \node[shift=(a.corner 5),anchor=north] {$6$};
        \node[shift=(a.corner 6),anchor=north] {$5$};
        \node[shift=(a.corner 7),anchor=west] {$4$};
        \node[shift=(a.corner 8),anchor=west] {$3$};
        \node[shift=(a.corner 9),anchor=west] {$2$};
        \end{tikzpicture}
    }
    \resizebox{0.3\linewidth}{!}{
        \begin{tikzpicture}
        \node[draw, minimum size=3.3cm, regular polygon, regular polygon sides=9] (a) {};
        \filldraw[fill=gray!30!white](a.corner 1)--(a.corner 2)--(a.corner 3)--(a.corner 4)--(a.corner 1);
        \filldraw[fill=gray!30!white](a.corner 4)--(a.corner 6)--(a.corner 7)--(a.corner 8)--(a.corner 9)--(a.corner 4);
        \foreach \x in {1,2,...,9}
        \fill (a.corner \x) circle[radius=2pt];
        \draw[black,thin](a.corner 1)--(a.corner 4);
        \draw[black,thin](a.corner 9)--(a.corner 4);
        \draw[black,thin](a.corner 6)--(a.corner 4);
        \draw[black,thin](a.corner 2)--(a.corner 4);
        \draw[black,thin](a.corner 7)--(a.corner 4);
        \draw[black,thin](a.corner 8)--(a.corner 4);
        \node[shift=(a.corner 1), anchor=south] {$1$};
        \node[shift=(a.corner 2),anchor=south] {$9$};
        \node[shift=(a.corner 3),anchor=east] {$8$};
        \node[shift=(a.corner 4),anchor=east] {$7$};
        \node[shift=(a.corner 5),anchor=north] {$6$};
        \node[shift=(a.corner 6),anchor=north] {$5$};
        \node[shift=(a.corner 7),anchor=west] {$4$};
        \node[shift=(a.corner 8),anchor=west] {$3$};
        \node[shift=(a.corner 9),anchor=west] {$2$};
        \end{tikzpicture}
    }
    \par\vspace{5mm}
    \resizebox{0.3\linewidth}{!}{
        \begin{tikzpicture}
        \node[draw, minimum size=3.3cm, regular polygon, regular polygon sides=9] (a) {};
        \filldraw[fill=gray!30!white](a.corner 1)--(a.corner 2)--(a.corner 3)--(a.corner 4)--(a.corner 1);
        \filldraw[fill=gray!30!white](a.corner 4)--(a.corner 6)--(a.corner 7)--(a.corner 8)--(a.corner 9)--(a.corner 4);
        \foreach \x in {1,2,...,9}
        \fill (a.corner \x) circle[radius=2pt];
        \draw[black,thin](a.corner 1)--(a.corner 4);
        \draw[black,thin](a.corner 9)--(a.corner 4);
        \draw[black,thin](a.corner 6)--(a.corner 4);
        \draw[black,thin](a.corner 2)--(a.corner 4);
        \draw[black,thin](a.corner 7)--(a.corner 4);
        \draw[black,thin](a.corner 8)--(a.corner 4);
        \node[shift=(a.corner 1), anchor=south] {$1$};
        \node[shift=(a.corner 2),anchor=south] {$9$};
        \node[shift=(a.corner 3),anchor=east] {$8$};
        \node[shift=(a.corner 4),anchor=east] {$7$};
        \node[shift=(a.corner 5),anchor=north] {$6$};
        \node[shift=(a.corner 6),anchor=north] {$5$};
        \node[shift=(a.corner 7),anchor=west] {$4$};
        \node[shift=(a.corner 8),anchor=west] {$3$};
        \node[shift=(a.corner 9),anchor=west] {$2$};
        \node[rectangle, draw=red, fill=none, thick, minimum size=1mm, inner sep=0pt] at ($(a.corner 1)!0.5!(a.corner 2)$) {};
        \node[rectangle, draw=red, fill=none, thick, minimum size=1mm, inner sep=0pt] at ($(a.corner 2)!0.5!(a.corner 3)$) {};
        \node[rectangle, draw=red, fill=none, thick, minimum size=1mm, inner sep=0pt] at ($(a.corner 3)!0.5!(a.corner 4)$) {};
        \node[rectangle, draw=red, fill=none, thick, minimum size=1mm, inner sep=0pt] at ($(a.corner 1)!0.5!(a.corner 4)$) {};
        \node[rectangle, draw=red, fill=none, thick, minimum size=1mm, inner sep=0pt] at ($(a.corner 9)!0.5!(a.corner 4)$) {};
        \node[rectangle, draw=red, fill=none, thick, minimum size=1mm, inner sep=0pt] at ($(a.corner 9)!0.5!(a.corner 8)$) {};
        \node[rectangle, draw=red, fill=none, thick, minimum size=1mm, inner sep=0pt] at ($(a.corner 7)!0.5!(a.corner 8)$) {};
        \node[rectangle, draw=red, fill=none, thick, minimum size=1mm, inner sep=0pt] at ($(a.corner 6)!0.5!(a.corner 7)$) {};
        \node[rectangle, draw=red, fill=none, thick, minimum size=1mm, inner sep=0pt] at ($(a.corner 6)!0.5!(a.corner 4)$) {};
        \node[circle, draw=red, fill=red,minimum size=1mm, inner sep=0pt] at ($(a.corner 2)!0.5!(a.corner 4)$) {};
        \node[circle, draw=red, fill=red,minimum size=1mm, inner sep=0pt] at ($(a.corner 8)!0.5!(a.corner 4)$) {};
        \node[circle, draw=red, fill=red,minimum size=1mm, inner sep=0pt] at ($(a.corner 7)!0.5!(a.corner 4)$) {};
        \draw[red,thick,shorten <=0mm, shorten >=0.5mm,decoration={markings,mark=at position 0.5 with {\arrow{>}}},postaction={decorate}] ($(a.corner 2)!0.5!(a.corner 4)$)--($(a.corner 1)!0.5!(a.corner 2)$);
        \draw[red,thick,shorten <=0.6mm, shorten >=0.5mm,decoration={markings,mark=at position 0.5 with {\arrow{>}}},postaction={decorate}] ($(a.corner 1)!0.5!(a.corner 4)$)--($(a.corner 2)!0.5!(a.corner 4)$);
        \draw[red,thick,shorten <=0.6mm, shorten >=0.5mm,decoration={markings,mark=at position 0.5 with {\arrow{>}}},postaction={decorate}] ($(a.corner 2)!0.5!(a.corner 3)$)--($(a.corner 2)!0.5!(a.corner 4)$);
        \draw[red,thick,shorten <=0.6mm, shorten >=0.5mm,decoration={markings,mark=at position 0.5 with {\arrow{>}}},postaction={decorate}] ($(a.corner 2)!0.5!(a.corner 4)$)--($(a.corner 3)!0.5!(a.corner 4)$);
        \draw[red,thick,shorten <=0mm, shorten >=0.5mm,decoration={markings,mark=at position 0.5 with {\arrow{>}}},postaction={decorate}] ($(a.corner 4)!0.5!(a.corner 8)$)--($(a.corner 4)!0.5!(a.corner 9)$);
        \draw[red,thick,shorten <=0mm, shorten >=0.5mm,decoration={markings,mark=at position 0.5 with {\arrow{>}}},postaction={decorate}] ($(a.corner 4)!0.5!(a.corner 8)$)--($(a.corner 7)!0.5!(a.corner 8)$);
        \draw[red,thick,shorten <=0.6mm, shorten >=0.5mm,decoration={markings,mark=at position 0.5 with {\arrow{>}}},postaction={decorate}] ($(a.corner 8)!0.5!(a.corner 9)$)--($(a.corner 4)!0.5!(a.corner 8)$);
        \draw[red,thick,shorten <=0.6mm, shorten >=0.5mm,decoration={markings,mark=at position 0.5 with {\arrow{>}}},postaction={decorate}] ($(a.corner 4)!0.5!(a.corner 7)$)--($(a.corner 4)!0.5!(a.corner 8)$);
        \draw[red,thick,shorten <=0.6mm, shorten >=0.5mm,decoration={markings,mark=at position 0.5 with {\arrow{>}}},postaction={decorate}] ($(a.corner 4)!0.5!(a.corner 7)$)--($(a.corner 6)!0.5!(a.corner 7)$);
        \draw[red,thick,shorten <=0.6mm, shorten >=0.5mm,decoration={markings,mark=at position 0.5 with {\arrow{>}}},postaction={decorate}] ($(a.corner 4)!0.5!(a.corner 6)$)--($(a.corner 4)!0.5!(a.corner 7)$);
        \draw[red,thick,shorten <=0.6mm, shorten >=0.5mm,decoration={markings,mark=at position 0.5 with {\arrow{>}}},postaction={decorate}] ($(a.corner 7)!0.5!(a.corner 8)$)--($(a.corner 4)!0.5!(a.corner 7)$);
        \end{tikzpicture}
    }
    \resizebox{0.3\linewidth}{!}{
        \begin{tikzpicture}
        \node[draw, minimum size=3.3cm, regular polygon, regular polygon sides=9] (a) {};
        \filldraw[fill=gray!30!white](a.corner 1)--(a.corner 2)--(a.corner 3)--(a.corner 4)--(a.corner 1);
        \filldraw[fill=gray!30!white](a.corner 4)--(a.corner 6)--(a.corner 7)--(a.corner 8)--(a.corner 9)--(a.corner 4);
        \foreach \x in {1,2,...,9}
        \fill (a.corner \x) circle[radius=2pt];
        \draw[black,thin](a.corner 1)--(a.corner 4);
        \draw[black,thin](a.corner 9)--(a.corner 4);
        \draw[black,thin](a.corner 6)--(a.corner 4);
        \draw[black,thin](a.corner 2)--(a.corner 4);
        \draw[black,thin](a.corner 7)--(a.corner 4);
        \draw[black,thin](a.corner 8)--(a.corner 4);
        \node[shift=(a.corner 1), anchor=south] {$1$};
        \node[shift=(a.corner 2),anchor=south] {$9$};
        \node[shift=(a.corner 3),anchor=east] {$8$};
        \node[shift=(a.corner 4),anchor=east] {$7$};
        \node[shift=(a.corner 5),anchor=north] {$6$};
        \node[shift=(a.corner 6),anchor=north] {$5$};
        \node[shift=(a.corner 7),anchor=west] {$4$};
        \node[shift=(a.corner 8),anchor=west] {$3$};
        \node[shift=(a.corner 9),anchor=west] {$2$};
        \node[rectangle, draw=red, fill=none, thick, minimum size=1mm, inner sep=0pt] at ($(a.corner 1)!0.5!(a.corner 2)$) {};
        \node[rectangle, draw=red, fill=none, thick, minimum size=1mm, inner sep=0pt] at ($(a.corner 2)!0.5!(a.corner 3)$) {};
        \node[rectangle, draw=red, fill=none, thick, minimum size=1mm, inner sep=0pt] at ($(a.corner 3)!0.5!(a.corner 4)$) {};
        \node[rectangle, draw=red, fill=none, thick, minimum size=1mm, inner sep=0pt] at ($(a.corner 1)!0.5!(a.corner 4)$) {};
        \node[rectangle, draw=red, fill=none, thick, minimum size=1mm, inner sep=0pt] at ($(a.corner 9)!0.5!(a.corner 4)$) {};
        \node[rectangle, draw=red, fill=none, thick, minimum size=1mm, inner sep=0pt] at ($(a.corner 9)!0.5!(a.corner 8)$) {};
        \node[rectangle, draw=red, fill=none, thick, minimum size=1mm, inner sep=0pt] at ($(a.corner 7)!0.5!(a.corner 8)$) {};
        \node[rectangle, draw=red, fill=none, thick, minimum size=1mm, inner sep=0pt] at ($(a.corner 6)!0.5!(a.corner 7)$) {};
        \node[rectangle, draw=red, fill=none, thick, minimum size=1mm, inner sep=0pt] at ($(a.corner 6)!0.5!(a.corner 4)$) {};
        \node[circle, draw=red, fill=red,minimum size=1mm, inner sep=0pt] at ($(a.corner 2)!0.5!(a.corner 4)$) {};
        \node[circle, draw=red, fill=red,minimum size=1mm, inner sep=0pt] at ($(a.corner 8)!0.5!(a.corner 4)$) {};
        \node[circle, draw=red, fill=red,minimum size=1mm, inner sep=0pt] at ($(a.corner 7)!0.5!(a.corner 4)$) {};
        \draw[red,thick,shorten <=0mm, shorten >=0.5mm,decoration={markings,mark=at position 0.5 with {\arrow{>}}},postaction={decorate}] ($(a.corner 2)!0.5!(a.corner 4)$)--($(a.corner 1)!0.5!(a.corner 2)$);
        \draw[red,thick,shorten <=0.6mm, shorten >=0.5mm,decoration={markings,mark=at position 0.5 with {\arrow{>}}},postaction={decorate}] ($(a.corner 1)!0.5!(a.corner 4)$)--($(a.corner 2)!0.5!(a.corner 4)$);
        \draw[red,thick,shorten <=0.6mm, shorten >=0.5mm,decoration={markings,mark=at position 0.5 with {\arrow{>}}},postaction={decorate}] ($(a.corner 2)!0.5!(a.corner 3)$)--($(a.corner 2)!0.5!(a.corner 4)$);
        \draw[red,thick,shorten <=0.6mm, shorten >=0.5mm,decoration={markings,mark=at position 0.5 with {\arrow{>}}},postaction={decorate}] ($(a.corner 2)!0.5!(a.corner 4)$)--($(a.corner 3)!0.5!(a.corner 4)$);
        \draw[red,thick,shorten <=0mm, shorten >=0.5mm,decoration={markings,mark=at position 0.5 with {\arrow{>}}},postaction={decorate}] ($(a.corner 4)!0.5!(a.corner 8)$)--($(a.corner 4)!0.5!(a.corner 9)$);
        \draw[red,thick,shorten <=0mm, shorten >=0.5mm,decoration={markings,mark=at position 0.5 with {\arrow{>}}},postaction={decorate}] ($(a.corner 4)!0.5!(a.corner 8)$)--($(a.corner 7)!0.5!(a.corner 8)$);
        \draw[red,thick,shorten <=0.6mm, shorten >=0.5mm,decoration={markings,mark=at position 0.5 with {\arrow{>}}},postaction={decorate}] ($(a.corner 8)!0.5!(a.corner 9)$)--($(a.corner 4)!0.5!(a.corner 8)$);
        \draw[red,thick,shorten <=0.6mm, shorten >=0.5mm,decoration={markings,mark=at position 0.5 with {\arrow{>}}},postaction={decorate}] ($(a.corner 4)!0.5!(a.corner 7)$)--($(a.corner 4)!0.5!(a.corner 8)$);
        \draw[red,thick,shorten <=0.6mm, shorten >=0.5mm,decoration={markings,mark=at position 0.5 with {\arrow{>}}},postaction={decorate}] ($(a.corner 4)!0.5!(a.corner 7)$)--($(a.corner 6)!0.5!(a.corner 7)$);
        \draw[red,thick,shorten <=0.6mm, shorten >=0.5mm,decoration={markings,mark=at position 0.5 with {\arrow{>}}},postaction={decorate}] ($(a.corner 4)!0.5!(a.corner 6)$)--($(a.corner 4)!0.5!(a.corner 7)$);
        \draw[red,thick,shorten <=0.6mm, shorten >=0.5mm,decoration={markings,mark=at position 0.5 with {\arrow{>}}},postaction={decorate}] ($(a.corner 7)!0.5!(a.corner 8)$)--($(a.corner 4)!0.5!(a.corner 7)$);
        \node[blue, yshift=2mm, font = \tiny\boldmath, scale=0.7] at ($(a.corner 1)!0.5!(a.corner 2)$) {$-$};
        \node[blue, xshift=2mm, yshift = 1mm, font = \tiny\boldmath, scale=0.7] at ($(a.corner 1)!0.5!(a.corner 4)$) {$-$};
        \node[blue, xshift=0mm, yshift = 2mm, font = \tiny\boldmath, scale=0.7] at ($(a.corner 9)!0.5!(a.corner 4)$) {$-$};
        \node[blue, xshift=0mm, yshift = -1.5mm, font = \tiny\boldmath, scale=0.7] at ($(a.corner 4)!0.5!(a.corner 7)$) {$-$};
        \node[blue, xshift=-1mm, yshift = 1mm, font = \tiny\boldmath, scale=0.7] at ($(a.corner 2)!0.5!(a.corner 3)$) {$+$};
        \node[blue, xshift=-2mm, yshift = 0mm, font = \tiny\boldmath, scale=0.7] at ($(a.corner 3)!0.5!(a.corner 4)$) {$+$};
        \node[blue, xshift=1mm, yshift = -1mm, font = \tiny\boldmath, scale=0.7] at ($(a.corner 2)!0.5!(a.corner 4)$) {$+$};
        \node[blue, xshift=1mm, yshift = 1mm, font = \tiny\boldmath, scale=0.7] at ($(a.corner 8)!0.5!(a.corner 9)$) {$+$};
        \node[blue, xshift=2mm, yshift = 0mm, font = \tiny\boldmath, scale=0.7] at ($(a.corner 7)!0.5!(a.corner 8)$) {$+$};
        \node[blue, xshift=2mm, yshift = -1mm, font = \tiny\boldmath, scale=0.7] at ($(a.corner 7)!0.5!(a.corner 6)$) {$+$};
        \node[blue, xshift=0mm, yshift = -1.5mm, font = \tiny\boldmath, scale=0.7] at ($(a.corner 4)!0.5!(a.corner 6)$) {$+$};
        \node[blue, xshift=-2mm, yshift = 0mm, font = \tiny\boldmath, scale=0.7] at ($(a.corner 4)!0.5!(a.corner 8)$) {$+$};
        \end{tikzpicture}
    }
    \caption{A bicolored subdivision of type $(5,9)$, a triangulation
    of it, the corresponding quiver, and the signed seed.}\label{fig:subdivision}
\end{figure}
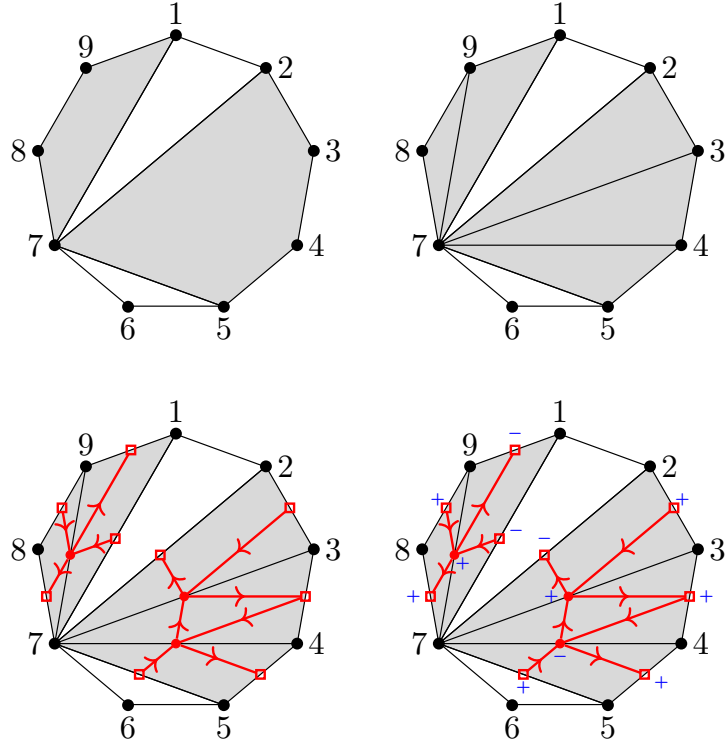

Given a bicolored subdivision, choose $1\leq i<j \leq n$ such 
that the diagonal $i-j$ does not cross any of the diagonals
defining the bicolored subdivision.  If we then choose a triangulation which 
is compatible with $i-j$, 
we define $\area(i\to j)$ to be the number of black triangles to the
left of the arc $i\to j$. For instance, in \cref{fig:subdivision},
$\area(3\to 7) = 2$.

\begin{theorem}\cite[Theorem 4.25 and Theorem 4.28]{PSW2}
The tiles for the amplituhedron $\mathcal{A}_{n,k,2}(Z)$
are in bijection with bicolored subdivisions
of type $(k,n)$.  Moreover, if $\mathcal{T}$
is a triangulated bicolored subdivision, then we have
that the corresponding tile is given by
\begin{equation*}
	\gto{S_{\mathcal{T}}}=\{Y\in \Gr_{k,k+2}: 
	(-1)^{\area(i\to j)} \lrangle{Z_{ij}(Y)} > 0 \text{ for all 
	black arcs }i \to j \text{ of }\mathcal{T} \text{ with }
	i<j\}.
\end{equation*}
\end{theorem}

\begin{corollary}\label{cor:signed}
Given a triangulated bicolored subdivision of type 
$(k,n)$, we associate a quiver by placing a frozen node
on each boundary segment of a black polygon, and a mutable
node on each diagonal of a black polygon.  In each black triangle
we construct three arrows connecting the three nodes, oriented
clockwise
       as in \cref{fig:subdivision}
	(however as usual we can ignore arrows between
frozen nodes). The assignment ${(-1)}^{\area(i\to j)}$ of sign
to the node on the diagonal $i-j$ (where $i<j$) 
then gives this quiver
the structure of a signed seed.
In particular, \cref{conj:signed} holds for all tiles of the amplituhedron
$\AAA(Z)$ when $m=2$.
\end{corollary}

\begin{proof}
We need to show that at each mutable vertex, our assignment of signs
is balanced as in~\eqref{eq:grading2}. Suppose we have a mutable vertex at the diagonal
$i-k$, which lies inside a triangulated black quadrilateral on vertices $i<j<k<l$, see \cref{fig:verify}.

\begin{figure}
	\centering
    \resizebox{0.3\linewidth}{!}{
\begin{tikzpicture}
    \node[draw, circle, minimum size=2cm,
          inner sep=0pt, outer sep=0pt] (a) at (0,0) {};

    \foreach \direction in {north,east,south,west} {
        \fill (a.\direction) circle[radius=1pt];
    }
	\node[font = \tiny, xshift=0mm, yshift = 2mm] at (a.north) {$k$};
	\node[font = \tiny, xshift=2mm, yshift = 0mm] at (a.east) {$\ell$};
	\node[font = \tiny, xshift=-2mm, yshift = 0mm] at (a.west) {$j$};
	\node[font = \tiny, xshift=0mm, yshift = -2mm] at (a.south) {$i$};
	\draw (a.west) -- (a.north) -- (a.east) -- (a.south) -- cycle;
	\draw (a.north) -- (a.south);
	\node[font = \tiny] at ($1/3*(a.north) + 1/3*(a.west) + 1/3*(a.south)$) {$1$};
	\node[font = \tiny] at ($1/3*(a.north) + 1/3*(a.east) + 1/3*(a.south)$) {$1$};
	\node[font = \tiny, xshift = -1mm, yshift =1mm] at ($1/2*(a.north) + 1/2*(a.west)$) {$B$};
	\node[font = \tiny, xshift = 1mm, yshift =1mm] at ($1/2*(a.north) + 1/2*(a.east)$) {$C$};
	\node[font = \tiny, xshift = -1mm, yshift =-1mm] at ($1/2*(a.west) + 1/2*(a.south)$) {$A$};
	\node[font = \tiny, xshift = 1mm, yshift =-1mm] at ($1/2*(a.east) + 1/2*(a.south)$) {$D$};
\end{tikzpicture}
    }
	\caption{Verifying the signed seed property}\label{fig:verify}
\end{figure}
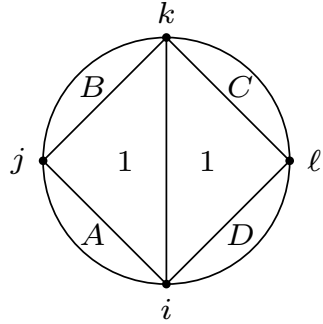

Let $A=\area(i\to j)$,$B=\area(j\to k)$, $C=\area(k\to l)$. To verify \eqref{eq:grading2}, we need to show that $\sigma(i-j) \sigma(k-l) = \sigma(j-k) \sigma(i-l)$, where $\sigma(i-j)$ denotes the sign (either $\pm 1$) associated to the diagonal $i-j$. Thus, we need to show that $\area(i\to j) + \area(k\to l)$ is congruent to $\area(j\to k) + \area(i\to l)$  modulo $2$, or equivalently $A+C$ is congruent to $B + (A+B+C+2)$.   This holds so we are done.
\end{proof}

\cref{conj:signed} is also true 
for $m=4$ BCFW tiles, 
see \cite[Proposition 8.10, Definition 9.8, Proposition 9.27]{ELPSTW}. 
As the combinatorics of these tiles is significantly more complicated
than the $m=2$ case, we will not recall it here, and simply 
refer the reader to~\cite{ELPSTW}.

\section{\texorpdfstring{$\R$}{R}-balanced 
and \texorpdfstring{$\Z$}{Z}-balanced 
exchange matrices and tropical mutation}\label{sec:RZ}

In this section we briefly explain some applications of 
$\R$-balanced and $\Z$-balanced exchange matrices, which 
came up in the study of Newton-Okounkov bodies \cite{RW, RW2}.
The first application 
(\cref{lem:tropmutation})
says that whenever we have an $\R$-balanced
exchange matrix
$(\tilde{B},(\sigma_1,\dots,\sigma_s))$,
the \emph{tropicalized mutation map} is linear at $\sigma$.
The second application 
(\cref{l:monomialXmutation})
says that  the \emph{$Y$-seed mutation}
	$\Mut_{Y,k}(y^{\sigma})$ 
	of a monomial $y^{\sigma}:=y_1^{\sigma_1} \dots y_s^{\sigma_s}$ 
	is a monomial in the new mutated variables $y'_1 \dots y'_s$ 
	if and only if $(\tilde B, (\sigma_1,\dots, \sigma_s))$
	is $\Z$-balanced at $k$.  
And in this case the new exponent vector is obtained from the original
by tropicalized mutation.

\begin{definition}\cite[Definition 7.5]{RW}\label{def:tropmutation}
Let $\Sigma=(\tilde{B}, (x_1,\dots,x_s))$ be a seed and let 
$\Sigma':=\mu_k(\Sigma)$ be the seed obtained by mutation.
	We define the \emph{tropicalized  mutation  map} to be the map
$\Psi_{\Sigma, \Sigma'}: \R^s \to \R^s$ defined by 
\[(\sigma_1,\dots, \sigma_s) \mapsto
	(\sigma_1,\dots, \sigma_{k-1}, \sigma'_k, \sigma_{k+1},\dots, \sigma_s),\] where
\[
	\sigma'_k := \min \left( \sum_{b_{ik}>0} b_{ik} \sigma_i,\sum_{b_{ik}<0} -b_{ik} \sigma_i \right) - \sigma_k.
\]
\end{definition}

Note that the tropicalized mutation map is a piecewise linear map.
However, if $(\tilde{B},(\sigma_1,\dots,\sigma_s))$ is $\R$-balanced, that is, if 
\[
	\sum_{b_{ik}>0} b_{ik} \sigma_i = \sum_{b_{ik}<0} -b_{ik} \sigma_i
\] 
for every mutable direction $k$, we get a strong linearity property.

\begin{lemma}\label{lem:tropmutation}\cite[Lemma 7.12]{RW2}
	We use the notation of \cref{def:tropmutation}. Suppose that $\sigma=(\sigma_1,\dots,\sigma_s)$ and $\tau=(\tau_1,\dots,\tau_s)$ lie in $\R^s$, and $(\tilde{B},\sigma)$ is $\R$-balanced. Then the tropicalized mutation map is linear at $\sigma$: that is, we have 
	\[
		\Psi_{\Sigma,\Sigma'}(\sigma+\tau)=\Psi_{\Sigma,\Sigma'}(\sigma)+\Psi_{\Sigma,\Sigma'}(\tau).
	\]
\end{lemma}

\cref{lem:tropmutation} and the notion of $\R$-balanced exchange matrices
 were used as a tool in~\cite{RW2} to prove that the Newton-Okounkov body of a Schubert variety (with respect to a chosen seed) equals the corresponding superpotential polytope; when one changes the seed, both polytopes are transformed via the tropicalized mutation map.

In the earlier parts of this paper we have been working with the usual version of mutation (e.g.~\cref{eq:exch-rel-geom}) which is sometimes called $\mathcal{A}$-cluster mutation~\cite{FG}. There is another version of the mutation rule which is called $Y$-seed mutation~\cite[Definition 2.4]{ca4}, or $\mathcal{X}$-cluster mutation~\cite{FG}.  We'll explain below how $\Z$-balanced exchange matrices have applications to $Y$-seed mutation, and are also connected to tropicalized mutation.

\begin{definition}[$Y$-seed mutation]\label{Xseed}
	Let $\tilde{B}$ be an $s \times r$ exchange matrix and let $(y_1,\dots,y_s)$ be a set of parameters. For $1 \leq k \leq r$, we define a new set of parameters $(y'_1,\dots,y'_s)$ where
	\begin{equation}\label{e:XclusterMut}
		y'_{i} = \begin{cases}
			\frac{1}{y_{k}} & \text{if }i = k, \\
			y_{i}(1+y_{k})^{b_{i k}} & \text{if }b_{i k}\geq 0 \text{ and }i\neq k,\\
			{y_{i}}(1+y_{k}^{-1})^{b_{i k}} & \text{if }b_{i k}\leq 0 \text{ and }i \neq k
		\end{cases}
	\end{equation}
	We say that $(\mu_{k}(\tilde B), (y'_1,\dots,y'_s))$ is obtained from $(\tilde B, (y_1,\dots,y_s))$ by \emph{$Y$-seed mutation} in direction $k$, and we refer to the ordered pairs $(\mu_{k}(\tilde B), (y'_1,\dots,y'_s))$  and $(\tilde B, (y_1,\dots,y_s))$ as \emph{labeled $Y$-seeds}. If $f$ is a rational expression in the parameters $(y_1,\dots,y_s)$, we use $\Mut_{Y,k}(f)$ to denote the new expression for $f$ obtained by rewriting it in terms of   $(y'_1,\dots,y'_s)$.
\end{definition}

\begin{lemma}\label{l:monomialXmutation}~\cite[Lemma 15.9]{RW}
	Let $(\tilde B, (y_1,\dots,y_s))$ be a {labeled ${Y}$-seed} as above, and let $\sigma=(\sigma_1,\dots,\sigma_s)\in\Z^s$. Let $y^{\sigma}:=y_1^{\sigma_1} \dots y_s^{\sigma_s}$ be the corresponding Laurent monomial in $y_1 \dots y_s$. Choose $1 \leq k \leq r$. Then the $Y$-seed mutation $\Mut_{Y,k}(y^{\sigma})$ of $y^{\sigma}$ is a Laurent monomial in $y'_1 \dots y'_s$ if and only if $(\tilde B, (\sigma_1,\dots, \sigma_s))$ is $\Z$-balanced at $k$. Moreover, in that case, $\Mut_{Y,k}(y^{\sigma})$ is a Laurent monomial with new exponent vector $(\sigma'_1,\dots, \sigma'_s)\in \Z^s$ given by 

	\begin{equation}\label{e:balancedMutation}
		\sigma'_j=
			\begin{cases}
				(\sum_{b_{ik}>0} b_{ik} \sigma_i) - \sigma_k, & j=k,\\
				\sigma_j, & j \ne k.
			\end{cases}
	\end{equation}
	Note that since $(\tilde B, \sigma)$ is $\Z$-balanced at vertex $k$, this is an instance of tropicalized  mutation. 
\end{lemma}

\begin{corollary}
	Let $(\tilde B, (y_1,\dots,y_s))$ be a {labeled ${Y}$-seed}, and let $\sigma=(\sigma_1,\dots,\sigma_s)\in\Z^s$. If $(\tilde B, (\sigma_1,\dots, \sigma_s))$ is $\Z$-balanced (at all vertices), then $y^{\sigma}$ is regular on the corresponding cluster variety.\footnote{A regular function that looks like an ordinary Laurent monomial in at least one cluster chart is also called a \emph{global monomial}.}
\end{corollary}

\bibliographystyle{alpha}
\bibliography{bibliography}
\end{document}